\documentstyle[twoside,amsmath, amssymb, mathrsfs, amsfonts, eucal, epsfig, 11pt]{article}

 \newtheorem{theorem}{Theorem}[section]
 \newtheorem{lemma}[theorem]{Lemma}
 \newtheorem{proposition}[theorem]{Proposition}
 \newtheorem{corollary}[theorem]{Corollary}
 \newtheorem{definition}[theorem]{Definition}
 \newenvironment{proof}{\begin{trivlist} \item[]{\em Proof.}}{\end{trivlist}}
 \newcount\refno
 \def\B{{\text {BMO}}}
 \def\BC{{\text {BMC}}}

\def\CC{{\mathbb C}}

\def\BB{{\mathbb B}}
 
 \def\RR{{\mathbb R}}

 \def\ZZ{{\mathbb Z}}

 \def\SF{{\mathscr F}}
 \def\SH{{\mathscr H}}
  
 \def\SS{{\mathscr S}}

 \title{\bf The dual of the Hardy space associated with the Dunkl operators
 \thanks{{Supported by the National Natural
 Science Foundation of China (No. 12071295).}
 \newline
 \indent \,\,$^\dag$Corresponding author.
 \newline
 \indent\,\, E-mail: jiujiaxi78@163.com (J.-X. Jiu); lizk@shnu.edu.cn (Zh.-K. Li).}}

\author{{Jiaxi Jiu and Zhongkai Li$^{\dag}$}\\
{\small Department of Mathematics, Shanghai Normal University, Shanghai 200234, China}
}

\begin{document}
 \maketitle \setcounter{page}{1} \pagestyle{myheadings}
 \markboth{Jiu and Li}{Hardy space associated with the Dunkl operators}

 \begin{abstract}
 \noindent
 The rational Dunkl operators are commuting differential-reflection operators on the Euclidean space $\RR^d$ associated with a root system, that contain some non-local refection terms, and the associated Hardy space is defined by means of the Riesz transforms with respect to the Dunkl operators. The aim of the paper is to prove that its dual can be realized by a class of functions on $\RR^d$, denoted by $\BC_{\kappa}$ in the text, that consists of the underlying functions of a certain type of weighted Carleson measures. Our method is ``purely analytic" and does not depend on the atomic decomposition. As a corollary we obtain the Fefferman-Stein decomposition of functions in $\BC_{\kappa}$.

 \vskip .2in
 \noindent
 {\bf 2020 MS Classification:} 42B30, 42B20 (Primary), 51F15 (Secondary)
 \end{abstract}

\setcounter{page}{1}

\section{Introduction and main results}

Let $R$ be a root system in $\RR^d$ normalized so that $\langle\alpha,\alpha\rangle=2$ for $\alpha\in R$ and with $R_+$ a fixed positive subsystem, and $G$ the finite reflection group generated by the reflections $\sigma_\alpha$ ($\alpha\in R$), where $\sigma_{\alpha}(x)=x-\langle\alpha,x\rangle\alpha$ for $x\in\RR^d$ and $\langle\cdot,\cdot\rangle$ stands for the usual inner product on $\RR^d$. For a nonnegative multiplicity function $\kappa$ defined on $R$ (invariant under $G$), the (rational) Dunkl operators $D_j$ ($1\le j\le d)$ are defined by (cf. \cite{Du2})
\begin{eqnarray}\label{Dunkl-operator-1}
D_j=\partial_j+\sum_{\alpha\in R_+}\kappa(\alpha)\alpha_j\frac{1-\sigma_{\alpha}}{\langle\alpha,x\rangle},
 \end{eqnarray}
where $\sigma_{\alpha}f=f\circ\sigma_{\alpha}$ and $\alpha_j$ ($1\le j\le d$) is the $j$th coordinate of $\alpha$.
It is proved in \cite{Du2} that the Dunkl operators generate a commutative algebra and are anti-symmetry associated to the measure
\begin{align*}
d\omega_{\kappa}=W_{\kappa}\,dx,\qquad{\rm where}\quad  W_{\kappa}(x)=\prod_{\alpha\in R_+}|\langle\alpha,x\rangle|^{2\kappa(\alpha)}.
\end{align*}
The weight $W_{\kappa}$ is $G$-invariant and positive homogeneous of degree $2|\kappa|$, so that the homogeneous dimension of $d\omega_{\kappa}$ is
\begin{align*}
N=2|\kappa|+d,\qquad \hbox{where}\quad |\kappa|=\sum_{\alpha\in R_+}k(\alpha).
\end{align*}
It is also noted that $\RR^d$,
equipped with the Euclidean distance $(x,y)\in\RR^d\times\RR^d\mapsto|x-y|$ and with the measure $\omega_{\kappa}$,
is a space of homogeneous type in the sense of Coifman and Weiss \cite{CW1}.
For $0<p<\infty$, the space $L_{\kappa}^p(\RR^d)$ consists of all measurable functions $f$ on $\RR^d$ such
that $\|f\|_{L_{\kappa}^p}<+\infty$, where $\|f\|^p_{L_{\kappa}^p}=c_{\kappa}\int_{\RR^d}|f|^p\,d\omega_{\kappa}$ and $c_{\kappa}^{-1}=\int_{\RR^d}e^{-|x|^2/2}d\omega_{\kappa}$. $L_{\kappa}^{\infty}(\RR^d)$ is identical to the usual $L^{\infty}(\RR^d)$.

A typical example is $G=S_d$ (the type $A_{d-1}$ case), the symmetric group in $d$ elements, a root system of which is given by $R=\{\pm(e_i-e_j):\,1\le i<j\le d\}$. Since $S_d$ has only one orbit in $R$, a multiplicity function $\kappa$ is constant. The Dunkl operators for a given $\kappa\ge0$ have the form
\begin{eqnarray*}
D_j=\partial_j+\kappa\sum_{i\neq j}\frac{1-\sigma_{ij}}{x_j-x_i},\qquad j=1,2,\cdots,d,
 \end{eqnarray*}
where $\sigma_{ij}(\dots,x_i,\dots,x_j,\dots)=(\dots,x_j,\dots,x_i,\dots)$.

In the general setting, the Laplacian $\Delta_{\kappa}$ associated to a root system $R$ and a multiplicity function $\kappa$, called the $\kappa$-Laplacian, is defined by $\Delta_{\kappa}=\sum_{j=1}^{d}D_{j}^{2}$, which can be explicitly expressed by (cf. \cite{Du1})
\begin{align*}
\Delta_{\kappa}=\sum_{j=1}^{d}\partial_{j}^{2}+2\sum_{\alpha\in R_+}\kappa(\alpha)\left(\frac{\langle\alpha,\nabla\rangle}{\langle\alpha,x\rangle}-\frac{1-\sigma_{\alpha}}{\langle\alpha,x\rangle^2}\right),
\end{align*}
where $\nabla$ is the usual gradient operator on $\RR^{d}$. Note that the operator $-\Delta_{\kappa}$ is well defined on ${\mathscr S}(\RR^d)$ (the Schwartz space), symmetric and positive on $L_{\kappa}^2(\RR^d)$, and has a unique positive self-adjoint extension (cf. \cite{AH1}).

During the past several decades, the Dunkl operators have gained considerable interest in various fields of
mathematics and in physical applications; in particular, they admit a far reaching generalization of Fourier analysis which
includes most special functions related to root systems, such as spherical functions
on Riemannian symmetric spaces. See, for example, \cite{DX1,Ro5} on the rational Dunkl theory, \cite{Op1} on the trigonometric Dunkl theory, \cite{Ch1,Ma1} on affine Hecke algebras and the $q$-Dunkl theory, \cite{GRY1} on the probabilistic aspects of Dunkl theory, and \cite{Et1,LV} on integrable systems related to the Dunkl theory.

We consider the Riesz transforms associated with the Dunkl operators defined by
$$
{\mathcal{R}}_j=D_j(-\Delta_{\kappa})^{-1/2},\qquad j=1,\cdots,d,
$$
and call them the $\kappa$-Riesz transforms. They can be interpreted according to the associated Fourier transform, i.e., the Dunkl transform; see Section 3 below. We note that these transforms were considered first in \cite{Liu1}. In the rank-one case ($d=1$), the boundedness in $L_{\kappa}^p$ for $1<p<\infty$, of the only one member named the $\kappa$-Hilbert transform, was proved in \cite{TX2} whenever $2|\kappa|$ is a nonnegative integer, and then in \cite{AGS1} without restriction on $2|\kappa|$; furthermore, the weak (1,1) boundedness of the maximal $\kappa$-Hilbert transform was proved in \cite{LL1,Liao1}. In the higher-dimensional case, the boundedness of the $\kappa$-Riesz transforms in $L_{\kappa}^p$ for $1<p<\infty$ was proved in \cite{AS1}.

Analogously to the classical case, we define the Hardy space in the Dunkl setting by
\begin{align*}
H_{\kappa}^1(\RR^d)=\left\{f\in L_{\kappa}^1(\RR^d):\quad {\mathcal{R}}_jf\in L_{\kappa}^1(\RR^d),\,\,\,j=1,\cdots,d\right\},
\end{align*}
and endow it with the norm $\|f\|_{H_{\kappa}^1}=\|f\|_{L_{\kappa}^1}+\sum_{j=1}^d\|{\mathcal{R}}_jf\|_{L_{\kappa}^1}$. The space $H_{\kappa}^1$ for the rank-one case was studied first in \cite{LL1,Liao1} (including $H_{\kappa}^p$ for some values $p<1$; also see \cite{LL2} for its analogue on the circle), and for the product case in \cite{ABDH1,Dz1}. Associated to general root systems in the higher-dimensional case, the Hardy spaces $H_{\kappa}^1(\RR^d)$ was recently exploited in \cite{ADH1,DH1}, and in particular, it was proved in \cite{DH1} that $H_{\kappa}^1(\RR^d)$ admits the atomic decomposition in the sense of Coifman and Weiss \cite{CW1} basing on a square function characterization of $H_{\kappa}^1(\RR^d)$ given in \cite{ADH1} and a decomposition of the tent space developed in \cite{CMS1,Ru1}. However in \cite{ADH1}, such a square function characterization of $H_{\kappa}^1(\RR^d)$ was obtained with arduous effort, and besides some subtle estimates of the associated heat kernel, through a variety of distinct approaches from \cite{DKKP1,HLMMY,Ru1,SY1}, etc. For other aspects of the Hardy spaces, see \cite{DY2,Dur1,FS1,MS1,St2,St3,SW1,Uch1} for example.

We focus on the ``analytical" presentation of related theories, because in Dunkl's setting there are the relevant Fourier transform, the generalized translation, Green's formula, etc, and we expect that in this way one can better clarify the treatment of some problems, and make more in-depth and extensive research, or in other words, extend the theory of harmonic analysis as far as possible to the Dunkl setting. For example, in a recent work \cite{JL2} a Lusin-type area integral is defined by
\begin{align*}
\left(S_{a,h}u\right)(x)=\left(\iint_{\Gamma_{a}^{h}(0)}
\left[\tau_{x}(\Delta_{\tilde{\kappa}}u^{2})\right](x_0,t)\,x_0^{1-d-2|\kappa|}\,dx_0d\omega_{\kappa}(t)\right)^{1/2}
\end{align*}
for a $C^2$ function $u$ on the upper half-space $\RR^{1+d}_+:=(0,\infty)\times\RR^d$, where $\Gamma^h_a(0)$ is the positive truncated cone of aperture $a>0$ and height $h>0$ with vertex $(0,0)$, $\Delta_{\tilde{\kappa}}=\partial_{x_0}^2+\Delta_{\kappa}$, and $\tau_x$ is the generalized translation in the Dunkl setting. We have applied the function $S_{a,h}u$ in \cite{JL2} to characterize the local existence of non-tangential boundary values of a generalized harmonic function $u$ on $\RR^{1+d}_+$.

In the present paper we take up the $\kappa$-Riesz transforms ${\mathcal{R}}_j$ ($1\le j\le d$) and the Hardy spaces $H_{\kappa}^1(\RR^d)$ from the so-called analytic point of view. These transforms are treated in conjunction with the associated conjugate Poisson integrals, which allows us to borrow some conclusions about the generalized harmonic functions and yields a better understanding of them. We give a direct proof of the boundedness of the $\kappa$-Riesz transforms from $L^\infty(\RR^d)$ to the space $\B_\kappa(\RR^d)$ of functions of bounded mean oscillation with respect to the measure $d\omega_{\kappa}$. The major portion of the the paper is to prove that the dual of the Hardy spaces $H_{\kappa}^1(\RR^d)$ can be realized by a class of functions on $\RR^d$, denoted by $\BC_{\kappa}(\RR^d)$, that consists of the underlying functions of a certain type of Carleson measures associated with $d\omega_{\kappa}$.
The proof is based on the original ideas of Fefferman and Stein \cite{FS1} and does not depend on the atomic decomposition. As a corollary we obtain the Fefferman-Stein decomposition of functions in $\BC_{\kappa}(\RR^d)$. Our argument naturally implies that $\BC_{\kappa}(\RR^d)$ is a subclass of $\B_{\kappa}(\RR^d)$, and the converse of this inclusion is a consequence of the atomic decomposition in \cite{DH1}.

The paper is organized as follows. Some basic facts on the rational Dunkl theory are summarized in Section 2. Section 3 is devoted to the theory of the associated conjugate Poisson integrals, i.e., the conjugate $\kappa$-Poisson integrals, and the $\kappa$-Riesz transforms ${\mathcal{R}}_j$ ($1\le j\le d$); and in particular, it is concluded that the maximal $\kappa$-Riesz transforms ${\mathcal{R}}_j^*$ are of strong type $(p,p)$ for $1<p<\infty$, and of weak type $(1,1)$. The fundamental results of the ``analytic" Hardy space on $\RR^{1+d}_+$ in the Dunkl setting, denoted by $\SH_{\kappa}^1(\RR^{1+d}_+)$, and the ``real" Hardy space $H_{\kappa}^1(\RR^d)$ are contained also in Section 3. That the $\kappa$-Riesz transforms map $L^\infty(\RR^d)$ into the space $\B_\kappa(\RR^d)$ is proved in Section 4. In Section 5, we introduce the class $\BC_{\kappa}(\RR^d)$ of functions on $\RR^d$ in terms of Carleson measures associated with $d\omega_{\kappa}$ (see Definition \ref{bmc-a}), and prove that the $\kappa$-Riesz transforms are bounded from $L^\infty(\RR^d)$ to $\BC_{\kappa}(\RR^d)$. Section 6 is devoted to the proof of that the space $\BC_{\kappa}(\RR^d)$ is the dual of the Hardy spaces $H_{\kappa}^1(\RR^d)$. Several remarks on the space $H_{\kappa}^1(\RR^d)$, $\BC_{\kappa}(\RR^d)$ and $\B_\kappa(\RR^d)$ are given in the last section.

Throughout the paper, $X\lesssim Y$ means that $X\le cY$ for some constant $c>0$ independent of variables, functions, etc.
We say $X\asymp Y$ if both $X\lesssim Y$ and $Y\lesssim X$ hold, and for a measurable set $E\subset\RR^d$, put
$|E|_{\kappa}=\int_{E}d\omega_{\kappa}$.
As usual $\SS(\RR^d)$ designates the space of
$C^{\infty}$ functions on $\RR^d$ rapidly decreasing together with
their derivatives,
$C_0(\RR^d)$ the space of continuous functions on $\RR^d$
vanishing at infinity with norm $\|f\|_{C_0}=\sup_{x\in\RR^d}|f(x)|$,
$L_{\kappa,{\rm loc}}(\RR^d)$ the set of locally integrable functions
on $\RR^d$ with respect to the measure $d\omega_{\kappa}$, and ${\frak
B}_{\kappa}(\RR^d)$ the space of Borel measures $d\nu$ on $\RR^d$ for
which $\|d\nu\|_{{\frak B_{\kappa}}}
=c_{\kappa}\int_{\RR^d}W_{\kappa}\,|d\nu|$ is finite.
We shall use the notation $B(x,r)$ (or $B((x_0,x),r)$) to denote the open ball in $\RR^d$ (or $\RR^{1+d}$) with radius $r$ centered at $x\in\RR^d$ (or $(x_0,x)\in\RR^{1+d}$).

\section{Some facts in the Dunkl theory}

In this section we give an account on results from the Dunkl theory which will be relevant for the sequel. Concerning root systems
and reflection groups, see \cite{Hu}.

\subsection{The Dunkl kernel and the Dunkl transform}

As in the first section, for a root system $R$ let $G$ be the associated reflection group and $\kappa$ a given nonnegative multiplicity function. As shown in \cite{Du3}, there exists a unique degree-of-homogeneity-preserving linear
isomorphism $V_{\kappa}$ on polynomials, which intertwines the associated commutative algebra of Dunkl operators and the algebra of usual partial differential operators, namely,
$D_jV_{\kappa}=V_{\kappa}\partial_j$ for $1\le i\le d$.  $V_{\kappa}$ commutes with the group action of $G$. (For a thorough analysis on $V_{\kappa}$ with general $\kappa$, see \cite{DJO}.)
It was proved in \cite{Ro3} that $V_{\kappa}$ is positive on the space of polynomials, and moreover, for each $x\in\RR^d$ there exists a unique (Borel) probability measure $\mu^{\kappa}_{x}$ of $\RR^d$ such that
$$
V_{\kappa}f(x)=\int_{\RR^d}f(\xi)\,d\mu^{\kappa}_{x}(\xi),
$$
where $\mu^{\kappa}_{x}$ is of compact support with
\begin{align}\label{intertwining-support-1}
\hbox{supp}\,\mu^{\kappa}_{x}\subseteq\hbox{co}\,\{\sigma(x):\,\,\sigma\in G\},
\end{align}
the convex hull of the orbit of $x$ under $G$.
The intertwining operator $V_{\kappa}$ has an extension to $C^{\infty}(\RR^d)$ and establishes a homeomorphism of this space (see \cite{Tr1,Tr2}). For $\delta>0$ and $x\in\RR^d$£¬ we have the distribution estimate (see \cite{JL1})
\begin{align*}
\int_{\langle x,\xi\rangle>|x|^2-\delta^2}\,d\mu^{\kappa}_{x}(\xi)\asymp
\frac{\delta^{2|\kappa|+d}}{|\BB(x,\delta)|_{\kappa}},
\end{align*}
which reveals the behavior of the representing measure $\mu^{\kappa}_{x}$ near the point $x$.

Associated with $G$ and $\kappa$, the Dunkl kernel is defined by
\begin{eqnarray*}
E_{\kappa}(x,z)=V_{\kappa}\left(e^{\langle\cdot,z\rangle}\right)(x)
=\int_{\RR^d}e^{\langle\xi,z\rangle}\,d\mu^{\kappa}_{x}(\xi), \qquad x\in\RR^d,\,\,z\in\CC^d.
\end{eqnarray*}
According to \cite{Op}, for fixed $z\in\CC^d$, $x\mapsto E_{\kappa}(x,z)$ may
be characterized as the unique analytic solution of the system
\begin{align}\label{Dunkl-kernel-eigenfunction-1}
D^x_{j}E_{\kappa}(x,z)=z_jE_{\kappa}(x,z), \qquad x\in\RR^d,\,\,\,j=1,\dots,d,
\end{align}
with the initial value $E_{\kappa}(0,z)=1$; and $E_{\kappa}$ has a unique holomorphic extension to $\CC^d\times\CC^d$ and is symmetric in its arguments (see \cite{dJ,Op}). Furthermore (cf. \cite{dJ,Du3,Ro3}), for $z,w\in\CC^d$, $\lambda\in\CC$ and $\sigma\in G$, one has
\begin{align}\label{Dunkl-kernel-2-3}
E_{\kappa}(\lambda z,w)=E_{\kappa}(z,\lambda w),\qquad E_{\kappa}(\sigma(z),\sigma(w))=E_{\kappa}(z,w),
\end{align}
and for $x\in\RR^d$, $z\in\CC^d$, and all multi-indices $\nu=(\nu_1,\dots,\nu_d)$,
\begin{align}\label{Dunkl-kernel-2-4}
\left|\partial_z^{\nu}E_{\kappa}(x,z)\right|\le|x|^{|\nu|}\max_{\sigma\in G}e^{{\rm Re}\,\langle\sigma(x),z\rangle}.
\end{align}

For $f\in L_{\kappa}^1(\RR^d)$, its Dunkl transform is defined by
\begin{eqnarray*}
\left(\SF_{\kappa}f\right)(\xi)=c_{\kappa}\int_{\RR^d}f(x)E_{\kappa}(-i\xi,x)\,d\omega_{\kappa}(x),\qquad \xi\in\RR^d.
\end{eqnarray*}
The Dunkl transform $\SF_{\kappa}$ commutes with the $G$-action, and shares many of the important properties with the
usual Fourier transform (see \cite{Du4,dJ,Ro3,RV1}), part of which are listed as follows.

\begin{proposition} \label{transform-a} {\rm(i)} If $f\in L_{\kappa}^1(\RR^d)$, then ${\SF}_{\kappa}f\in C_0({\RR^d})$ and $\|{\SF}_{\kappa}f\|_{C_0}\leq\|f\|_{L_{\kappa}^1}$.

{\rm (ii)} {\rm (Inversion)} \ If $f\in L_{\kappa}^1(\RR^d)$
such that $\SF_{\kappa}f\in L_{\kappa}^1(\RR^d)$, then
$f(x)=[\SF_{\kappa}(\SF_{\kappa}f)](-x)$.

{\rm(iii)} For $f\in{\mathscr S}({\RR^d})$, we have $[\SF_{\kappa}(D_jf)](\xi)=i\xi_j(\SF_{\kappa}f)(\xi)$, $[\SF_{\kappa}(x_jf)](\xi)=i[D_j(\SF_{\kappa}f)](\xi)$ for
$\xi\in\RR^d$ and $1\le j\le d$; and $\SF_{\kappa}$ is a homeomorphism of ${\mathscr S}(\RR^d)$.

{\rm(iv)} {\rm (product formula)} For $f_1,f_2\in
L_{\kappa}^1(\RR^d)$, we have $\int_{\RR^d}f_1\cdot\SF_{\kappa}f_2\,d\omega_{\kappa}=\int_{\RR^d}f_2\cdot\SF_{\kappa}f_1\,d\omega_{\kappa}$.

{\rm (v)} {\rm (Plancherel)} There exists a unique extension of
$\SF_{\kappa}$ to $L_{\kappa}^2(\RR^d)$ with
$\|\SF_{\kappa}f\|_{L_{\kappa}^2}=\|f\|_{L_{\kappa}^2}$ and $\SF_{\kappa}^{-1}f=(\SF_{\kappa}f)(-\cdot)$.

{\rm (vi)} If $f\in L_{\kappa}^1(\RR^d)\cup L_{\kappa}^2(\RR^d)$ is radial, then its Dunkl transform ${\SF}_{\kappa}f$ is also radial.

{\rm (vii)} If $f\in L_{\kappa}^1(\RR^d)\cup L_{\kappa}^2(\RR^d)$, then  $[\SF_{\kappa}(\sigma f)](\xi)=(\SF_{\kappa}f)(\sigma(\xi))$ for $\xi\in\RR^d$ and $\sigma\in G$; and $[\SF_{\kappa}(f(a\cdot))](\xi)=|a|^{-2|\kappa|-d}(\SF_{\kappa}f)(a^{-1}\xi)$ for $\xi\in\RR^d$ and $a\in\RR\setminus\{0\}$.
\end{proposition}

\subsection{The generalized translation and convolution}

For $x\in\RR^d$, the generalized translation $\tau_x$ is defined in $L_{\kappa}^2(\RR^d)$ by (see \cite{TX})
\begin{align}\label{translation-2-0}
\left[\SF_{\kappa}(\tau_xf)\right](\xi)=E_{\kappa}(i\xi,x)(\SF_{\kappa}f)(\xi),\qquad \xi\in\RR^d.
\end{align}
By Proposition \ref{transform-a}(v), $\tau_x$ is well defined and satisfies $\|\tau_xf\|_{L_{\kappa}^2}\le\|f\|_{L_{\kappa}^2}$ in view of (\ref{Dunkl-kernel-2-4}).

If $f$ is in an appropriate subclass, for example, $A_{\kappa}(\RR^d)=\left\{f\in L_{\kappa}^1(\RR^d):\,\SF_{\kappa}f\in L_{\kappa}^1(\RR^d)\right\}$ (cf. \cite{TX}) or ${\mathscr S}(\RR^d)$, $\tau_xf$ may be expressed pointwise by (see \cite{Ro2,TX})
\begin{align}\label{translation-2-2}
(\tau_x f)(t)=c_{\kappa}\int_{\RR^d}E_{\kappa}(i\xi,x)E_k(i\xi,t)(\SF_{\kappa}f)(\xi)\,d\omega_{\kappa}(\xi),\qquad x\in\RR^d.
 \end{align}

\begin{proposition}\label{translation-2-a}
{\rm(i)} If $f,g\in L_{\kappa}^2(\RR^d)$, then for $x\in\RR^d$,
\begin{align}\label{translation-2-1}
\int_{\RR^d}(\tau_xf)(t)g(t)\,d\omega_{\kappa}(t)=\int_{\RR^d}f(t)(\tau_{-x}g)(t)\,d\omega_{\kappa}(t),
 \end{align}
and both sides are continuous in $x$.

{\rm(ii)} For $f\in A_{\kappa}(\RR^d)$, the function $(x,t)\mapsto(\tau_x f)(t)$ is continuous on $\RR^d\times\RR^d$.
\end{proposition}

The equality (\ref{translation-2-1}) was proved in \cite{TX} for $f\in A_{\kappa}(\RR^d)$ and $g\in L_{\kappa}^1(\RR^d)\cap L^{\infty}(\RR^d)$, and the general case follows from a density argument; for the continuity, see \cite{JL2}.

In \cite{Tr2} an abstract form of $\tau_x$ is given by $(\tau_xf)(t)=V_{\kappa}^tV_{\kappa}^x\left[(V_{\kappa}^{-1}f)(x+t)\right]$ for $f\in C^{\infty}(\RR^d)$ and $x,t\in\RR^d$.

\begin{proposition}\label{translation-2-b} {\rm(\cite{Tr2})}
{\rm(i)} For fixed $x\in\RR^d$, $\tau_x$ is a continuous linear mapping from $C^{\infty}(\RR^d)$ into itself, and satisfies $D_j(\tau_xf)=\tau_x(D_jf)$ ($j=1,\dots,d$) on $\RR^d$ for $f\in C^{\infty}(\RR^d)$;

{\rm(ii)} for fixed $x,t\in\RR^d$, the mapping $f\mapsto(\tau_xf)(t)$ defines a compactly supported distribution, whose support is contained in the ball $\{\xi\in\RR^d:\,|\xi|\le|x|+|t|\}$;

{\rm(iii)} for fixed $x\in\RR^d$, if $f\in{\mathscr S}(\RR^d)$, then $\tau_xf\in{\mathscr S}(\RR^d)$ too, and both (\ref{translation-2-0}) and (\ref{translation-2-2}) hold.

{\rm(iv)} If $f\in C^{\infty}(\RR^d)$, then the function $(x,t)\mapsto (\tau_x f)(t)$ is in $C^{\infty}(\RR^d\times\RR^d)$, and for $x,t\in\RR^d$, $(\tau_tf)(x)=(\tau_xf)(t)$.
\end{proposition}

\begin{proposition}\label{translation-2-b-1}
If $f\in C^{\infty}(\RR^d)\cup L_{\kappa}^2(\RR^d)$, then {\rm (i)} $\left[\tau_x(\sigma f)\right](t)=(\tau_{\sigma(x)}f)(\sigma(t))$ for $x,t\in\RR^d$ and $\sigma\in G$; {\rm (ii)} $\left[\tau_x(f(a\cdot))\right](t)=(\tau_{ax}f)(at)$ for $x,t\in\RR^d$ and $a\in\RR\setminus\{0\}$; {\rm (iii)} $\left[\tau_t(\tau_xf)\right](z)=\left[\tau_x(\tau_tf)\right](z)$ for $x,t,z\in\RR^d$.
\end{proposition}

If $f\in L_{\kappa}^2(\RR^d)$, the required identities in the above proposition are valid by Proposition \ref{transform-a}(vii), (\ref{Dunkl-kernel-2-3}) and (\ref{translation-2-0}). For general $f\in C^{\infty}(\RR^d)$ and for fixed  $x,t\in\RR^d$, we take a compactly supported $\phi\in C^{\infty}(\RR^d)$ with $\phi(\xi)=1$ for $|\xi|\le|x|+|t|$. It then follows that $\left[\tau_x(\sigma(\phi f))\right](t)=\left[\tau_{\sigma(x)}(\phi f)\right](\sigma(t))$, which is identical with $\left[\tau_x(\sigma f)\right](t)=(\tau_{\sigma(x)}f)(\sigma(t))$ by Proposition \ref{translation-2-b}(ii), so that part (i) is done. Parts (ii) and (iii) can be verified similarly.

\begin{proposition}\label{translation-2-b-2} {\rm (\cite[Corollary 4.21]{DH0})}
If $f\in{\mathscr S}(\RR^d)$, then $\sup_{x\in\RR^d}\|\tau_xf\|_{L^{1}_{\kappa}}<\infty$.
\end{proposition}

As for radial functions in $\RR^d$, more information about the translation operator $\tau_x$ has been found. In fact, for $x,t\in\RR^d$, by \cite[Theorem 5.1]{Ro4} there exists a unique compactly supported, radial (Borel) probability measure $\rho^{\kappa}_{x,t}$ on $\RR^d$ such that for all radial $f\in C^{\infty}(\RR^d)$,
\begin{align}\label{translation-2-3}
(\tau_xf)(t)=\int_{\RR^d}f\,d\rho^{\kappa}_{x,t};
\end{align}
and the support of $\rho^{\kappa}_{x,t}$ is contained in
\begin{align}\label{translation-support-2-1}
\left\{\xi\in\RR^d:\quad \min_{\sigma\in G}|x+\sigma(t)|\le|\xi|\le\max_{\sigma\in G}|x+\sigma(t)|\right\}.
\end{align}
In particular, if $0\in{\rm supp}\,\rho^{\kappa}_{x,t}$, then the $G$-orbits of $x$ and $-t$ coincide. Since ${\rm supp}\,\rho^{\kappa}_{x,t}$ is compact, (\ref{translation-2-3}) leads to a natural extension of  the translation operator $\tau_x$ to all radial $f\in C(\RR^d)$. Furthermore, if $f\in C^{\infty}(\RR^d)$ is radial, say $f(x)=f_0(|x|)$, from the proof of \cite[Theorem 5.1]{Ro4} it follows that, for $x,t\in\RR^d$,
\begin{align}\label{translation-2-4}
(\tau_xf)(t)=\int_{\RR^d}f_0(\sqrt{|x|^2+|t|^2+2\langle t,\xi\rangle})\,d\mu^{\kappa}_{x}(\xi).
\end{align}
Again, since $\mu^{\kappa}_{x}$ is of compact support, (\ref{translation-2-4}) allows an extension of $\tau_x$ to all radial $f\in C(\RR^d)$ by a density argument.

\begin{proposition}\label{translation-2-d} {\rm (\cite[Propositions 3.1 and 3.2]{JL2})}
{\rm(i)} If $f\in C^{\infty}(\RR^d)$, and $g\in C(\RR^d)$ is radial and of compact support, or if $f\in C^{\infty}(\RR^d)$ is of compact support and $g\in C(\RR^d)$ is radial, then (\ref{translation-2-1}) holds.

{\rm(ii)} If $f\in C(\RR^d)$ is radial, then $\left[\tau_x(f(a\cdot))\right](t)=(\tau_{ax}f)(at)$ for $x,t\in\RR^d$ and $a\in\RR\setminus\{0\}$.
\end{proposition}

\begin{proposition}\label{translation-2-e} {\rm (\cite[Theorems 3.6, 3.7 and 3.8]{TX})}
Let $x\in\RR^d$ be given.

{\rm(i)} The generalized translation operator $\tau_x$, initially defined on $L_{\kappa}^1(\RR^d)\cap L^{\infty}(\RR^d)$, can be extended to all radial functions $f$ in $L_{\kappa}^p(\RR^d)$, $1\le p\le2$, and for these $f$, one has $\|\tau_xf\|_{L_{\kappa}^p}\le\|f\|_{L_{\kappa}^p}$, and moreover, $\left[\SF_{\kappa}(\tau_xf)\right](\xi)=E_{\kappa}(i\xi,x)(\SF_{\kappa}f)(\xi)$ for $\xi\in\RR^d$.


{\rm(ii)} If $f\in L_{\kappa}^1(\RR^d)\cap L^{\infty}(\RR^d)$ is radial and nonnegative, then $(\tau_xf)(t)\ge0$ for almost every $t\in\RR^d$.

{\rm(iii)} If $f\in L_{\kappa}^1(\RR^d)$ is radial, then
$\int_{\RR^d}(\tau_xf)(t)\,d\omega_{\kappa}(t)=\int_{\RR^d}f(t)\,d\omega_{\kappa}(t)$.
\end{proposition}

\begin{corollary}\label{translation-2-f}
If  $f\in L_{\kappa}^p(\RR^d)$ ($1\le p\le2$) is radial, then $\lim_{t\rightarrow0}\|\tau_tf-f\|_{L_{\kappa}^p}=0$.
\end{corollary}

By Proposition \ref{translation-2-e}(i), it suffices to show the limit for $f$ in a dense set of $L_{\kappa}^p(\RR^d)$. For a radial $f\in{\mathscr S}(\RR^d)$, the Plancherel formula (Propositions \ref{transform-a}(v)) implies that
$$
\|\tau_tf-f\|_{L_{\kappa}^2}^2= c_{\kappa}\int_{\RR^d}\left|(E_{\kappa}(i\xi,t)-1)(\SF_{\kappa}f)(\xi)\right|^2\,d\omega_{\kappa}(\xi),
$$
which, in view of (\ref{Dunkl-kernel-2-4}), approaches to zero as $t\rightarrow0$ by the dominated convergence theorem. To deal with the case for $1\le p<2$, assume that $|f(x)|\lesssim(1+|x|^2)^{-m}$ with a suitably large $m>0$. If $|t|<1$ and $|x|\ge2$, then by (\ref{intertwining-support-1}) and (\ref{translation-2-4}),
\begin{align*}
|(\tau_tf)(x)|\lesssim \int_{|\xi|\le|t|}(1+|x|^2+|t|^2+2\langle x,\xi\rangle)^{-m}\,d\mu^{\kappa}_{t}(\xi)
\lesssim (1+|x|^2)^{-m}.
\end{align*}
Thus for arbitrary $M>2$, H\"older's inequality gives
\begin{align*}
\|\tau_tf-f\|_{L_{\kappa}^p}^p \lesssim |B_{M}(0)|_{\kappa}^{\frac{2-p}{2}}\|\tau_tf-f\|_{L_{\kappa}^2}^{p} +\int_{|x|>M}(1+|x|^2)^{-mp}\,d\omega_{\kappa}(x),
\end{align*}
so that $\limsup_{t\rightarrow0}\|\tau_tf-f\|_{L_{\kappa}^p}^p\lesssim \int_{|x|>M}(1+|x|^2)^{-mp}\,d\omega_{\kappa}(x)$. Letting $M\rightarrow\infty$ concludes the proof.

Now we consider the associated convolution $\ast_{\kappa}$. For suitable $f,g$ on $\RR^d$, define $f\ast_{\kappa}g$ by
\begin{align}\label{convolution-2-1}
(f\ast_{\kappa}g)(x)
=c_{\kappa}\int_{\RR^d}(\SF_{\kappa}f)(\xi)(\SF_{\kappa}g)(\xi)E_{\kappa}(i\xi,x)\,d\omega_{\kappa}(\xi),\qquad x\in\RR^d.
 \end{align}

\begin{proposition}\label{convolution-2-a}
{\rm(i)} If $f,g\in L_{\kappa}^2(\RR^d)$, or if $f\in L_{\kappa}^1(\RR^d)$ and $g\in A_{\kappa}(\RR^d)$, then $f\ast_{\kappa}g$ is well defined and continuous on $\RR^d$, and moreover,
\begin{eqnarray}\label{convolution-2-2}
(f\ast_{\kappa}g)(x)=c_{\kappa}\int_{\RR^d}f(t)(\tau_xg)(-t)\,d\omega_{\kappa}(t),\qquad x\in\RR^d.
\end{eqnarray}

{\rm(ii)} If $f\in L_{\kappa}^2(\RR^d)$ and $g\in  L_{\kappa}^1(\RR^d)\cap L_{\kappa}^2(\RR^d)$, then
$\|f\ast_{\kappa}g\|_{L_{\kappa}^2}\le\|f\|_{L_{\kappa}^2}\|g\|_{L_{\kappa}^1}$.

{\rm(iii)} If $f\in L_{\kappa}^1(\RR^d)$, and $g\in A_{\kappa}(\RR^d)$ is radial, then
$\|f\ast_{\kappa}g\|_{L_{\kappa}^1}\le\|f\|_{L_{\kappa}^1}\|g\|_{L_{\kappa}^1}$.

{\rm(iv)} If $f\in L_{\kappa}^1(\RR^d)$, and $g\in \SS(\RR^d)$, then $\|f\ast_{\kappa}g\|_{L_{\kappa}^1}\le C_{g}\|f\|_{L_{\kappa}^1}$, where $C_g>0$ is a constant depending on $g$.

{\rm(v)} In all cases (i), (ii) and (iii),
\begin{align*}
[\SF_{\kappa}(f\ast_{\kappa}g)](\xi)
=(\SF_{\kappa}f)(\xi)(\SF_{\kappa}g)(\xi),\qquad \xi\in\RR^d.
 \end{align*}
\end{proposition}

Part (i) follows from the product formula and the Plancherel formula, part (ii) from (\ref{convolution-2-1}) and the Plancherel formula, part (iii) from (\ref{convolution-2-2}), Propositions \ref{translation-2-a}(ii) and \ref{translation-2-e}(i), and Fubini's theorem, and part (iv) from (\ref{convolution-2-2}), Proposition \ref{translation-2-b-2} and Fubini's theorem. Part (v) is a consequence of (\ref{convolution-2-1}) and parts (i)-(iii).

\begin{proposition}\label{convolution-2-c}
{\rm(i)} If $f\in L_{\kappa}^1(\RR^d)$ and $\varphi\in\SS(\RR^d)$, then $f\ast_{\kappa}\varphi\in C^{\infty}(\RR^d)$.

{\rm(ii)} If $f\in C(\RR^d)$ and radial $\varphi\in\SS(\RR^d)$ satisfy ${\rm supp}\,f\subseteq\{x:\, r_2\le|x|\le r_1\}$ and ${\rm supp}\,\varphi\subseteq\{x:\,|x|\le r_3\}$ separately, where $r_1>r_2>r_3>0$, then
$$
{\rm supp}\, f\ast_{\kappa}\varphi\subseteq\{x:\,r_2-r_3\le|x|\le r_1+r_3\}.
$$
\end{proposition}

For part (i), since ${\SF}_{\kappa}f\in C_0({\RR^d})$ and ${\SF}_{\kappa}\varphi\in\SS(\RR^d)$ by Proposition \ref{transform-a}(i) and (iii), that $f\ast_{\kappa}\varphi\in C^{\infty}(\RR^d)$ follows immediately from (\ref{Dunkl-kernel-2-4}) and (\ref{convolution-2-1}). As for part (ii), if $|x|<r_2-r_3$ or $|x|>r_1+r_3$, then for $t\in{\rm supp}\,f$, $r_3<r_2-|x|\le|t|-|x|\le\min_{\sigma\in G}|x-\sigma(t)|$, or $r_3<|x|-r_1\le|x|-|t|\le\min_{\sigma\in G}|x-\sigma(t)|$, so that in view of (\ref{translation-support-2-1}), ${\rm supp}\,\varphi\cap{\rm supp}\,\rho^{\kappa}_{x,-t}=\emptyset$, and from (\ref{translation-2-3}), $(\tau_xg)(-t)=0$. Thus by (\ref{convolution-2-2}), $(f\ast_{\kappa}\varphi)(x)=0$ for $|x|<r_2-r_3$ or $|x|>r_1+r_3$, and part (ii) is done.

\begin{proposition}\label{convolution-2-b}{\rm (\cite[Theorems 4.1 and 4.2]{TX})}
Let $\varphi\in A_{\kappa}(\RR^d)$ be radial. Then

{\rm(i)} for $1\le p\le\infty$, the operator $T_{\varphi}f=f\ast_{\kappa}\varphi$, initially defined on $L_{\kappa}^1(\RR^d)\cap L^{\infty}(\RR^d)$, has a bounded extension to $L_{\kappa}^p(\RR^d)$, denoted by $T_{\varphi}f=f\ast_{\kappa}\varphi$ as well, and moreover, $\|f\ast_{\kappa}\varphi\|_{L_{\kappa}^p}\le\|\varphi\|_{L_{\kappa}^1}\|f\|_{L_{\kappa}^p}$;

{\rm(ii)} if, in addition, $c_{\kappa}\int_{\RR^d}\varphi(x)\,d\omega_{\kappa}(x)=1$, and set $\varphi_{\epsilon}(x)=\epsilon^{-2|\kappa|-d}\varphi(\epsilon^{-1}x)$ for $\epsilon>0$, then for $f\in X=L_{\kappa}^p(\RR^d)$, $1\le
p<\infty$, or $C_0(\RR^d)$, $\lim_{\epsilon\rightarrow0+}\|f\ast_{\kappa}\varphi_{\epsilon}-f\|_{X}=0$.
\end{proposition}

\subsection{The generalized spherical mean operator}

We note that it is still open whether the translation operator $\tau_x$ has a bounded extension to non-radial functions in $L_{\kappa}^1(\RR^d)$. However we may consider the generalized spherical mean operator $f\mapsto M_f$ associated to $G$ and $\kappa$, which is defined in \cite{MT}, for $x\in\RR^d$ and $r\in[0,\infty)$, by
\begin{align*}
M_f(x,r)=d_{\kappa}\int_{S^{d-1}}(\tau_{x}f)(rt')W_{\kappa}(t')dt',
\end{align*}
where $dt'$ denotes the area element on the unit sphere $S^{d-1}$, and $d_{\kappa}^{-1}=\int_{S^{d-1}}W_{\kappa}(t')dt'$. By Proposition \ref{translation-2-b}(i), $M_f$ is well defined for $f\in C^{\infty}(\RR^d)$ or $A_{\kappa}(\RR^d)$. Moreover, by \cite[Theorems 3.1, 4.1, and Corollary 5.2]{Ro4} the mapping $f\mapsto M_f$ is positivity-preserving on $C^{\infty}(\RR^d)$, and for $x\in\RR^d$ and $r\in[0,\infty)$, there exists a unique compactly supported probability (Borel) measure $\sigma^{\kappa}_{x,r}$ on $\RR^d$ such that for all $f\in C^{\infty}(\RR^d)$,
\begin{align}\label{spherical-mean-2-2}
M_f(x,r)=\int_{\RR^d}f\,d\sigma^{\kappa}_{x,r};
\end{align}
the support of $\sigma^{\kappa}_{x,r}$ is contained in
\begin{align*}
\{\xi\in\RR^d:\,\,|\xi|\ge||x|-r|\}\cap\left[\cup_{\sigma\in G} \{\xi\in\RR^d:\,\,|\xi-\sigma(x)|\le r\}\right];
\end{align*}
the mapping $(x,r)\mapsto\sigma_{x,r}^{\kappa}$ is
continuous with respect to the weak topology on $M^1(\RR^d)$ (the
space of probability measures);
and moreover,
\begin{align*}
\sigma_{\sigma(x),r}^{\kappa}(A)=\sigma_{x,r}^{\kappa}(\sigma^{-1}(A)), \qquad \sigma_{ax,ar}^{\kappa}(A)=\sigma_{x,r}^{\kappa}(a^{-1}A)
\end{align*}
for all $\sigma\in G$, $a>0$, and all Borel sets $A\in{\mathcal{B}}(\RR^d)$.

Since ${\rm supp}\,\rho^{\kappa}_{x,t}$ is compact, (\ref{spherical-mean-2-2}) gives a natural extension of the generalized spherical mean operator $f\mapsto M_f$ to all $f\in C(\RR^d)$.

\begin{proposition}\label{harmonic-2-a-1} {\rm (\cite[Corollary 3.6]{JL2})}
Assume that $\Omega$ is a $G$-invariant domain of $\RR^d$ and $f\in C^2(\Omega)$. Then $f$ satisfies $\Delta_{\kappa}f=0$ in $\Omega$ if and only if for all $x\in\Omega$, $M_f(x,r)=f(x)$
provided $\overline{B(x,r)}\subset\Omega$ for $r>0$. Moreover, in both cases, $f\in C^{\infty}(\Omega)$.
\end{proposition}

\section{Some fundamental aspects on the $\kappa$-Riesz transforms and the Hardy space $H_{\kappa}(\RR^d)$}

\subsection{The Green formulas}

In what follows, associated to the reflection group $G$ on $\RR^{d}$ and the multiplicity function $\kappa$, we consider the reflection group $\tilde{G}=\ZZ_2 \otimes G$ on $\RR^{1+d}=\RR\times\RR^d$ and the multiplicity function $\tilde{\kappa}=(0,\kappa)$. In this case, we write $\tilde{x}=(x_0,x)\in\RR^{1+d}$ and $\partial_0=\partial/\partial x_0$. The associated Laplacian, denoted by $\Delta_{\tilde{\kappa}}$ and called the $\tilde{\kappa}$-Laplacian, is given by $\Delta_{\tilde{\kappa}}=\partial_{0}^2+\Delta_{\kappa}$ or explicitly
\begin{align*}
\Delta_{\tilde{\kappa}}=\sum_{j=0}^{d}\partial_{j}^{2}+2\sum_{\alpha\in R_+}\kappa(\alpha)\left(\frac{\langle\alpha,\nabla\rangle}{\langle\alpha,x\rangle}-\frac{1-\sigma_{\alpha}}{\langle\alpha,x\rangle^2}\right),
\end{align*}
where $\nabla$ is the gradient operator on $\RR^{d}$.

Since the multiplicity value associated to $\ZZ_2$ is zero, the $\ZZ_2 \otimes G$-invariance of the domain $\Omega$ in \cite[Theorem 4.11]{MT} and Proposition \ref{harmonic-2-a-1} can be relaxed to the $G$-invariance. We reformulate the Green formulas in \cite[Theorem 4.11]{MT} and \cite[Proposition 3.8]{JL2} as follows.

\begin{proposition}\label{Green-formula-2-a}
Let $\Omega$ be a bounded, regular, and $G$-invariant domain of $\RR^{1+d}$. Then for $u,v\in C^2(\overline{\Omega})$,
\begin{align}\label{Green-formula-2-1}
\iint_{\Omega}(v\Delta_{\tilde{\kappa}}u-u\Delta_{\tilde{\kappa}}v)\,d\omega_{\kappa}(x)dx_0
=\int_{\partial\Omega}\left(v\partial_{\mathbf{n}}u
-u\partial_{\mathbf{n}}v\right)W_{\kappa}(x)d\sigma(x_0,x),
\end{align}
where $\partial_{\mathbf{n}}$ denotes the directional derivative of the outward normal, and $d\sigma(x_0,x)$ the area element on $\partial\Omega$; moreover the formula has the following variant
\begin{align*}
\iint_{\Omega}(v\Delta_{\tilde{\kappa}}u-u\Delta_{\tilde{\kappa}}v)\,d\omega_{\kappa}(x)dx_0
=\int_{\partial\Omega}\left(vD_{\mathbf{n}}u
-uD_{\mathbf{n}}v\right)W_{\kappa}(x)d\sigma(x_0,x),
\end{align*}
where $D_{\mathbf{n}}=\langle\nabla_{\tilde{\kappa}},\mathbf{n}\rangle$ is the directional Dunkl operator and $\nabla_{\tilde{\kappa}}=(\partial_0,D_1,\dots,D_d)$ is the $\tilde{\kappa}$-gradient.
\end{proposition}

A $C^2$ function $u$ defined on a $G$-invariant domain $\Omega$ of $\RR^{1+d}$ is said to be $\tilde{\kappa}$-harmonic if $\Delta_{\tilde{\kappa}}u=0$ there.

\subsection{The $\kappa$-Poisson integral}

We define the function $P$ by $P(x)=c_{d,\kappa}(1+|x|^2)^{-|\kappa|-(d+1)/2}$, where $c_{d,\kappa}=2^{|\kappa|+d/2}\pi^{-1/2}\Gamma(|\kappa|+(d+1)/2)$. It follows that $c_{\kappa}\int_{\RR^d}P\,d\omega_{\kappa}=1$ and $\left(\SF_{\kappa}P\right)(\xi)=e^{-|\xi|}$ (cf. \cite{RV1,TX}). For $f\in L_{\kappa}^p(\RR^d)$, $1\le p\le\infty$, the $\kappa$-Poisson integral of $f$ is defined by
\begin{align}\label{Poisson-2-2}
(Pf)(x_0,x)=c_{\kappa}\int_{\RR^{d}}f(t)(\tau_{x}P_{x_0})(-t)\,d\omega_{\kappa}(t),\qquad (x_0,x)\in\RR^{1+d}_+,
\end{align}
where $(\tau_{x}P_{x_0})(-t)$ is the $\kappa$-Poisson kernel with $P_{x_0}(x)=x_0^{-2|\kappa|-d}P(x/x_0)$ for $x_0>0$, i.e.,
\begin{align*}
P_{x_0}(x)=\frac{c_{d,\kappa}\,x_0}{(x_0^2+|x|^2)^{|\kappa|+(d+1)/2}},\qquad x\in\RR^d.
\end{align*}
Note that the function $(x_0,x)\mapsto P_{x_0}(x)$ is $\tilde{\kappa}$-harmonic in $\RR_+^{1+d}$, and so is the function $(x_0,x)\mapsto(\tau_{x}P_{x_0})(-t)$ by Propositions \ref{translation-2-b}(i) and (iv). In view of \cite[Proposition 5.1]{ADH1}, taking differentiation under the integral sign on the right hand side of (\ref{Poisson-2-2}) is legitimate, and hence, the $\kappa$-Poisson integral $(Pf)(x_0,x)$ of $f\in L_{\kappa}^p(\RR^d)$ is $\tilde{\kappa}$-harmonic in $\RR_+^{1+d}$.

Since $(\SF_{\kappa}P_{x_0})(\xi)=e^{-x_0|\xi|}$, from (\ref{translation-2-0}) and Proposition \ref{translation-2-b-1}(ii) one has
\begin{align}\label{Poisson-kernel-1}
\SF_{\kappa}[(\tau_{x}P_{x_0})(-\cdot)](\xi)=e^{-x_0|\xi|}E_{\kappa}(-ix,\xi),
\end{align}
and then for $f\in L_{\kappa}^1(\RR^d)$, the product formula (Propositions \ref{transform-a}(iv)) gives
\begin{align}\label{Poisson-2-3}
(Pf)(x_0,x)=c_{\kappa}\int_{\RR^{d}}e^{-x_0|\xi|}\left(\SF_{\kappa}f\right)(\xi)E_{\kappa}(i\xi,x)\,d\omega_{\kappa}(\xi),\qquad (x_0,x)\in\RR^{1+d}_+.
\end{align}
If $f\in L_{\kappa}^2(\RR^d)$, the formula (\ref{Poisson-2-3}) holds too by the Plancherel formula (Propositions \ref{transform-a}(v)).

For the following two propositions, see \cite[Theorems 4.1 and 5.4]{TX} and \cite[Theorem 7.5]{ADH1}.

\begin{proposition} \label{Poisson-c} For $f\in L_{\kappa}^p(\RR^d)$ ($1\le p\le\infty$) and $x_0>0$, $\|(Pf)(x_0,\cdot)\|_{L_{\kappa}^p}\le\|f\|_{L_{\kappa}^p}$; and for $f\in X=L_{\kappa}^p(\RR^d)$, $1\le
p<\infty$, or $C_0(\RR^d)$, $\lim_{x_0\rightarrow0+}\|(Pf)(x_0,\cdot)-f\|_{X}=0$.
\end{proposition}

\begin{proposition} \label{Poi-Int-Cha}
Let $u$ be $\tilde{\kappa}$-harmonic in $\RR^{1+d}_+$. Then
{\rm (i)} for $1<p<\infty$, $u$ is the $\kappa$-Poisson integral of some $f\in L^p_{\kappa}(\RR^d)$ if and only if $u$ satisfies
\begin{eqnarray}\label{Poi-3}
A:=\sup_{x_0>0}\|u(x_0,\cdot)\|_{L^p_{\kappa}}<+\infty,
\end{eqnarray}
and moreover $\|f\|_{L^p_{\kappa}}=A$; and {\rm (ii)} $u$ is the $\kappa$-Poisson integral of $d\mu\in {\frak B}_{\kappa}(\RR^d)$ if and only if $u$ satisfies (\ref{Poi-3}) for $p=1$.
\end{proposition}

As usual, we denote by $\Gamma_a(x)$ the positive cone of aperture $a>0$ with vertex $(0,x)\in \partial\RR^{1+d}_+=\RR^d$, and $\Gamma^h_a(x)$ the truncated cone with height $h>0$, that is,
$$
\Gamma^h_a(x)=\left\{(x_0,t):\quad |t-x|<ax_0 \quad\hbox{for}\,\,t\in\RR^d,\,\,0<x_0<h\right\}.
$$
For a function $u$ on $\RR^{1+d}_+$ and for $a>0$, the non-tangential maximal function $u^*_a(x)$ is defined by
$$
u^*_a(x)=\sup_{(x_0,t)\in\Gamma_a(x)}|u(x_0,t)|;
$$
and for $x\in\RR^d$, that $u$ has a non-tangential limit $A$ at $(0,x)$ means that for every $a>0$, $\lim u(x_0,t)=A$ as $(x_0,t)\in\Gamma_a(x)$ approaching to $(0,x)$.

If $u(x_0,x)$ is the $\kappa$-Poisson integral of $f\in
L^p_{\kappa}(\RR^d)$, $1\le p\le\infty$, we write $(P^*_af)(x)=u^*_a(x)$, and $(P^+f)(x)=\sup_{x_0>0}|(Pf)(x_0,x)|$, the associated perpendicular maximal function.

\begin{proposition} \label{Poisson-Integral-Max-N}
Both $P^+$ and $P^*_a$ for $a>0$ are bounded from $L_{\kappa}^p(\RR^d)$ ($1<p\le\infty$) to itself and from $L_{\kappa}^1(\RR^d)$ to $L_{\kappa}^{1,\infty}(\RR^d)$, where
$$
L_{\kappa}^{1,\infty}(\RR^d)=\{f:\,\|f\|_{L_{\kappa}^{1,\infty}}:=\sup_{\lambda>0}\lambda|\{x\in\RR^d:\,|f(x)|>\lambda\}|_{\kappa}<\infty\}.
$$
\end{proposition}

The conclusions in the proposition are known. Those for $P^+$ are consequences of the Hopf-Dunford-Schwartz ergodic theorem, see \cite[pp. 48-9]{St11} and \cite[p. 46]{TX}. As for $P^*_a$,
since $(\tau_{x}P_{x_0})(-t)=(\tau_{-t}P_{x_0})(x)$ by (\ref{translation-2-2}), it follows from (\ref{translation-2-4}) that, for $x,t\in\RR^d$,
\begin{align*}
(\tau_{x}P_{x_0})(-t)=\int_{\RR^d}\frac{c_{d,\kappa}\,x_0}{(x_0^2+|x|^2+|t|^2+2\langle x,\xi\rangle)^{|\kappa|+(d+1)/2}}\,d\mu^{\kappa}_{-t}(\xi);
\end{align*}
but since $|\xi|\le|t|$ for $\xi\in\hbox{supp}\,\mu^{\kappa}_{-t}$ by (\ref{intertwining-support-1}), one has
$$
x_0^2+|x|^2+|t|^2+2\langle x,\xi\rangle \asymp\left(x_0+|x+\xi|+\sqrt{|t|^2-|\xi|^2}\right)^2,
$$
and then, for $x'\in\RR^d$ satisfying $|x'-x|<ax_0$,
\begin{align*}
x_0^2+|x|^2+|t|^2+2\langle x,\xi\rangle
\lesssim(a+1)^2\left(x_0^2+|x'|^2+|t|^2+2\langle x',\xi\rangle\right).
\end{align*}
This implies that $(\tau_{x'}P_{x_0})(-t)\lesssim(a+1)^{2|\kappa|+d+1}(\tau_{x}P_{x_0})(-t)$ for $(x_0,x')\in\Gamma_a(x)$, and hence
\begin{align*}
(P^*_af)(x)\lesssim (a+1)^{2|\kappa|+d+1}\left[P^+(|f|)\right](x).
\end{align*}
The conclusions for $P^*_a$ follow immediately.

\begin{corollary} \label{Poisson-d}
If $f\in L^p_{\kappa}(\RR^d)$, $1\le p\le\infty$, then its $\kappa$-Poisson integral $Pf$ converges non-tangentially
to $f$ almost every on $\RR^d$.
\end{corollary}

The above corollary is a direct consequence of Proposition \ref{Poisson-Integral-Max-N} (cf. \cite[Corollary 5.4]{ADH1})), and is also a special case of the following local Fatou theorem for $\tilde{\kappa}$-harmonic functions.

\begin{theorem}\label{local-Fatou-1} {\rm (\cite[Theorem 2.1]{JL2})}
Suppose that the function $u$ is $\tilde{\kappa}$-harmonic in $\RR^{1+d}_+$ and $E$ is a $G$-invariant measurable subset of $\partial\RR^{1+d}_+=\RR^d$. Then $u$ has a finite non-tangential limit at $(0,x)$ for almost every $x\in E$ if and only if $u$ is non-tangentially bounded at $(0,x)$ for almost every $x\in E$.
\end{theorem}

\subsection{The conjugate $\kappa$-Poisson integrals and the $\kappa$-Riesz transforms}

In this subsection we consider the $\kappa$-Riesz transforms by the ``analytic" approach, i.e., starting from the associated conjugate Poisson integrals. This is different from the way in \cite{AS1} and some conclusions about these transform can be implied by the know results of the $\tilde{\kappa}$-harmonic functions. The detailed proofs of the claims in this subsection will be given in a separate paper (cf. \cite{JL3}).

For $f\in L^p_{\kappa}(\RR^d)$ ($1\le p<\infty$) and for $j=1,\cdots,d$, the $j$th conjugate $\kappa$-Poisson integral of $f$ is defined by
\begin{align*}
(Q_jf)(x_0,x)=c_{\kappa}\int_{\RR^{d}}f(t)(\tau_{x}Q^{(j)}_{x_0})(-t)\,d\omega_{\kappa}(t),\qquad (x_0,x)\in\RR^{1+d}_+,
\end{align*}
where
\begin{align*}
Q^{(j)}_{x_0}(x)=\frac{c_{d,\kappa}x_j}{(x_0^2+|x|^2)^{|\kappa|+(d+1)/2}},\qquad (x_0,x)\in\RR^{1+d}_+.
\end{align*}
The function $(\tau_{x}Q^{(j)}_{x_0})(-t)$ is called the $j$th conjugate $\kappa$-Poisson kernel.

\begin{proposition}\label{conjugate-kernel-a}
{\rm (i)} If we set $u_0(x_0,x)=P_{x_0}(x), u_1(x_0,x)=Q^{(1)}_{x_0}(x),\cdots,u_d(x_0,x)=Q^{(d)}_{x_0}(x)$ for $(x_0,x)\in\RR^{1+d}_+$, then they
satisfies the generalized Cauchy-Riemann equations (the M. Riesz system)
\begin{align}\label{C-R-1}
\left\{\begin{array}{l}
D_ju_k=D_ku_j,\qquad 0\le j<k\le d,\\
\sum_{j=0}^dD_ju_j=0,
\end{array}\right.
\end{align}
where $D_0=\partial_0$.

{\rm (ii)} For $x_0>0$ and $j=1,\cdots,d$, $(\SF_{\kappa}Q^{(j)}_{x_0})(\xi)=-i(\xi_j/|\xi|)e^{-x_0|\xi|}$.

{\rm (iii)} For $j=1,\cdots,d$, and for $(x_0,x)\in\RR^{1+d}_+$, the function $t\mapsto(\tau_{x}Q^{(j)}_{x_0})(-t)$ is in $L^q_{\kappa}(\RR^d)$ for $1<q\le\infty$.

{\rm (iv)} For $j=1,\cdots,d$, and for $(x_0,x)\in\RR^{1+d}_+$ and $t\in\RR^d$,
\begin{align}\label{conjugate-kernel-1}
(\tau_{x}Q^{(j)}_{x_0})(-t)=\frac{x_j-t_j}{x_0}(\tau_{x}P_{x_0})(-t)
\end{align}
and
\begin{align}\label{conjugate-kernel-2}
(\tau_{x}Q^{(j)}_{x_0})(-t)=\frac{c_{d,\kappa}}{2|\kappa|+d-1}\int_{\RR^d}D_j^t\left[\frac{1}{(x_0^2+A(x,t;\xi)^2)^{|\kappa|+(d-1)/2}}\right]\,d\mu^{\kappa}_{x}(\xi),
\end{align}
where
\begin{align*}
A(x,t;\xi)=\sqrt{|x|^2+|t|^2-2\langle t,\xi\rangle}.
\end{align*}
\end{proposition}

To show part (ii), define
\begin{align*}
\phi(x_0,x)=c_{\kappa}\int_{\RR^{d}}(-i\xi_j/|\xi|)e^{-x_0|\xi|}E_{\kappa}(i\xi,x)\,d\omega_{\kappa}(\xi),\qquad (x_0,x)\in\RR^{1+d}_+.
\end{align*}
The Lebesgue dominated convergence theorem asserts that $\lim_{x_0\rightarrow+\infty}\phi(x_0,x)=0$ uniformly for $x\in\RR^d$, and by use of (\ref{Dunkl-kernel-eigenfunction-1}),
\begin{align*}
\partial_{x_0}\phi(x_0,x) &=D_j^x\left(c_{\kappa}\int_{\RR^{d}}e^{-x_0|\xi|}E_{\kappa}(i\xi,x)\,d\omega_{\kappa}(\xi)\right)\\
&=D_j^x\left[P_{x_0}(x)\right]
=-\frac{c_{d,\kappa}(2|\kappa|+d+1)\,x_0x_j}{(x_0^2+|x|^2)^{|\kappa|+(d+3)/2}},\qquad (x_0,x)\in\RR^{1+d}_+.
\end{align*}
It follows that $\phi(x_0,x)=Q^{(j)}_{x_0}(x)$, and then part (ii) is concluded.

As for part (iv), we first prove (\ref{conjugate-kernel-2}). Since
\begin{align*}
(\tau_{x}Q^{(j)}_{x_0})(t)
=-\frac{c_{d,\kappa}}{2|\kappa|+d-1}\tau_x\left[D_j\left(\frac{1}{(x_0^2+|\cdot|^2)^{|\kappa|+(d-1)/2}}\right)\right](t),
\end{align*}
by Proposition \ref{translation-2-b}(i) and (\ref{translation-2-4}) we have
\begin{align*}
(\tau_{x}Q^{(j)}_{x_0})(t)
&=-\frac{c_{d,\kappa}}{2|\kappa|+d-1}D_j^t\left[\tau_x\left(\frac{1}{(x_0^2+|\cdot|^2)^{|\kappa|+(d-1)/2}}\right)\right](t)\\
&=-\frac{c_{d,\kappa}}{2|\kappa|+d-1}\int_{\RR^d}D_j^t\left[\frac{1}{(x_0^2+|x|^2+|t|^2+2\langle t,\xi\rangle)^{|\kappa|+(d-1)/2}}\right]\,d\mu^{\kappa}_{x}(\xi).
\end{align*}
From this the formula (\ref{conjugate-kernel-2}) follows.

To show (\ref{conjugate-kernel-1}), we use the subordination formula for the $\kappa$-Poisson kernel (see \cite[(5.1)]{ADH1}, for example) $(\tau_{x}P_{x_0})(-t)=\int_0^{\infty}e^{-u}h_{x_0^2/4u}(x,t)\,du/\sqrt{\pi u}$, or its equivalent form,
\begin{align}\label{Poisson-kernel-2}
(\tau_{x}P_{x_0})(-t)=\frac{x_0}{2\sqrt{\pi}}\int_0^{\infty}e^{-\frac{x_0^2}{4v}}\,h_v(x,t)\,\frac{dv}{v^{3/2}},
\end{align}
where $h_v(x,t)$ is the heat kernel associated to the Dunkl operators given by (cf. \cite{Ro2})
\begin{align}\label{heat-kernel-1}
h_v(x,t)=\frac{1}{(2v)^{|\kappa|+d/2}}\exp\left(-\frac{|x|^2+|t|^2}{4v}\right)
E_{\kappa}\left(\frac{x}{\sqrt{2v}},\frac{t}{\sqrt{2v}}\right),\quad x,t\in\RR^d,\,\,v>0.
\end{align}
By Propositions \ref{translation-2-b}(i), \ref{translation-2-b-1}(ii) and \ref{conjugate-kernel-a}(i),
$$
\partial_{x_0}\left[(\tau_{x}Q^{(j)}_{x_0})(-t)\right]=-D_j^t\left[(\tau_{x}P_{x_0})(-t)\right],
$$
and since, in view of (\ref{conjugate-kernel-2}), $\lim_{x_0\rightarrow+\infty}(\tau_{x}Q^{(j)}_{x_0})(-t)=0$ uniformly for $x,t\in\RR^d$, it follows that
\begin{align}\label{conjugate-kernel-3}
(\tau_{x}Q^{(j)}_{x_0})(-t)=\int_{x_0}^{\infty}D_j^t\left[(\tau_{x}P_{\rho})(-t)\right]d\rho.
\end{align}
Now applying (\ref{Dunkl-kernel-eigenfunction-1}), direction calculations yield
\begin{align}\label{heat-kernel-2}
D_j^th_v(x,t)=\frac{x_j-t_j}{2v}h_v(x,t),
\end{align}
and then from (\ref{Poisson-kernel-2}) and (\ref{conjugate-kernel-3}), Fubini's theorem gives
\begin{align}\label{conjugate-kernel-4}
(\tau_{x}Q^{(j)}_{x_0})(-t)=\frac{x_j-t_j}{2\sqrt{\pi}}\int_0^{\infty}e^{-\frac{x_0^2}{4v}}\,h_v(x,t)\,\frac{dv}{v^{3/2}},
\end{align}
Thus the formula (\ref{conjugate-kernel-1}) is proved.

\begin{proposition}\label{conjugate-Poisson-a}
{\rm (i)} If $f\in L^p_{\kappa}(\RR^d)$ for $1\le p<\infty$, then the functions
\begin{align*}
u_0(x_0,x)=(Pf)(x_0,x),\,\, u_1(x_0,x)=(Q_1f)(x_0,x),\,\, \cdots,\,\, u_d(x_0,x)=(Q_df)(x_0,x)
\end{align*}
satisfy the equations (\ref{C-R-1}) in $\RR^{1+d}_+$.

{\rm (ii)} If $f\in L^p_{\kappa}(\RR^d)$ for $1\le p\le2$, then for $j=1,\cdots,d$,
\begin{align*}
(Q_jf)(x_0,x)=c_{\kappa}\int_{\RR^{d}}\left(-i\frac{\xi_j}{|\xi|}\right)e^{-x_0|\xi|}\left(\SF_{\kappa}f\right)(\xi)E_{\kappa}(i\xi,x)\,d\omega_{\kappa}(\xi),\,\,\, (x_0,x)\in\RR^{1+d}_+.
\end{align*}
\end{proposition}

The following theorem extends the theorem of M. Riesz and that of Kolmogorov to the conjugate $\kappa$-Poisson integrals.

\begin{theorem} \label{conjugacy-thm-a}
{\rm (i)} There is $C_1>0$ such that $\int_{\{x:\,|(Q_jf)(x_0,x)|>s\}}d\omega_{\kappa}\le C_1\|f\|_{L^1_{\kappa}}/s$ holds for all $f \in
L^1_{\kappa}(\RR^d)$, $s>0$ and for all $x_0>0$, where $j=1,\cdots,d$.

{\rm (ii)} If $1 < p < \infty$, then there is $C_p>0$ such that
$\|(Q_jf)(x_0,\cdot)\|_{L^p_{\kappa}} \le C_p \|f\|_{L^p_{\kappa}}$ for all $f \in
L^p_{\kappa}(\RR^d)$ and $x_0>0$, where $j=1,\cdots,d$.
\end{theorem}

\begin{corollary} \label{Riesz-2-a}
If $1 < p < \infty$, then there exists a function in $L^p_{\kappa}(\RR^d)$, denoted by
${\mathcal{R}}_jf$, such that $(Q_jf)(x_0,\cdot)$ converges to ${\mathcal{R}}_jf$ as $x_0\rightarrow 0+$, both in the
$L^p_{\kappa}$-norm and almost everywhere non-tangentially on $\RR^d$, and
$\|{\mathcal{R}}_jf\|_{L^p_{\kappa}} \leq C_p \|f\|_{L^p_{\kappa}}$, where $j=1,\cdots,d$.
Moreover,
\begin{eqnarray}\label{conjugate-Poisson-4}
(Q_jf)(x_0,x)=[P({\mathcal{R}}_jf)](x_0,x),
\end{eqnarray}
and for $1<p\le2$,
\begin{align}\label{Riesz-2-1}
\left[\SF_{\kappa}({\mathcal{R}}_jf)\right](\xi)=-i(\xi_j/|\xi|)\left(\SF_{\kappa}f\right)(\xi).
\end{align}
\end{corollary}

\begin{corollary}\label{Riesz-2-b} If $f\in L_{\kappa}^p(\RR^d)$ for $1<p<\infty$, and $g\in L_{\kappa}^{p'}(\RR^d)$, then
\begin{align*}
\int_{\RR^d}({\mathcal{R}}_jf)(x)g(x)\,d\omega_{\kappa}=-\int_{\RR^d}f(x)({\mathcal{R}}_jg)(x)\,d\omega_{\kappa}.
 \end{align*}
\end{corollary}

From Proposition \ref{conjugate-kernel-a}(ii) and (\ref{Riesz-2-1}) it is obvious that
$$
{\mathcal{R}}_jP_{x_0}=Q^{(j)}_{x_0}.
$$

A stronger version of Theorem \ref{conjugacy-thm-a}, with respect to the associated maximal functions $(Q^*_{j,a}f)(x)=\sup_{(x_0,t)\in\Gamma_a(x)}|(Q_jf)(x_0,t)|$ for $a>0$ and $(Q_j^+f)(x)=\sup_{x_0>0}|(Q_jf)(x_0,x)|$, does hold too.

\begin{theorem} \label{conjugacy-maximal-a}
The operator $Q^*_{j,a}$ for $a>0$ is bounded from $L_{\kappa}^p(\RR^d)$ ($1<p<\infty$) to itself, and the operator $Q_j^+$ is bounded from $L_{\kappa}^1(\RR^d)$ to $L_{\kappa}^{1,\infty}(\RR^d)$.
\end{theorem}

Now we turn to a truncated version of the $\kappa$-Riesz transforms ${\mathcal{R}}_j$ ($1\le j\le d$).
At first, from (\ref{translation-2-3}), (\ref{translation-support-2-1}) and (\ref{translation-2-4}) it follows that, for $\xi\in\hbox{supp}\,\mu^{\kappa}_{x}$,
\begin{align*}
\min_{\sigma\in G}|t-\sigma(x)|\le A(x,t;\xi)\le\max_{\sigma\in G}|t-\sigma(x)|.
\end{align*}
(cf. \cite[(3.9)]{AS1}. Thus for $j=1,\cdots,d$, and for $x,t\in\RR^d$ satisfying $\min_{\sigma\in G}|t-\sigma(x)|>0$, the function
\begin{align*}
K_j(x,t)=\frac{c_{d,\kappa}}{2|\kappa|+d-1}\int_{\RR^d}D_j^t\left[\frac{1}{A(x,t;\xi)^{2|\kappa|+d-1}}\right]\,d\mu^{\kappa}_{x}(\xi)
\end{align*}
is well defined, and from (\ref{conjugate-kernel-2}) and (\ref{conjugate-kernel-4}), it can be rewritten as
\begin{align}\label{truncated-Riesz-kernel-2}
K_j(x,t)=\frac{x_j-t_j}{2\sqrt{\pi}}\int_0^{\infty}h_v(x,t)\,\frac{dv}{v^{3/2}}.
\end{align}
For $f\in L^p_{\kappa}(\RR^d)$ ($1\le p<\infty$) and $\epsilon>0$, we consider
\begin{align*}
({\mathcal{R}}_j^{\epsilon}f)(x)=c_{\kappa}\int_{d(x,t)>\epsilon}f(t)K_j(x,t)\,d\omega_{\kappa}(t),\qquad x\in\RR^d,
\end{align*}
and call ${\mathcal{R}}_j^{\epsilon}$ the $j$th truncated $\kappa$-Riesz transform ($1\le j\le d$). Here
$$
d(x,t):=\min_{\sigma\in G}|t-\sigma(x)|
$$
denotes the distance between the $G$-orbits of $x,t\in\RR^d$.

\begin{theorem} \label{truncated-Riesz-a}
{\rm (i)} For $1\le p<\infty$, there is $C_p>0$ such that for $f\in L^p_{\kappa}(\RR^d)$ and $\epsilon>0$,
$\|(Q_jf)(\epsilon,\cdot)-{\mathcal{R}}_j^{\epsilon}f\|_{L^p_{\kappa}}\le C_p\|f\|_{L^p_{\kappa}}$, where $j=1,\cdots,d$; and moreover,
\begin{align*}
\lim_{\epsilon\rightarrow0+}\|(Q_jf)(\epsilon,\cdot)-{\mathcal{R}}_j^{\epsilon}f\|_{L^p_{\kappa}}=0.
\end{align*}

{\rm (ii)} For $1\le p<\infty$, there is $C_p'>0$ such that for $f\in L^p_{\kappa}(\RR^d)$,
\begin{align*}
\sup_{\epsilon>0}\left|(Q_jf)(\epsilon,x)-({\mathcal{R}}_j^{\epsilon}f)(x)\right|
\le C_p'\left[P^+(|f|)\right](x),\qquad x\in\RR^d,
\end{align*}
where $j=1,\cdots,d$; and moreover, $\lim\limits_{\epsilon\rightarrow0+}[(Q_jf)(\epsilon,x)-({\mathcal{R}}_j^{\epsilon}f)(x)]=0$ for almost every $x\in\RR^d$.
\end{theorem}

\begin{corollary} \label{Riesz-2-c}
If $1<p<\infty$, then for $f\in L^p_{\kappa}(\RR^d)$, ${\mathcal{R}}_j^{\epsilon}f$ converges to ${\mathcal{R}}_jf$ as $\epsilon\rightarrow 0+$, both in the
$L^p_{\kappa}$-norm and almost everywhere on $\RR^d$.
\end{corollary}

\begin{corollary} \label{Riesz-2-d}
The maximal $\kappa$-Riesz transform ${\mathcal{R}}_j^*$ for $j=1,\cdots,d$, defined by $({\mathcal{R}}_j^*f)(x)=\sup_{\epsilon>0}|({\mathcal{R}}_j^{\epsilon}f)(x)|$, is bounded from $L_{\kappa}^p(\RR^d)$ ($1<p<\infty$) to itself, and from $L_{\kappa}^1(\RR^d)$ to $L_{\kappa}^{1,\infty}(\RR^d)$.
\end{corollary}



Finally we formulate the $\kappa$-Riesz transforms of functions in $L_{\kappa}^1(\RR^d)$, in the spirit of \cite[p. 221]{St2}, as follows.

\begin{definition}\label{Riesz-2-e}
For $f\in L_{\kappa}^1(\RR^d)$, we say that it has the $j$th $\kappa$-Riesz transform in $L_{\kappa}^1(\RR^d)$, if there exists an $f_j\in L_{\kappa}^1(\RR^d)$ so that
\begin{align}\label{Riesz-2-2}
\left(\SF_{\kappa}f_j\right)(\xi)=-i(\xi_j/|\xi|)\left(\SF_{\kappa}f\right)(\xi),\qquad \xi\in\RR^d\setminus\{0\},
\end{align}
and certainly write ${\mathcal{R}}_jf=f_j$.
\end{definition}

\begin{proposition}\label{Riesz-2-f} If $f\in L_{\kappa}^1(\RR^d)$ and ${\mathcal{R}}_jf\in L_{\kappa}^1(\RR^d)$, then (\ref{conjugate-Poisson-4}) holds
and $(Q_jf)(x_0,\cdot)$ converges to ${\mathcal{R}}_jf$ as $x_0\rightarrow 0+$, both in the
$L^1_{\kappa}$-norm and almost everywhere non-tangentially on $\RR^d$.
\end{proposition}

\subsection{The Hardy spaces $\SH_{\kappa}^1(\RR^{1+d}_+)$ and $H_{\kappa}^1(\RR^d)$}

We consider the vector-valued function $F=(u_0,u_1,\cdots,u_d)$ (a function system of Stein-Weiss type) defined on the upper half-space $\RR^{1+d}_+$, that satisfies the generalized Cauchy-Riemann equations (\ref{C-R-1}).
If the components of such a vector valued function $F$ are all twice continuously differentiable in $\RR^{1+d}$, i.e. $F\in C^2(\RR^{1+d}_+)$ says, then they are all $\tilde{\kappa}$-harmonic because of commutativity of $D_j$, $j=0,1,\cdots,d$.
This further implies $F\in C^{\infty}(\RR^{1+d}_+)$ by Proposition \ref{harmonic-2-a-1}.

The ``analytic" Hardy space on $\RR^{1+d}_+$ in the Dunkl setting, denoted by $\SH_{\kappa}^1(\RR^{1+d}_+)$, consists of all vector-valued functions $F\in C^2(\RR^{1+d}_+)$ satisfying the equations in (\ref{C-R-1}) and the condition
\begin{align*}
\|F\|_{\SH_{\kappa}^1}:=\sup_{x_0>0}\||F(x_0,\cdot)|\|_{L^{1}_\kappa}<\infty,
\end{align*}
where $|F(x_0,x)|=\left(\sum_{j=0}^d|u_j(x_0,x)|^2\right)^{1/2}$.

The following theorem contains the fundamental results about the Hardy space $\SH_{\kappa}^1(\RR^{1+d}_+)$.

\begin{theorem} \label{Hardy-thm-1}
Assume that $F\in \SH_{\kappa}^1(\RR^{1+d}_+)$. Then

{\rm (i)} For almost every $x\in\RR^d$, $\lim F(x_0,t)=\tilde{F}(x)$ exists as $(x_0,t)$ approaches to the point
$(0,x)$ nontangentially;

{\rm (ii)} As $x_0\rightarrow0+$, $F(x_0,\cdot)$ converges to $\tilde{F}$ in the $L_{\kappa}^1$ norm; and $\|F\|_{\SH_{\kappa}^1}=\|\tilde{F}\|_{L_{\kappa}^1}$;

{\rm (iii)} $F(x_0,x)$ is the $\kappa$-Poisson integral of $\tilde{F}$ in vector form;

{\rm (iv)} For all $a>0$, $\|F^*_a\|_{L_{\kappa}^1}\lesssim\|F\|_{\SH_{\kappa}^1}$, where $F^*_a(x)=\sup_{(x_0,t)\in\Gamma_a(x)}|F(x_0,t)|$.
\end{theorem}

The above theorem has partly been proved in \cite{ADH1}. Indeed, on account of Proposition \ref{Poisson-Integral-Max-N}, part (iv) is essentially the same as \cite[Proposition 7.6]{ADH1}, which certainly implies that each component of $F$ is non-tangentially bounded at $(0,x)$ for almost every $x\in\RR^d$. Consequently part (i) follows from Theorem \ref{local-Fatou-1}. Furthermore, since $|F(x_0,x)-\tilde{F}(x)|\le2F^*_1(x)\in L_{\kappa}^1(\RR^d)$, the Lebesgue dominated convergence theorem gives that $\|F(x_0,\cdot)-\tilde{F}\|_{L_{\kappa}^1}\rightarrow0$ as $x_0\rightarrow0+$, and also $\|\tilde{F}\|_{L_{\kappa}^1}\le\|F\|_{\SH_{\kappa}^1}$. On the other hand, from the proof of \cite[Proposition 7.6]{ADH1} one has, in vector form,
$$
F(x_0+\epsilon,x)=\left[P(F(\epsilon,\cdot))\right](x_0,x),
$$
and letting $\epsilon\rightarrow0+$ yields $F(x_0,x)=(P\tilde{F})(x_0,x)$. Finally since
\begin{align}\label{Hardy-Poisson-1}
|F(x_0,x)|\le[P(|\tilde{F}|)](x_0,x),
\end{align}
by Proposition \ref{Poisson-c} we get $\|F\|_{\SH_{\kappa}^1}\le\|\tilde{F}\|_{L_{\kappa}^1}$. Thus both parts (ii) and (iii) are proved.

\begin{proposition}
For $F=(u_0,u_1,\cdots,u_d)\in \SH_{\kappa}^1(\RR^{1+d}_+)$, the functions $u_j$, $j=1,\cdots,d$, are uniquely determined by the first component $u_0$.
\end{proposition}

\begin{proof}
Since $0\le P_{x_0}(x)\le x_0^{-2|\kappa|-d}$ for $x\in\RR^d$ so that $0\le(\tau_{x}P_{x_0})(-t)\le x_0^{-2|\kappa|-d}$ for $x,t\in\RR^d$ and $x_0>0$,
from (\ref{Poisson-2-2}) and (\ref{Hardy-Poisson-1}) one has $|F(x_0,x)|\le x_0^{-2|\kappa|-d}\|F\|_{\SH_{\kappa}^1}$ for $(x_0,x)\in\RR^{1+d}_+$. If the vector-valued function $G=(u_0,v_1,\cdots,v_d)$ is also in $\SH_{\kappa}^1(\RR^{1+d}_+)$, the generalized Cauchy-Riemann equations in (\ref{C-R-1}) imply $\partial_0(u_j-v_j)\equiv0$ for $j=1,\cdots,d$, and so $u_j-v_j=:\varphi_j(x)$, independent of $x_0$. But by the inequality just proved, $|\varphi_j(x)|\le x_0^{-2|\kappa|-d}(\|F\|_{\SH_{\kappa}^1}+\|G\|_{\SH_{\kappa}^1})$, and then letting $x_0\rightarrow+\infty$ proves that $u_j-v_j=\varphi_j(x)\equiv0$ on $\RR^{1+d}_+$.
\end{proof}

We define the ``real" Hardy space in the Dunkl setting by
\begin{align*}
H_{\kappa}^1(\RR^d)=\left\{f\in L_{\kappa}^1(\RR^d):\quad {\mathcal{R}}_jf\in L_{\kappa}^1(\RR^d),\,\,\,j=1,\cdots,d\right\},
\end{align*}
endowed with the norm $\|f\|_{H_{\kappa}^1}=\|f\|_{L_{\kappa}^1}+\sum_{j=1}^d\|{\mathcal{R}}_jf\|_{L_{\kappa}^1}$.

\begin{proposition}\label{completeness-a}
The space $H_{\kappa}^1(\RR^d)$ is a Banach space.
\end{proposition}

\begin{proposition}\label{Hardy-integral-a}
If $f\in H_{\kappa}^1(\RR^d)$, then $\int_{\RR^{d}}f\,d\omega_{\kappa}=0$.
\end{proposition}

Since, for $f\in H_{\kappa}^1(\RR^d)$, $\SF_{\kappa}f$ and $\SF_{\kappa}({\mathcal{R}}_jf)$ are all continuous at the origin, the proposition follows immediately from (\ref{Riesz-2-2}) with $f_j={\mathcal{R}}_jf$.

\begin{theorem} \label{Hardy-thm-2}
The function $f$ defined on $\RR^d$ is in $H_{\kappa}^1(\RR^d)$ if and only if there exists a vector-valued function $F=(u_0,u_1,\cdots,u_d)\in \SH_{\kappa}^1(\RR^{1+d}_+)$ with $f=\lim u_0$, the boundary function of the first component of $F$ given in Theorem \ref{Hardy-thm-1}. Moreover, $\|f\|_{H_{\kappa}^1}\asymp\|F\|_{\SH_{\kappa}^1}$, and for $j=1,\cdots,d$, the boundary function $f_j$ of $u_j$ is the $j$th $\kappa$-Riesz transform ${\mathcal{R}}_jf$ of $f$.
\end{theorem}

This theorem was proved in \cite{ADH1} for the $\kappa$-Riesz transforms defined in the sense of distributions. Here we restate the proof based on Definition \ref{Riesz-2-e}.
For $f\in H_{\kappa}^1(\RR^d)$, set $u_0=Pf$ and $u_j=P({\mathcal{R}}_jf)$ for $1\le j\le d$. Using the expressions of $u_0$ and $u_j$ ($1\le j\le d$) in the form of (\ref{Poisson-2-3}) and noting (\ref{Riesz-2-2}) with $f_j={\mathcal{R}}_jf$, it is easy to check that the vector-valued function $F=(u_0,u_1,\cdots,u_d)$ satisfies the equations in (\ref{C-R-1}); and furthermore, by Proposition \ref{Poisson-c}, $\|F\|_{\SH_{\kappa}^1}\le\|f\|_{H_{\kappa}^1}$. Therefore $F\in \SH_{\kappa}^1(\RR^{1+d}_+)$ with $f=\lim u_0$ by Corollary \ref{Poisson-d}.

Conversely suppose $F=(u_0,u_1,\cdots,u_d)\in \SH_{\kappa}^1(\RR^{1+d}_+)$ with the boundary function $\tilde{F}=(f_0,f_1,\cdots,f_d)$ determined in Theorem \ref{Hardy-thm-1}. It suffices to show (\ref{Riesz-2-2}) for $j=1,\cdots,d$.
By Theorem \ref{Hardy-thm-1}(iii), $u_0=Pf_0$ and $u_j=Pf_j$, and the generalized Cauchy-Riemann equations in (\ref{C-R-1}) and the expressions of $u_0$ and $u_j$ in the form of (\ref{Poisson-2-3}) give, for $x_0>0$,
\begin{align*}
0&\equiv D_ju_0-D_0u_j\\
&=c_{\kappa}\int_{\RR^{d}}e^{-x_0|\xi|}\left[i\xi_j\left(\SF_{\kappa}f_0\right)(\xi)+|\xi|\left(\SF_{\kappa}f_j\right)(\xi)\right]E_{\kappa}(i\xi,x)\,d\omega_{\kappa}(\xi),\quad x\in\RR^d.
\end{align*}
It then follows from Proposition \ref{transform-a}(ii) that $i\xi_j\left(\SF_{\kappa}f_0\right)(\xi)+|\xi|\left(\SF_{\kappa}f_j\right)(\xi)=0$ for a.e. $x\in\RR^d$, which is identical with (\ref{Riesz-2-2}) as desired. Thus we have proved $f_0\in H_{\kappa}^1(\RR^d)$; and finally, by Theorem \ref{Hardy-thm-1}(ii), $\|f_0\|_{H_{\kappa}^1}\lesssim\|\tilde{F}\|_{L_{\kappa}^1}=\|F\|_{\SH_{\kappa}^1}$. The proof of Theorem \ref{Hardy-thm-2} is completed.

\subsection{A dense subclass of the Hardy space $H_{\kappa}^1(\RR^d)$}

We shall prove a density result on $H_{\kappa}^1(\RR^d)$ that is an analog of \cite[p. 225, Lemma]{St2}). Set
\begin{align*}
L_{\kappa,0}^1(\RR^d)
&=\left\{f\in L_{\kappa}^1(\RR^d): \,\,(\SF_{\kappa}f)(0)=0\right\},\\
L_{\kappa,00}^1(\RR^d)
&=\left\{f\in L_{\kappa}^1(\RR^d): \,\, \SF_{\kappa}f \,\, \hbox{has compact support in $\RR^d\setminus\{0\}$}\right\}.
\end{align*}

\begin{lemma} \label{density-a}
$H_{\kappa,0}^1(\RR^d):=\SS(\RR^d)\bigcap L_{\kappa,00}^1(\RR^d)$ is s dense subclass of $H_{\kappa}^1(\RR^d)$.
\end{lemma}

\begin{proof}
Note that $H_{\kappa}^1(\RR^d)\subseteq L_{\kappa,0}^1(\RR^d)$ by Proposition \ref{Hardy-integral-a}. Choose a fixed radial $\eta\in \SS(\RR^d)$ with the real Dunkl transform $\SF_\kappa \eta$ satisfying $\chi_{\{|\xi|\leq1\}}\leq\SF_\kappa\eta\leq \chi_{\{|\xi|\leq2\}}$, and write $\eta_\epsilon(x)=\epsilon^{-2|\kappa|-d}\eta(\epsilon^{-1}x)$ for $\epsilon>0$.
We shall prove the following two claims:

{\bf Claim 1.} $\SS(\RR^d)\bigcap L_{\kappa,00}^1(\RR^d)$ is dense in $L_{\kappa,00}^1(\RR^d)$ as a subspace of $H_{\kappa}^1(\RR^d)$.

{\bf Claim 2.} $L_{\kappa,00}^1(\RR^d)$ is dense in $H_{\kappa}^1(\RR^d)$.

To show Claim 1, define the operators $T_\epsilon$ for $\epsilon>0$ by
$\SF_\kappa(\widetilde{T_\epsilon}f)=\eta(0)^{-1}(\SF_\kappa f)*_\kappa(\SF_\kappa \eta)_\epsilon$.
For $f\in L_{\kappa,00}^1(\RR^d)$, by Proposition \ref{convolution-2-c} we have that $\SF_\kappa(\widetilde{T_\epsilon}f)\in C^\infty(\RR^d)$ and the supports of $\SF_{\kappa}f$ and $\SF_\kappa(\widetilde{T_\epsilon}f)$ for those $\epsilon$ sufficiently small are contained in a common compact set $K$ which is at a positive distance from the origin. This asserts that
$\widetilde{T_\epsilon}f\in \SS(\RR^d)\bigcap L_{\kappa,00}^1(\RR^d)$.
In addition Propositions \ref{transform-a}(ii) and \ref{convolution-2-a}(v) imply $(\widetilde{T_\epsilon}f)(x)=f(x)\eta(\epsilon x)/\eta(0)$, so that
\begin{align*}
\|\widetilde{T_\epsilon}f-f\|_{L^{1}_\kappa}
=c_{\kappa}\int_{\RR^{d}}|f(x)|\left|\eta(\epsilon x)/\eta(0)-1\right|\,d\omega_{\kappa}\rightarrow0, \qquad \hbox{as}\quad \epsilon\rightarrow0+,
\end{align*}
by Lebesgue's dominated convergence theorem.

Now for $j=1,\cdots,d$, we select the functions $m_j(\xi)\in\SS(\RR^d)$ satisfying $m_j(\xi)=-i\xi_j/|\xi|$ when $\xi\in K$, and put $\psi_j=\SF_{\kappa}^{-1}m_j$. By Proposition \ref{convolution-2-a}(iv), $\|f\ast_{\kappa}\psi_j\|_{L_{\kappa}^1}\le C_{\psi_j}\|f\|_{L_{\kappa}^1}$ and $\|(\widetilde{T_\epsilon}f)\ast_{\kappa}\psi_j\|_{L_{\kappa}^1}\le C_{\psi_j}\|\widetilde{T_\epsilon}f\|_{L_{\kappa}^1}$, and then, according to Proposition \ref{convolution-2-a}(v) and Definition \ref{Riesz-2-e}, ${\mathcal{R}}_jf=f\ast_{\kappa}\psi_j$ and ${\mathcal{R}}_j(\widetilde{T_\epsilon}f)=\widetilde{T_\epsilon}f\ast_{\kappa}\psi_j$. Thus both $f$ and $\widetilde{T_\epsilon}f$ belong to $H_{\kappa}^1(\RR^d)$,
so that $L_{\kappa,00}^1(\RR^d)\subseteq H_{\kappa}^1(\RR^d)$, and again by Proposition \ref{convolution-2-a}(iv),
$$
\|{\mathcal{R}}_j(\widetilde{T_\epsilon}f-f)\|_{L_{\kappa}^1}=\|(\widetilde{T_\epsilon}f-f)\ast\psi_j\|_{L_{\kappa}^1}\leq C_{\psi_j}\|\widetilde{T_\epsilon}f-f\|_{L_{\kappa}^1}\rightarrow0,  \qquad \hbox{as}\quad \epsilon\rightarrow0+.
$$
Therefore $\lim_{\epsilon\rightarrow0+}\|\widetilde{T_\epsilon}f-f\|_{H_{\kappa}^1}=0$, and Claim 1 is concluded,

In order to prove Claim 2, define the operator $T_\epsilon$ by $(T_\epsilon f)(x)=(f\ast_\kappa\eta_\epsilon)(x)$, and for $f\in L_{\kappa}^1(\RR^d)$, consider $T_\epsilon f-T_{\epsilon^{-1}}f$ with $0<\epsilon<1$.
By Proposition \ref{convolution-2-a}(iii), $\|T_\epsilon f-T_{\epsilon^{-1}}f\|_{L_{\kappa}^1}\le2\|\eta\|_{L_{\kappa}^1}\|f\|_{L_{\kappa}^1}$, and by Proposition \ref{convolution-2-a}(v),
$$
[\SF_\kappa(T_\epsilon f-T_{\epsilon^{-1}}f)](\xi)=(\SF_\kappa f)(\xi)[(\SF_\kappa \eta)(\epsilon\xi)-(\SF_\kappa \eta)(\epsilon^{-1}\xi)]=0
$$
whenever $|\xi|<\epsilon$ or $|\xi|>2\epsilon^{-1}$. Consequently $T_\epsilon f-T_{\epsilon^{-1}}f\in L_{\kappa,00}^1(\RR^d)\subseteq H_{\kappa}^1(\RR^d)$.

We assert that for $f\in L_{\kappa,0}^1(\RR^d)$, $\lim_{\epsilon\rightarrow\infty}\|T_\epsilon f\|_{L^{1}_\kappa}=0$. Since
$$
(T_\epsilon f)(x)=c_{\kappa}\int_{\RR^d}f(-t)[(\tau_t\eta_\epsilon)(x)-\eta_\epsilon(x)]d\omega_\kappa(t),
$$
one has
$\|T_\epsilon f\|_{L^{1}_\kappa}\leq c_{\kappa}\int_{\RR^d}|f(-t)|\|\tau_{t}\eta_\epsilon-\eta_\epsilon\|_{L^{1}_\kappa}d\omega_\kappa(t)$.
But by Proposition \ref{translation-2-b-1}(ii), $(\tau_{t}\eta_\epsilon)(x)=\left(\tau_{\epsilon^{-1}t}\eta\right)_{\epsilon}(x)$, so that for $t\in\RR^d$, $\|\tau_{t}\eta_\epsilon-\eta_\epsilon\|_{L^{1}_\kappa}=\|\tau_{\epsilon^{-1}t}\eta-\eta\|_{L^{1}_\kappa}\rightarrow0$ as $\epsilon\rightarrow\infty$ by Corollary \ref{translation-2-f};
and then, on account of $\|\tau_{t}\eta_\epsilon-\eta_\epsilon\|_{L^{1}_\kappa}\le 2\|\eta\|_{L^{1}_\kappa}$ by Proposition \ref{translation-2-e}, applying Lebesgue's dominated convergence theorem yields the desired assertion.
Moreover $\lim_{\epsilon\rightarrow0+}\|T_\epsilon f-f\|_{L^{1}_\kappa}=0$ by Proposition \ref{convolution-2-b}(ii), and hence, for $f\in L_{\kappa,0}^1(\RR^d)$,
$$
\|(T_\epsilon f-T_{\epsilon^{-1}}f)-f\|_{L^{1}_\kappa}\leq \|T_\epsilon f-f\|_{L^{1}_\kappa}+\|T_{\epsilon^{-1}}f\|_{L^{1}_\kappa}\rightarrow0 \qquad \hbox{as}\quad \epsilon\rightarrow0+.
$$

Finally assume that $f\in H_{\kappa}^1(\RR^d)$. Since ${\mathcal{R}}_jf\in L_{\kappa,0}^1(\RR^d)$ for $1\le j\le d$ and $T_\epsilon f-T_{\epsilon^{-1}}f\in H_{\kappa}^1(\RR^d)$, by taking the Dunkl transform we get
${\mathcal{R}}_j(T_\epsilon f-T_{\epsilon^{-1}}f)=(T_\epsilon-T_{\epsilon^{-1}})({\mathcal{R}}_j f)$.
Thus
\begin{align*}
\|{\mathcal{R}}_j(T_\epsilon f-T_{\epsilon^{-1}}f)-{\mathcal{R}}_jf\|_{L_{\kappa}^1}
=\|(T_\epsilon-T_{\epsilon^{-1}})({\mathcal{R}}_jf)-{\mathcal{R}}_jf\|_{L_{\kappa}^1}\rightarrow0 \qquad \hbox{as}\quad \epsilon\rightarrow0+,
\end{align*}
and consequently,
$\|(T_\epsilon f-T_{\epsilon^{-1}}f)-f\|_{H_{\kappa}^1}\rightarrow0$ as $\epsilon\rightarrow0^+$. Claim 2 is proved and so is the lemma.
\end{proof}

From the proof of the above lemma, we have the following corollary.
\begin{corollary} \label{density-b}
If $f\in H_{\kappa,0}^1(\RR^d)$, then ${\mathcal{R}}_jf\in H_{\kappa,0}^1(\RR^d)$.
\end{corollary}

Let $\SH_{\kappa,0}^1(\RR^{1+d}_+)$ be the collection of the vector-valued functions of $\SH_{\kappa}^1(\RR^{1+d}_+)$ whose first components have boundary values in $H_{\kappa,0}^1(\RR^d)=\SS(\RR^d)\bigcap L_{\kappa,00}^1(\RR^d)$.

\begin{corollary} \label{den-aug-2}
The subclass $\SH_{\kappa,0}^1(\RR^{1+d}_+)$ is dense in the Hardy space $\SH_{\kappa}^1(\RR^{1+d}_+)$ and has the following properties: $F$ in $\SH_{\kappa,0}^1(\RR^{1+d}_+)$ and its every order partial derivative $\partial^{\gamma}F$ ($\gamma$ is a $(d+1)$-tuple of non-negative integers) in the $x_0,x_1,\cdots,x_d$ are continuous on  $\overline{\RR^{1+d}_+}$; and they are all fast decreasing, that means, for each polynomial $p$ in the $x_0,x_1,\cdots,x_d$, $p$ times $F$ and each $\partial^{\gamma}F$ are bounded on $\overline{\RR^{1+d}_+}$.
\end{corollary}

\section{The space $\B_{\kappa}(\RR^d)$ and the $\kappa$-Riesz transforms}

The $\B$ type space to be considered is the class of those $f\in L_{\kappa,{\rm loc}}(\RR^d)$ satisfying
\begin{eqnarray*}
\|f\|_{*,\kappa}:=\sup_{B}\frac{1}{|B|_\kappa}\int_{B}|f-f_{B}|\,d\omega_{\kappa}<\infty,
\end{eqnarray*}
where the supremum ranges over all finite balls $B$ in $\RR^d$, and $f_B=|B|_{\kappa}^{-1}\int_{B}f\,d\omega_\kappa$. Such a class is denoted by $\B_{\kappa}(\RR^d)$.
As usual, $f\in \B_\kappa(\RR^d)$ represents an element of the quotient space of $\B_\kappa(\RR^d)$ modulo constants.

\begin{proposition}\label{bmo-a}
There exists some $C>0$ such that for all $f\in\B_\kappa(\RR^d)$ and $x'\in\RR^d$,
\begin{eqnarray}\label{bmo-2}
\int_{\RR^d}\frac{|f(x)-f_{B_1(x')}|}{(1+|x-x'|)|B_{1+|x-x'|}(x')|_\kappa}
d\omega_\kappa(x)\leq C\|f\|_{*,\kappa},
\end{eqnarray}
where $B_{\delta}(x')=B(x',\delta)$, the ball with center at $x'$ and radius $\delta>0$. Moreover for $\delta>0$,
\begin{eqnarray}\label{bmo-3}
\int_{\RR^d}\frac{|f(x)-f_{B_{\delta}(x')}|}{(\delta+|x-x'|)|B_{\delta+|x-x'|}(x')|_\kappa}
\,d\omega_\kappa(x)
\leq\frac{C}{\delta}\|f\|_{*,\kappa}.
\end{eqnarray}
\end{proposition}

Observe that (\ref{bmo-3}) is a consequence of (\ref{bmo-2}) by dilations, while some estimates similar to (\ref{bmo-2}) are known (see \cite{DY1}, for example).
In our style, since
 \begin{align}\label{cube-estimate-1}
\left|B_{\delta}(x')\right|_{\kappa}\asymp \delta^d\prod_{\alpha\in R}(|\langle\alpha,x'\rangle|+\delta)^{\kappa(\alpha)},
 \end{align}
which certainly implies the doubling property of the measure $\omega_k$, it is easy to find that $|f_{B_{2^{\ell}}(x')}-f_{B_1(x')}|\lesssim \ell\|f\|_{*,\kappa}$ for $\ell=1,2,\cdots$, and hence
$$
\frac{1}{|B_{2^{\ell}}(x')|_\kappa}\int_{B_{2^{\ell}}(x')}|f-f_{B_1(x')}|\,d\omega_{\kappa}
\lesssim (1+\ell)\|f\|_{*,\kappa}.
$$
Thus, with $B_{\delta}=B_{\delta}(x')$, the left hand side of (\ref{bmo-2}) is dominated by a multiple of
\begin{align*}
\frac{1}{|B_1|_\kappa}\int_{B_1}|f-f_{B_1}|d\omega_\kappa
+\sum_{\ell=1}^\infty\frac{1}{2^{\ell}|B_{2^{\ell}}|_\kappa}\int_{B_{2^{\ell}}\setminus B_{2^{\ell-1}}}|f-f_{B_1}|d\omega_\kappa
 \lesssim\|f\|_{*,\kappa},
\end{align*}
and then (\ref{bmo-2}) is proved.

The John-Nirenberg inequality stated as follows, is a special case of that in spaces of homogeneous type, which could be proved along the proof process of the classical one.

\begin{proposition}\label{J-N-a}
There exist constants $c_0,c_1>0$ such that for all $f\in\B_\kappa(\RR^d)$ and all finite balls $B$ in $\RR^d$,
\begin{eqnarray}\label{J-N-1}
\left|\left\{x\in B:|f(x)-f_{B}|>\lambda\right\}\right|_\kappa \leq c_1 e^{-c_0\lambda/\|f\|_{*,\kappa}}|B|_\kappa,\qquad \lambda>0.
 \end{eqnarray}
\end{proposition}

A direct consequence of (\ref{J-N-1}) is that, for $f\in\B_\kappa(\RR^d)$,
\begin{eqnarray}\label{bmo-4}
\|f\|_{*,\kappa}\asymp\sup_{B}\left(\frac{1}{|B|_\kappa}\int_{B}|f-f_B|^2\,d\omega_{\kappa}\right)^{1/2}.
\end{eqnarray}

For $j=1,\cdots,d$ and $\epsilon>0$, let
$$
K_j^{\epsilon}(x,t)=K_j(x,t)\chi_{d(x,t)>\epsilon}(x,t),
$$
where $\chi_{d(x,t)>\epsilon}$ denotes the indicator of the set $\{(x,t):\,d(x,t)>\epsilon\}$. We consider the variant of ${\mathcal{R}}_j^{\epsilon}$ given by
\begin{align}\label{truncated-Riesz-4}
(\tilde{\mathcal{R}}_j^{\epsilon}f)(x)=c_{\kappa}\int_{\RR^d}f(t)\left(K_j^{\epsilon}(x,t)-K_j^1(0,t)\right)\,d\omega_{\kappa}(t),\qquad x\in\RR^d.
\end{align}

\begin{proposition}\label{Riesz-BMO-b}
For $f\in L^{\infty}(\RR^d)$, the limit
 \begin{align}\label{Riesz-2-3}
(\tilde{\mathcal{R}}_jf)(x):=\lim_{\epsilon\rightarrow0+}(\tilde{\mathcal{R}}_j^{\epsilon}f)(x)
\end{align}
exists for almost every $x\in\RR^d$, and also in the $L_{\kappa}^2$ norm on each finite ball.
\end{proposition}

Indeed, write $f=f_1+f_2$ with $f_1=f\chi_{B(0,2M)}$ for $M>0$. Then for $x\in B(0,M)$ and $0<\epsilon<M$,  $(\tilde{\mathcal{R}}_j^{\epsilon}f)(x)=(\tilde{\mathcal{R}}_j^{\epsilon}f_1)(x)+(\tilde{\mathcal{R}}_j^{M}f_2)(x)$;
but since $f_1\in L_{\kappa}^2(\RR^d)$,
$$
\lim_{\epsilon\rightarrow0+}(\tilde{\mathcal{R}}_j^{\epsilon}f_1)(x)
=\lim_{\epsilon\rightarrow0+}({\mathcal{R}}_j^{\epsilon}f_1)(x)-c_{\kappa}\int_{\RR^d}f_1(t)K_j^1(0,t)\,d\omega_{\kappa}(t)
$$
exists for almost every $x\in B(0,M)$, and so does $\lim_{\epsilon\rightarrow0+}(\tilde{\mathcal{R}}_j^{\epsilon}f)(x)$. This certainly concludes (\ref{Riesz-2-3}) for almost every $x\in\RR^d$.
In addition, for $x\in B(0,M)$, $(\tilde{\mathcal{R}}_jf)(x)$ defined by (\ref{Riesz-2-3}) is
$$
({\mathcal{R}}_jf_1)(x)-c_{\kappa}\int_{\RR^d}f_1(t)K_j^1(0,t)\,d\omega_{\kappa}(t)
+(\tilde{\mathcal{R}}_j^{M}f_2)(x),
$$
so that, by Corollary \ref{Riesz-2-c},
$$
\int_{B(0,M)}|\tilde{\mathcal{R}}_j^{\epsilon}f-\tilde{\mathcal{R}}_jf|^2\,d\omega_{\kappa}
=\int_{B(0,M)}|{\mathcal{R}}_j^{\epsilon}f_1-{\mathcal{R}}_jf_1|^2\,d\omega_{\kappa}\rightarrow0\quad \hbox{as $\epsilon\rightarrow0+$.}
$$

The main conclusion in this section is stated in the following theorem.

\begin{theorem}\label{Riesz-BMO-a}
The $\kappa$-Riesz transform $\tilde{\mathcal{R}}_j$ defined by (\ref{truncated-Riesz-4}) and (\ref{Riesz-2-3}) is bounded from $L^\infty(\RR^d)$ to $\B_\kappa(\RR^d)$, where $j=1,\cdots,d$.
\end{theorem}

We need to make some preparations for the proof of the theorem. The estimates given in the following lemma are improvements of those in \cite{HLLW}.

\begin{lemma}\label{Riesz-kernel-a}
{\rm (i)} For $x,t\in\RR^d$ satisfying $d(x,t)>0$,
\begin{align}\label{Riesz-kernel-1}
\frac{c_1}{\left|B(x,|x-t|)\right|_{\kappa}}\frac{|x_j-t_j|}{|x-t|}
\leq |K_j(x,t)|| \leq
\frac{c_2|x_j-t_j|}{\left|B(x,d(x,t))\right|_{\kappa}}\frac{d(x,t)}{|x-t|^2},
\end{align}
where the constants $c_1,c_2>0$ are independent of $x,t\in\RR^d$.

{\rm (ii)} For $x,t,t'\in\RR^d$ satisfying $d(x,t)\ge2|t-t'|$,
\begin{align}\label{Riesz-kernel-2}
|K_j(x,t)-K_j(x,t')|\lesssim \frac{|t-t'|}{|x-t|}\frac{1}{\left|B(x,d(x,t))\right|_{\kappa}}.
\end{align}
\end{lemma}

\begin{proof}
In \cite{ADH1, DH1}, useful lower and upper bounds of the $\kappa$-Poisson kernel $(\tau_x P_{x_0})(-t)$ are obtained and they are simplified a little in \cite{JL2}, that is,
\begin{align}\label{Poisson-ker-bound-1}
\frac{c_1}{\left|B(x,x_0+|x-t|)\right|_{\kappa}}\frac{x_0}{x_0+|x-t|}
\leq(\tau_x P_{x_0})(-t)\leq
\frac{c_2x_0}{\left|B(x,x_0+d(x,t))\right|_{\kappa}}\frac{x_0+d(x,t)}{x_0^2+|x-t|^2},
\end{align}
where the constants $c_1,c_2>0$ are independent of $x_0>0$ and $x,t\in\RR^d$. From (\ref{conjugate-kernel-1}) and (\ref{Poisson-ker-bound-1}) it readily follows that
\begin{align*}
\frac{c_1}{\left|B(x,x_0+|x-t|)\right|_{\kappa}}\frac{|x_j-t_j|}{x_0+|x-t|}
\leq |(\tau_{x}Q^{(j)}_{x_0})(-t)| \leq
\frac{c_2|x_j-t_j|}{\left|B(x,x_0+d(x,t))\right|_{\kappa}}\frac{x_0+d(x,t)}{x_0^2+|x-t|^2}.
\end{align*}
Now for $x,t\in\RR^d$ satisfying $d(x,t)>0$, (\ref{Riesz-kernel-1}) is immediate by letting $x_0\rightarrow0+$.

In proving (\ref{Riesz-kernel-2}), we need the following estimate of the gradient of the $\kappa$-Poisson kernel (see \cite[(5.6)]{ADH1})
\begin{align}\label{Poisson-ker-bound-2}
\left|\nabla^{t}\left[(\tau_x P_{x_0})(-t)\right]\right|
\lesssim \frac{(\tau_x P_{x_0})(-t)}{x_0+d(x,t)}.
\end{align}
For $x,t,t'\in\RR^d$ and $x_0>0$, it follows from (\ref{conjugate-kernel-1}) that $|(\tau_{x}Q^{(j)}_{x_0})(-t)-(\tau_{x}Q^{(j)}_{x_0})(-t')|$ is dominated by
\begin{align*}
\frac{|t_j-t_j'|}{x_0}(\tau_x P_{x_0})(-t)+\frac{|x_j-t_j'|}{x_0}\left|\left\langle\nabla\left[(\tau_x P_{x_0})(-\cdot)\right](\tilde{t}),t-t'\right\rangle\right|,
\end{align*}
with some $\tilde{t}$ on the line between $t$ and $t'$. Thus applying (\ref{Poisson-ker-bound-1}) and (\ref{Poisson-ker-bound-2}), and noting that
\begin{align*}
d(x,\tilde{t})\ge d(x,t)-|t-t'|\ge d(x,t)/2 \quad \hbox{and}\quad |x-t'|/2\le|x-t|\le2|x-\tilde{t}|
\end{align*}
whenever $d(x,t)\ge2|t-t'|$, we get
\begin{align*}
\left|(\tau_{x}Q^{(j)}_{x_0})(-t)-(\tau_{x}Q^{(j)}_{x_0})(-t')\right|
\lesssim & \frac{|t_j-t_j'|}{\left|B(x,x_0+d(x,t))\right|_{\kappa}}
\left[\frac{x_0+d(x,t)}{x_0^2+|x-t|^2}
+\frac{|x_j-t_j'|}{x_0^2+|x-\tilde{t}|^2}\right]\\
\lesssim & \frac{|t-t'|}{x_0+|x-t|}\frac{1}{\left|B(x,x_0+d(x,t))\right|_{\kappa}}.
\end{align*}
Again by letting $x_0\rightarrow0+$, the estimate in (\ref{Riesz-kernel-2}) is proved for $d(x,t)\ge2|t-t'|$.
\end{proof}

\begin{corollary}\label{Riesz-truncated-kernel-a} There exists a constant $C>0$ independent of $\epsilon>0$ such that
\begin{align}\label{Riesz-Hormander-1}
\int_{d(x,t)>2|t-t'|}|K_j^{\epsilon}(x,t)-K_j^{\epsilon}(x,t')|\,d\omega_{\kappa}(x)\le C,\qquad t,t'\in\RR^d.
\end{align}
\end{corollary}

\begin{proof}
Let us first consider the $\kappa$-Riesz kernels $K_j(x,t)$ ($1\le j\le d$) themselves. Applying (\ref{Riesz-kernel-2}) and in view of that (see \cite[(32)]{JL2})
\begin{align}\label{ball-estimate-2}
\left|B(x,d(x,t))\right|_{\kappa}\asymp\left|B(t,d(x,t))\right|_{\kappa},
\end{align}
we have, for $t\neq t'$,
\begin{align}\label{Riesz-Hormander-1-1}
&\int_{d(x,t)>2|t-t'|}|K_j(x,t)-K_j(x,t')|\,d\omega_{\kappa}(x)\nonumber\\
\lesssim &\sum_{\ell=1}^{\infty}
\int_{2^{\ell}|t-t'|<d(x,t)\le2^{\ell+1}|t-t'|}\frac{|t-t'|}{|x-t|}\frac{1}{\left|B(t,d(x,t))\right|_{\kappa}}\,d\omega_{\kappa}(x)\nonumber\\
\lesssim &\sum_{\ell=1}^{\infty}\frac{1}{2^{\ell}} \frac{1}{\left|B(t,2^{\ell}|t-t'|)\right|_{\kappa}}
\int_{{\mathcal{O}}(B(t,2^{\ell+1}|t-t'|))}\,d\omega_{\kappa}(x),
\end{align}
where
$$
{\mathcal{O}}(B(x',r))=\{x:\, d(x,x')<r\}
$$
denotes the $G$-orbit of the ball $B(x',r)$. It follows immediately that there exists a constant $C>0$ such that
\begin{align}\label{Riesz-Hormander-2}
\int_{d(x,t)>2|t-t'|}|K_j(x,t)-K_j(x,t')|\,d\omega_{\kappa}(x)
\le C,\qquad t,t'\in\RR^d.
\end{align}

We now turn to the associated estimate for $K_j^{\epsilon}$. Assume that $d(x,t)>2|t-t'|$. It is obvious that $|d(x,t)-d(x,t')|\le|t-t'|$, and so $d(x,t)/2\le d(x,t')\le2d(x,t)$.

If $d(x,t)\le\epsilon$ and $d(x,t')\le\epsilon$, then $K_j^{\epsilon}(x,t)=K_j^{\epsilon}(x,t')=0$; if $d(x,t)>\epsilon$ and $d(x,t')>\epsilon$, then $K_j^{\epsilon}(x,t)-K_j^{\epsilon}(x,t')=K_j(x,t)-K_j(x,t')$; if $d(x,t)>\epsilon$ and $d(x,t')\le\epsilon$, then $K_j^{\epsilon}(x,t)-K_j^{\epsilon}(x,t')=K_j(x,t)$ and $\epsilon<d(x,t)\le 2d(x,t')\le2\epsilon$; and similarly, if $d(x,t)\le\epsilon$ and $d(x,t')>\epsilon$, then $K_j^{\epsilon}(x,t)-K_j^{\epsilon}(x,t')=-K_j(x,t')$ and $\epsilon<d(x,t')\le 2d(x,t)\le2\epsilon$.
Thus $\int_{d(x,t)>2|t-t'|}|K_j^{\epsilon}(x,t)-K_j^{\epsilon}(x,t')|\,d\omega_{\kappa}(x)$ is dominated by
\begin{align*}
&\int_{d(x,t)>2|t-t'|}|K_j(x,t)-K_j(x,t')|\,d\omega_{\kappa}(x)\\
& \qquad\qquad\qquad  +\int_{\epsilon<d(x,t)\le2\epsilon}|K_j(x,t)|\,d\omega_{\kappa}(x)
+\int_{\epsilon<d(x,t')\le2\epsilon}|K_j(x,t')|\,d\omega_{\kappa}(x)
\end{align*}
From (\ref{Riesz-Hormander-2}) the first term above is bounded by a constant; while, by use of (\ref{Riesz-kernel-1}) and (\ref{ball-estimate-2}), the second term is controlled by a multiple of
\begin{align*}
\int_{\epsilon<d(x,t)\le2\epsilon}\frac{|x_j-t_j|}{\left|B(t,d(x,t))\right|_{\kappa}}\frac{d(x,t)}{|x-t|^2}\,d\omega_{\kappa}(x)
\lesssim \frac{1}{\left|B(t,\epsilon)\right|_{\kappa}}\int_{{\mathcal{O}}(B(t,2\epsilon))}\,d\omega_{\kappa}(x)
\lesssim 1;
\end{align*}
and the same estimate holds for the third term. Summing these concludes (\ref{Riesz-Hormander-1}).
\end{proof}

We are now in a position to consider the operators ${\mathcal{R}}_j^{\epsilon,N}$, for $0<\epsilon<N<\infty$, defined by
 \begin{align}\label{truncated-Riesz-5}
({\mathcal{R}}_j^{\epsilon,N}f)(x)=c_{\kappa}\int_{\RR^d}f(t)K_j^{\epsilon,N}(x,t)\,d\omega_{\kappa}(t),\qquad x\in\RR^d,
\end{align}
where
$$
K_j^{\epsilon,N}(x,t)=K_j^{\epsilon}(x,t)-K_j^{N}(x,t).
$$

\begin{lemma}\label{Riesz-truncated-BMO-a}
There exists a constant $C>0$ independent of $\epsilon,N$ and $f\in L^\infty(\RR^d)$ such that
\begin{eqnarray}\label{Riesz-truncated-BMO-1}
\|{\mathcal{R}}_j^{\epsilon,N}f\|_{*,\kappa}\le C\|f\|_{L^{\infty}}.
\end{eqnarray}
\end{lemma}

\begin{proof}
We fix a ball $B=B_{\delta}(x')$ and decompose $f\in L^\infty(\RR^d)$ by $f=f_1+f_2$, where $f_1=f\chi_{{\mathcal{O}}(3B)}$ and  $f_2=f\chi_{[{\mathcal{O}}(3B)]^c}$.

By Corollary \ref{Riesz-2-d}, ${\mathcal{R}}_j^{\epsilon,N}$ is bounded from $L_{\kappa}^2(\RR^d)$ to itself with a bound of operator norms independent of $\epsilon,N$. Thus
\begin{align}\label{truncated-Riesz-6}
&\int_B|{\mathcal{R}}_j^{\epsilon,N}f_1|d\omega_{\kappa}
\le |B|_{\kappa}^{1/2}\left(\int_B|{\mathcal{R}}_j^{\epsilon,N}f_1|^2d\omega_{\kappa}\right)^{1/2}\nonumber\\
\lesssim & |B|_{\kappa}^{1/2}\left(\int_{{\mathcal{O}}(3B)}|f_1|^2d\omega_{\kappa}\right)^{1/2}
\lesssim |B|_{\kappa}\|f\|_{L^{\infty}}.
\end{align}
Passing to ${\mathcal{R}}_j^{\epsilon,N}f_2$, for $x\in B$ we have
\begin{align*}
\left|({\mathcal{R}}_j^{\epsilon,N}f_2)(x)-({\mathcal{R}}_j^{\epsilon,N}f_2)(x')\right|
&=c_{\kappa}\int_{\RR^d}|f_2(t)||K_j^{\epsilon,N}(x,t)-K_j^{\epsilon,N}(x',t)|\,d\omega_{\kappa}(t)\\
&\lesssim \|f\|_{L^{\infty}}\int_{[{\mathcal{O}}(3B)]^c}|K_j^{\epsilon,N}(x,t)-K_j^{\epsilon,N}(x',t)|\,d\omega_{\kappa}(t).
\end{align*}
It is noted that $K_j(x,t)=-K_j(t,x)$ by (\ref{truncated-Riesz-kernel-2}), and then by Corollary \ref{Riesz-truncated-kernel-a},
\begin{align}\label{truncated-Riesz-7}
\left|({\mathcal{R}}_j^{\epsilon,N}f_2)(x)-({\mathcal{R}}_j^{\epsilon,N}f_2)(x')\right|
\lesssim \|f\|_{L^{\infty}},\qquad x\in B.
\end{align}
Combining (\ref{truncated-Riesz-6}) and (\ref{truncated-Riesz-7}) yields
\begin{align*}
\int_B|({\mathcal{R}}_j^{\epsilon,N}f)(x)-({\mathcal{R}}_j^{\epsilon,N}f_2)(x')|d\omega_{\kappa}
\lesssim |B|_{\kappa}\|f\|_{L^{\infty}}.
\end{align*}
Therefore ${\mathcal{R}}_j^{\epsilon,N}f\in \B_{\kappa}$, and (\ref{Riesz-truncated-BMO-1}) holds.
\end{proof}

We now turn to the proof of Theorem \ref{Riesz-BMO-a}.

Assume $f\in L^\infty(\RR^d)$ and fix a ball $B=B_{\delta}(x')$. Note that $|d(x,t)-d(x',t)|\le|x-x'|\le\delta$ for $x\in B$.

It follows from (\ref{truncated-Riesz-4}) that, for $x\in B$,
\begin{align}\label{truncated-Riesz-4-1}
\left|(\tilde{\mathcal{R}}_j^{N}f)(x)-(\tilde{\mathcal{R}}_j^{N}f)(x')\right|
\lesssim \|f\|_{L^{\infty}}\int_{\RR^d}\left|K_j^{N}(x,t)-K_j^{N}(x',t)\right|\,d\omega_{\kappa}(t).
\end{align}
Adopting closely the argument in the proof of Corollary \ref{Riesz-truncated-kernel-a}, the integral above is controlled by
\begin{align}\label{Riesz-Hormander-3}
&\int_{d(x,t)>N}|K_j(x,t)-K_j(x',t)|\,d\omega_{\kappa}(t)\nonumber\\
& \qquad\qquad  +\int_{N<d(x,t)\le N+\delta}|K_j(x,t)|\,d\omega_{\kappa}(t)
+\int_{N<d(x',t)\le N+\delta}|K_j(x',t)|\,d\omega_{\kappa}(t).
\end{align}
We consider the case for $N\ge2(|x'|+\delta)$. The first term in (\ref{Riesz-Hormander-3}), after a similar treatment in (\ref{Riesz-Hormander-1-1}), is bounded by a multiple of
\begin{align}\label{Riesz-Hormander-4}
\sum_{\ell=1}^{\infty}
\int_{2^{\ell-1}N<d(x,t)\le2^{\ell}N}\frac{|x-x'|}{|x-t|}\frac{1}{\left|B(x,d(x,t))\right|_{\kappa}}\,d\omega_{\kappa}(t)
\lesssim \frac{\delta}{N}.
\end{align}
As for the second term in (\ref{Riesz-Hormander-3}), using (\ref{Riesz-kernel-1}) gives
\begin{align*}
\int_{N<d(x,t)\le N+\delta}|K_j(x,t)|\,d\omega_{\kappa}(t)
\lesssim & \frac{1}{\left|B(x,N)\right|_{\kappa}}\int_{N<d(x,t)\le N+\delta}\,d\omega_{\kappa}(t)\\
= & \frac{1}{\left|B(x/N,1)\right|_{\kappa}}\int_{1<d(x/N,t)\le 1+\delta/N}\,d\omega_{\kappa}(t);
\end{align*}
and since $B(x/N,1)\supseteq B(x'/N,1/2)$ and $|d(x'/N,t)-d(x/N,t)|\le \delta/N$, it follows that
\begin{align*}
\int_{N<d(x,t)\le N+\delta}|K_j(x,t)|\,d\omega_{\kappa}(t)
\lesssim \frac{1}{\left|B(x'/N,1/2)\right|_{\kappa}}\int_{1-\delta/N<d(x'/N,t)\le 1+2\delta/N}\,d\omega_{\kappa}(t).
\end{align*}
The same estimate holds for the third term in (\ref{Riesz-Hormander-3}). Inserting them and (\ref{Riesz-Hormander-4}) into (\ref{Riesz-Hormander-3}), and then (\ref{truncated-Riesz-4-1}), asserts that $(\tilde{\mathcal{R}}_j^{N}f)(x)-(\tilde{\mathcal{R}}_j^{N}f)(x')\rightarrow 0$ uniformly for $x\in B$ as $N\rightarrow\infty$.

Since, from (\ref{truncated-Riesz-4}) and (\ref{truncated-Riesz-5}),
$$
\tilde{\mathcal{R}}_j^{\epsilon}f-\tilde{\mathcal{R}}_j^{N}f={\mathcal{R}}_j^{\epsilon,N}f,
$$
by Lemma \ref{Riesz-truncated-BMO-a} and (\ref{bmo-4}) we have
\begin{align}\label{bmo-5}
\frac{1}{|B|_\kappa}\int_{B}\left|\tilde{\mathcal{R}}_j^{\epsilon}f-\tilde{\mathcal{R}}_j^{N}f-(\tilde{\mathcal{R}}_j^{\epsilon}f)_B+(\tilde{\mathcal{R}}_j^{N}f)_{B}\right|^2\,d\omega_{\kappa}
\le C\|f\|_{L^{\infty}}^2.
\end{align}
But, with $a=(\tilde{\mathcal{R}}_j^{N}f)(x')$,
$$
(\tilde{\mathcal{R}}_j^{N}f)(x)-(\tilde{\mathcal{R}}_j^{N}f)_{B}=(\tilde{\mathcal{R}}_j^{N}f)(x)-a-[\tilde{\mathcal{R}}_j^{N}f-a]_{B}
$$
tends to zero uniformly for $x\in B$ as $N\rightarrow\infty$, as just proved. Thus from (\ref{bmo-5}) we get
\begin{align*}
\frac{1}{|B|_\kappa}\int_{B}\left|\tilde{\mathcal{R}}_j^{\epsilon}f-(\tilde{\mathcal{R}}_j^{\epsilon}f)_B\right|^2\,d\omega_{\kappa}
\le C\|f\|_{L^{\infty}}^2.
\end{align*}
Finally letting $\epsilon\rightarrow0+$, Proposition \ref{Riesz-BMO-b} asserts $|B|_{\kappa}^{-1}\int_{B}\left|\tilde{\mathcal{R}}_jf-(\tilde{\mathcal{R}}_jf)_B\right|^2\,d\omega_{\kappa}
\le C\|f\|_{L^{\infty}}^2$. Thus, in view of (\ref{bmo-4}), $\|\tilde{\mathcal{R}}_jf\|_{*,\kappa}\lesssim\|f\|_{L^{\infty}}$, and the proof of Theorem \ref{Riesz-BMO-a} is completed.

\section{The associated Carleson measures}

As usual, for a ball $B=B_{r}(x')$ in $\RR^d$ the associated tent $T(B)$ is the closed set in $\RR^{1+d}_+$ given by
$$
T(B)=\{(x_0,x):\, |x-x'|\le r-x_0\}.
$$

\begin{definition}\label{car-def}
A Borel measure $d\nu$ on $\RR^{1+d}_+$ is called a $\kappa$-Carleson measure if there exists some $c>0$ such that for all balls $B$ in $\RR^d$,
$$
\iint_{T(B)}|d\nu|\leq c|B|_\kappa;
$$
and $\|d\nu\|_{\kappa}:=\sup_B|B|_{\kappa}^{-1}\iint_{T(B)}|d\nu|$ is named the $\kappa$-Carleson norm of $d\nu$.
\end{definition}

As in the classical case, the property underlying the definition of $\kappa$-Carleson measures is not limited to balls. Indeed we have the following lemma.

\begin{lemma}\label{car-pro}
If $d\nu$ is a $\kappa$-Carleson measure on $\RR^{1+d}_+$ and $E\subset\RR^d$ is open, then
$$
\iint_{\widehat{E}}|d\nu|\lesssim\|d\nu\|_{\kappa}|E|_\kappa,
$$
where
$\widehat{E}=\{(r,x)\in\RR^{1+d}_+:\,\, B_r(x)\subset E\}$.
\end{lemma}

The proof of the lemma is a straightforward adaptation of that of \cite[Lemma 9.4]{Duo}.
A $\kappa$-Carleson measure on $\RR^{1+d}_+$ can also be characterized as a measure $d\nu$ for which the $\kappa$-Poisson integral defines a bounded operator from $L_{\kappa}^p(\RR^d)$ to $L^p(\RR^{1+d}_+, |d\nu|)$ for $1<p<\infty$.

\begin{theorem}\label{car-pro-1}
Assume that $d\nu$ is a Borel measure on $\RR^{1+d}_+$.

{\rm(i)} If $d\nu$ is a $\kappa$-Carleson measure, then for every $p\in(1,\infty)$,
\begin{eqnarray*}
\iint_{\RR^{1+d}_+}|(Pf)(x_0,x)|^p\,|d\nu|\lesssim\|d\nu\|_\kappa\|f\|_{L_{\kappa}^p}^p,\qquad f\in L^p_{\kappa}(\RR^d).
 \end{eqnarray*}

{\rm(ii)} If for some $p\in[1,\infty)$, there holds $\iint_{\RR^{1+d}_+}|(Pf)(x_0,x)|^p\,|d\nu|\le c\|f\|_{L_{\kappa}^p}^p$ for $f\in L^p_{\kappa}(\RR^d)$, where $c>0$ is independent of $f$, then $d\nu$ is a $\kappa$-Carleson measure.
\end{theorem}

\begin{proof}
First suppose that $d\nu$ is a $\kappa$-Carleson measure.
Let $E_\lambda=\{x\in\RR^d:(P^*_1f)(x)>\lambda\}$. It then follows that $\{(x_0,x)\in\RR^{1+d}_+:\,|(Pf)(x_0,x)|>\lambda\}\subset\widehat{E_\lambda}$,
and by Lemma \ref{car-pro},
\begin{align*}
\iint_{\RR^{1+d}_+}|(Pf)(x_0,x)|^p\,|d\nu|
 &=p\int_{0}^\infty\lambda^{p-1}\iint_{\{(x_0,x):\,|(Pf)(x_0,x)|>\lambda\}}|d\nu|\,d\lambda\\
 &\lesssim
 p\|d\nu\|_\kappa\int_{0}^\infty\lambda^{p-1}|E_\lambda|_\kappa d\lambda,
 \end{align*}
that is identical with $\|d\nu\|_\kappa\int_{\RR^d}(P^*_1f)(x)^pd\omega_\kappa$. Part (i) is immediate by Proposition \ref{Poisson-Integral-Max-N}.

Conversely, suppose that $d\nu$ satisfies the condition in part (ii). For a ball $B=B_{r}(x')$ in $\RR^d$ and for $(x_0,x)\in T(B)$,  using the fact $B_{x_0}(x)\subset B$ and the first inequality in (\ref{Poisson-ker-bound-1}) we have
\begin{align*}
(P\chi_{B})(x_0,x)\ge c_{\kappa}\int_{B_{x_0}(x)}(\tau_{x}P_{x_0})(-t)\,d\omega_{\kappa}(t)
\gtrsim\int_{B_{x_0}(x)}\frac{x_0}{x_0+|x-t|}\frac{d\omega_{\kappa}(t)}{\left|B(x,x_0+|x-t|)\right|_{\kappa}},
\end{align*}
which implies $(P\chi_{B})(x_0,x)\gtrsim1$ for $(x_0,x)\in T(B)$. Consequently
\begin{align*}
\iint_{T(B)}|d\nu|\lesssim \iint_{T(B)}|(P\chi_{B})(x_0,x)|^p|d\nu| \lesssim\|\chi_{B}\|_{L_{\kappa}^p}^p=|B|_{\kappa}.
\end{align*}
Hence $d\nu$ is a $\kappa$-Carleson measure.
\end{proof}

We now introduce a class of functions on $\RR^d$ in terms of $\kappa$-Carleson measures.

\begin{definition}\label{bmc-a}
A function $\varphi\in L_{\kappa,{\rm loc}}(\RR^d)$ is said to be in $\BC_{\kappa}(\RR^d)$ if it satisfies the following two conditions:

{\rm(i)}
\begin{eqnarray}\label{H-1-BMO-1-1}
\int_{\RR^d}\frac{|\varphi(x)|}{(1+|x|)^{2|\kappa|+d+1}}
\,d\omega_\kappa(x)
<\infty;
\end{eqnarray}

{\rm(ii)} $d\nu_{\kappa,\varphi}(x_0,x)=x_0|\nabla_{\tilde{\kappa}} u_\varphi(x_0,x)|^2\,dx_0d\omega_\kappa(x)$
is a $\kappa$-Carleson measure on $\RR^{1+d}_+$, where $\nabla_{\tilde{\kappa}}=(\partial_0,D_1,\dots,D_d)$ is the $\tilde{\kappa}$-gradient,
and $u_\varphi(x_0,x)$ is the $\kappa$-Poisson integral of $\varphi$.

Moreover the norm of $\varphi$ in $\BC_{\kappa}(\RR^d)$ is defined by $\|\varphi\|_{*,C_{\kappa}}=\|d\nu_{\kappa,\varphi}\|_{\kappa}^{1/2}$.
\end{definition}

\begin{proposition}\label{bmc-b}
A bounded function $\varphi$ on $\RR^d$ belongs to $\BC_{\kappa}(\RR^d)$ and $\|\varphi\|_{*,C_{\kappa}}\lesssim\|\varphi\|_{L^{\infty}}$.
\end{proposition}

\begin{proof}
For a fixed ball $B=B_{\delta}(x')$ in $\RR^d$,
\begin{align}\label{Carleson-1}
\iint_{T(B)}d\nu_{\kappa,\varphi\chi_{{\mathcal{O}}(2B)}}(x_0,x)\le \iint_{\RR^{1+d}_+}x_0|\nabla_{\tilde{\kappa}} u_{\varphi\chi_{{\mathcal{O}}(2B)}}(x_0,x)|^2\,dx_0d\omega_\kappa(x);
\end{align}
and since for $f\in L_{\kappa}^1(\RR^d)\cup L_{\kappa}^2(\RR^d)$, by (\ref{Dunkl-kernel-eigenfunction-1}) and (\ref{Poisson-2-3}),
\begin{align}
\left[\SF_{\kappa}(\partial_{x_0}u_f(x_0,\cdot))\right](\xi)&=-|\xi|e^{-x_0|\xi|}\left(\SF_{\kappa}f\right)(\xi),\label{Dunkl-transform-1}\\
\left[\SF_{\kappa}(D_ju_f(x_0,\cdot))\right](\xi)&=i\xi_je^{-x_0|\xi|}\left(\SF_{\kappa}f\right)(\xi), \quad j=1,\cdots,d, \label{Dunkl-transform-2}
\end{align}
the Plancherel formula asserts that the integral on the right hand side of (\ref{Carleson-1}) is equal to $2^{-1}\int_{\RR^d}|\varphi\chi_{{\mathcal{O}}(2B)}|^2d\omega_\kappa\lesssim \|\varphi\|_{L^{\infty}}^2|B|_{\kappa}$.

Nevertheless, from (\ref{Poisson-2-2}),
\begin{align}\label{Carleson-2}
|\nabla_{\tilde{\kappa}} u_{\varphi\chi_{[{\mathcal{O}}(2B)]^c}}(x_0,x)|
\le c_{\kappa}\|\varphi\|_{L^{\infty}}\int_{[{\mathcal{O}}(2B)]^c}\left|\nabla_{\tilde{\kappa}}^{(x_0,x)}[(\tau_{x}P_{x_0})(-t)]\right|\,d\omega_{\kappa}(t),
\end{align}
and by \cite[Proposition 5.1]{ADH1} (also \cite[(32)]{JL2}),
\begin{align}\label{Poisson-ker-2}
\left|\nabla_{\tilde{\kappa}}^{(x_0,x)}[(\tau_{x}P_{x_0})(-t)]\right|
\lesssim \frac{(x_0+d(t,x))^{-1}}{\left|B(x,x_0+d(t,x))\right|_{\kappa}}.
\end{align}
Since for $(x_0,x)\in T(B)$ and $t\in[{\mathcal{O}}(2B)]^c$, $d(t,x)\ge d(t,x')/2\ge\delta$, it follows from (\ref{Carleson-2}) and (\ref{Poisson-ker-2}) that
\begin{align}\label{Carleson-3}
|\nabla_{\tilde{\kappa}} u_{\varphi\chi_{[{\mathcal{O}}(2B)]^c}}(x_0,x)|
\lesssim\|\varphi\|_{L^{\infty}}\int_{[{\mathcal{O}}(2B)]^c}\frac{d\omega_{\kappa}(t)}{d(t,x)\left|B(x,d(t,x))\right|_{\kappa}}
\lesssim \frac{\|\varphi\|_{L^{\infty}}}{\delta}.
\end{align}
Thus
$$
\iint_{T(B)}d\nu_{\kappa,\varphi\chi_{[{\mathcal{O}}(2B)]^c}}(x_0,x)
\lesssim \left(\frac{\|\varphi\|_{L^{\infty}}}{\delta}\right)^2\iint_{T(B)}x_0\,dx_0d\omega_\kappa(x)
\lesssim\|\varphi\|_{L^{\infty}}^2|B|_{\kappa},
$$
so that $\iint_{T(B)}d\nu_{\kappa,\varphi}(x_0,x)
\lesssim\|\varphi\|_{L^{\infty}}^2|B|_{\kappa}$, as desired.
\end{proof}

Note that, for the symbol ``$\BC_{\kappa}$", ``BM" means ``bounded mean" since, at least, a function $\varphi$ with the property $\sup_{A>0}A^{-d-2|\kappa|}\int_{|x|\le A}|\varphi|d\omega_\kappa<\infty$ satisfies condition (i), and ``C" certainly indicates that it is connected with Carleson measures as stated in condition (ii).

At this stage we shall prove an analog of Theorem \ref{Riesz-BMO-a} for the space $\BC_{\kappa}(\RR^d)$, that is necessary in characterizing the dual of the Hardy space $H_{\kappa}(\RR^d)$ in the next section.

\begin{theorem}\label{Riesz-BMC-a}
The $\kappa$-Riesz transform $\tilde{\mathcal{R}}_j$ defined by (\ref{truncated-Riesz-4}) and (\ref{Riesz-2-3}) is bounded from $L^\infty(\RR^d)$ to $\BC_{\kappa}(\RR^d)$, where $j=1,\cdots,d$.
\end{theorem}

If the multiplicity function $\kappa$ is identical to zero, i.e., $\kappa(\alpha)=0$ for all $\alpha\in R$, it is well known that $\BC_{0}(\RR^d)=\B_{0}(\RR^d)$ (see \cite[Theorem 3]{FS1}); but for general nonnegative $\kappa$, a partial relation of the two classes $\BC_{\kappa}(\RR^d)$ and $\B_{\kappa}(\RR^d)$ is a consequence of the duality of $H_{\kappa}(\RR^d)$ and $\BC_{\kappa}(\RR^d)$ to be proved in the next section. It is uncertain whether there was a direct proof of the converse of this inclusion.

In proving Theorem \ref{Riesz-BMC-a}, we shall need an estimate of the $\tilde{\kappa}$-gradient of the conjugate $\kappa$-Poisson kernels $(\tau_{x}Q^{(j)}_{x_0})(-t)$ ($j=1,2,\cdots,d$) given below.

\begin{lemma}\label{conjugate-Poisson-b}
The conjugate $\kappa$-Poisson kernels $(\tau_{x}Q^{(j)}_{x_0})(-t)$ ($j=1,2,\cdots,d$) satisfy
\begin{align}\label{conjugate-kernel-5}
\left|\nabla_{\tilde{\kappa}}^{(x_0,x)}[(\tau_{x}Q^{(j)}_{x_0})(-t)]\right|
\lesssim \frac{(x_0+d(t,x))^{-1}}{\left|B(x,x_0+d(t,x))\right|_{\kappa}}.
\end{align}
\end{lemma}

\begin{proof}
Define the function $U_{x_0}(x,t)$ by
\begin{align}\label{U-function-1}
U_{x_0}(x,t)=\frac{1}{2}\int_0^{\infty}e^{-\frac{x_0^2}{4v}}\,h_v(x,t)\,\frac{dv}{v^{3/2}}
=\frac{1}{x_0}\int_0^{\infty}e^{-u}h_{x_0^2/4u}(x,t)\,\frac{du}{\sqrt{u}},
\end{align}
where $h_v(x,t)$ is the heat kernel associated to the Dunkl operators given by (\ref{heat-kernel-1}).

We are going to show, for $i=1,2,\cdots,d$,
\begin{align}\label{U-function-2}
\left|D_i^x\left[U_{x_0}(x,t)\right]\right|
\lesssim
\frac{(x_0+d(x,t))^{-1}}{\left|B(x,x_0+d(x,t))\right|_{\kappa}}\frac{|x_i-t_i|}{x_0^2+|x-t|^2}.
\end{align}

From (\ref{heat-kernel-2}) and (\ref{U-function-1}),
\begin{align}\label{U-function-3}
D_i^x\left[U_{x_0}(x,t)\right]
=\frac{2(t_i-x_i)}{x_0^3}\int_0^{\infty}e^{-u}h_{x_0^2/4u}(x,t)\sqrt{u}\,du,
\end{align}
and making use of \cite[Theoem 3.1]{DH1}
\begin{align*}
h_v(x,t)\lesssim\left(1+\frac{|x-t|}{\sqrt{v}}\right)^{-2} \frac{e^{-cd(x,t)^2/v}}{|B(x,\sqrt{v})|_{\kappa}}
\end{align*}
yields
\begin{align*}
\left|D_i^x\left[U_{x_0}(x,t)\right]\right|
\lesssim \frac{|x_i-t_i|}{x_0^3}\int_0^{\infty}\left(1+\frac{u|x-t|^2}{x_0^2}\right)^{-1}
\frac{e^{-u(1+4cd(x,t)^2/x_0^2)}}{|B(x,x_0/\sqrt{u})|_{\kappa}}\sqrt{u}\,du.
\end{align*}

Since $|B(x,r_2)|_{\kappa}\lesssim(r_2/r_1)^{d+2|\kappa|}|B(x,r_1)|_{\kappa}$ for $r_2\ge r_1>0$, it follows that, for $u\ge u_0$ with $u_0=\left(x_0/(x_0+d(x,t))\right)^2$,
$$
\frac{e^{-u(1+4cd(x,t)^2/x_0^2)}}{|B(x,x_0/\sqrt{u})|_{\kappa}}
\lesssim \frac{e^{-c'u/u_0}}{|B(x,x_0+d(x,t))|_{\kappa}}\left(\frac{u}{u_0}\right)^{|\kappa|+d/2}
\lesssim \frac{(u_0/u)^{2}}{|B(x,x_0+d(x,t))|_{\kappa}}.
$$
Consequently,
\begin{align*}
\left|D_i^x\left[U_{x_0}(x,t)\right]\right|
\lesssim \frac{x_0^{-3}|x_i-t_i|}{|B(x,x_0+d(x,t))|_{\kappa}}
\left[\int_0^{u_0}\frac{x_0^2\sqrt{u}\,du}{x_0^2+u|x-t|^2}
+u_0^2\int_{u_0}^{\infty}\frac{x_0^2u^{-3/2}du}{x_0^2+u|x-t|^2}\right].
\end{align*}
Thus if $|x-t|\le x_0+d(x,t)$,
\begin{align*}
\left|D_i^x\left[U_{x_0}(x,t)\right]\right|
\lesssim \frac{x_0^{-3}|x_i-t_i|u_0^{3/2}}{|B(x,x_0+d(x,t))|_{\kappa}}
\asymp\frac{(x_0+d(x,t))^{-1}}{\left|B(x,x_0+d(x,t))\right|_{\kappa}}\frac{|x_i-t_i|}{x_0^2+|x-t|^2},
\end{align*}
and if $|x-t|>x_0+d(x,t)$,
\begin{align*}
\left|D_i^x\left[U_{x_0}(x,t)\right]\right|
\lesssim \frac{x_0^{-3}|x_i-t_i|}{|B(x,x_0+d(x,t))|_{\kappa}}
\frac{x_0^2 u_0^{1/2}}{|x-t|^2}
\asymp\frac{(x_0+d(x,t))^{-1}}{\left|B(x,x_0+d(x,t))\right|_{\kappa}}\frac{|x_i-t_i|}{x_0^2+|x-t|^2},
\end{align*}
so that (\ref{U-function-2}) is concluded.

If $i\neq j$, it then follows from (\ref{conjugate-kernel-4}), (\ref{U-function-1}) and (\ref{U-function-2}) that
\begin{align}\label{U-function-4}
\left|D_i^x[(\tau_{x}Q^{(j)}_{x_0})(-t)]\right|=\frac{|x_j-t_j|}{\sqrt{\pi}}\left|D_i^x\left[U_{x_0}(x,t)\right]\right|
\lesssim\frac{(x_0+d(x,t))^{-1}}{\left|B(x,x_0+d(x,t))\right|_{\kappa}}.
\end{align}
If $i=j$, from (\ref{Dunkl-operator-1}) and (\ref{conjugate-kernel-4}) direct calculations show
\begin{align}\label{conjugate-kernel-6}
&\sqrt{\pi}D_j^x\left[(\tau_{x}Q^{(j)}_{x_0})(-t)\right] \nonumber\\
=&(x_j-t_j)D_j^x\left[U_{x_0}(x,t)\right]+U_{x_0}(x,t)
+\sum_{\alpha\in R_+}\kappa(\alpha)\alpha_j^2U_{x_0}(\sigma_{\alpha}(x),t).
\end{align}
But from (\ref{Poisson-kernel-2}), (\ref{Poisson-ker-bound-1}) and (\ref{U-function-1}), for each $\alpha\in R_+$ one has
\begin{align*}
U_{x_0}(\sigma_{\alpha}(x),t)=\frac{\sqrt{\pi}}{x_0}(\tau_{\sigma_{\alpha}(x)}P_{x_0})(-t)
\lesssim\frac{(x_0+d(x,t))^{-1}}{\left|B(x,x_0+d(x,t))\right|_{\kappa}},
\end{align*}
and applying this and (\ref{U-function-2}) to (\ref{conjugate-kernel-6}) gives
\begin{align}\label{U-function-5}
\left|D_j^x[(\tau_{x}Q^{(j)}_{x_0})(-t)]\right|
\lesssim\frac{(x_0+d(x,t))^{-1}}{\left|B(x,x_0+d(x,t))\right|_{\kappa}}.
\end{align}

Finally, from (\ref{U-function-1}) we have
\begin{align*}
\partial_{x_0}[U_{x_0}(x,t)]=-\frac{x_0}{4}\int_0^{\infty}e^{-\frac{x_0^2}{4v}}\,h_v(x,t)\,\frac{dv}{v^{5/2}}
=-\frac{2}{x_0^2}\int_0^{\infty}e^{-u}h_{x_0^2/4u}(x,t)\sqrt{u}\,du,
\end{align*}
so that from (\ref{U-function-2}) and (\ref{U-function-3}),
\begin{align*}
\left|\partial_{x_0}[U_{x_0}(x,t)]\right|
\lesssim
\frac{(x_0+d(x,t))^{-1}}{\left|B(x,x_0+d(x,t))\right|_{\kappa}}\frac{x_0}{x_0^2+|x-t|^2}.
\end{align*}
Thus from (\ref{conjugate-kernel-4}) and (\ref{U-function-2}),
\begin{align}\label{U-function-6}
\left|\partial_{x_0}[(\tau_{x}Q^{(j)}_{x_0})(-t)]\right|=\frac{|x_j-t_j|}{\sqrt{\pi}}\left|\partial_{x_0}\left[U_{x_0}(x,t)\right]\right|
\lesssim\frac{(x_0+d(x,t))^{-1}}{\left|B(x,x_0+d(x,t))\right|_{\kappa}}.
\end{align}

Combining (\ref{U-function-4}), (\ref{U-function-5}) and (\ref{U-function-6}) completes the proof of the lemma.
\end{proof}

The following lemma is necessary too, and will also be used in the next section.

\begin{lemma}\label{sym-riez}
If $\psi\in L^\infty(\RR^d)$ and $f\in H_{\kappa}^1(\RR^d)$ satisfying $|f(x)|\lesssim (1+|x|)^{-2|\kappa|-d-1}$, then
\begin{eqnarray}\label{sym-riez-1}
\int_{\RR^d}({\mathcal{R}}_jf)(x)\,\psi(x)\,d\omega_\kappa(x)=-\int_{\RR^d}f(x)\,(\tilde{\mathcal{R}}_j\psi)(x)\,d\omega_\kappa(x),
\ \ \ j=1,2,\cdots,d.
\end{eqnarray}
\end{lemma}

\begin{proof}
Obviously the integral on the left hand side of (\ref{sym-riez-1}) converges absolutely; and since $\tilde{\mathcal{R}}_j\psi\in\B_\kappa(\RR^d)$ by Theorem \ref{Riesz-BMO-a},
applying (\ref{bmo-2}) to $\tilde{\mathcal{R}}_j\psi$ instead of $f$ with $x'=0$ and
in view of (\ref{cube-estimate-1}) it is asserted that the integral on the right hand side of (\ref{sym-riez-1}) also converges absolutely.

We set $\psi_N=\psi\chi_{|x|\leq N}$. Then $\psi_N\in L^{2}_\kappa(\RR^d)\cap L^\infty(\RR^d)$, and from (\ref{truncated-Riesz-4}) and (\ref{Riesz-2-3}), it follows that
$(\tilde{\mathcal{R}}_j\psi_N)(x)=({\mathcal{R}}_j\psi_N)(x)-a_N$, where $a_N=c_{\kappa}\int_{\RR^d}\psi_N(t)K_j^1(0,t)\,d\omega_{\kappa}(t)$. Consequently, by Corollary \ref{Riesz-2-b} and Proposition \ref{Hardy-integral-a}, we have
\begin{eqnarray*}
\int_{\RR^d}({\mathcal{R}}_jf)(x)\,\psi_N(x)\,d\omega_\kappa(x)=-\int_{\RR^d}f(x)\,(\tilde{\mathcal{R}}_j\psi_N)(x)\,d\omega_\kappa(x).
\end{eqnarray*}
Thus (\ref{sym-riez-1}) is concluded once it is shown that
\begin{eqnarray}\label{sym-riez-2}
\lim_{N\rightarrow\infty}\int_{\RR^d}f(x)\,[\tilde{\mathcal{R}}_j(\psi-\psi_N)](x)\,d\omega_\kappa(x)
=0.
\end{eqnarray}

For the purpose let $b_N=[\tilde{\mathcal{R}}_j(\psi-\psi_N)]_{B_{\sqrt{N}}(0)}$
and insert the number $b_N$ under the integral sign in (\ref{sym-riez-2}), so that (\ref{sym-riez-2}) is equivalent to
\begin{eqnarray}\label{sym-riez-3}
\lim_{N\rightarrow\infty}\int_{\RR^d}f(x)\,\left([\tilde{\mathcal{R}}_j(\psi-\psi_N)](x)-b_N\right)\,d\omega_\kappa(x)
=0.
\end{eqnarray}

It is easy to see that, for appropriately large $N$,
\begin{equation*}
b_N=\int_{|t|\geq N}\frac{c_{\kappa}}{|B_{\sqrt{N}}(0)|_\kappa}\int_{B_{\sqrt{N}}(0)}
 [K_j(x',t)-K_j(0,t)]\,d\omega_\kappa(x')\,\psi(t)d\omega_\kappa(t),
\end{equation*}
so that, for $|x|\leq \sqrt{N}$,
\begin{align*}
&[\tilde{\mathcal{R}}_j(\psi-\psi_N)](x)-b_N\\
=&\int_{|t|\geq N}\frac{c_{\kappa}}{|B_{\sqrt{N}}(0)|_\kappa}\int_{B_{\sqrt{N}}(0)}(K_j(x,t)-K_j(x',t))d\omega_\kappa(x')\,\psi(t)d\omega_\kappa(t).
\end{align*}
Now applying (\ref{Riesz-kernel-2}) we get, for $|x|\leq \sqrt{N}$,
\begin{align*}
\left|[\tilde{\mathcal{R}}_j(\psi-\psi_N)](x)-b_N\right|
\lesssim\frac{\|\psi\|_{L^{\infty}}}{N^{|\kappa|+d/2}}\int_{|t|\geq N}\int_{B_{\sqrt{N}}(0)}\frac{|x-x'|}{|x-t|}\frac{d\omega_\kappa(x')}{\left|B(t,d(x,t))\right|_{\kappa}}\,d\omega_\kappa(t).
\end{align*}
Note that $|x-t|\ge d(x,t)\ge|t|-|x|\ge|t|/2\ge N/2\ge 4\sqrt{N}\ge2|x-x'|$ for large $N$, and so
\begin{align*}
\left|[\tilde{\mathcal{R}}_j(\psi-\psi_N)](x)-b_N\right|
\lesssim
\int_{|t|\geq N}\frac{\sqrt{N}\|\psi\|_{L^{\infty}}}{|t|\left|B(t,|t|)\right|_{\kappa}}d\omega_\kappa(t)
\lesssim\frac{\|\psi\|_{L^{\infty}}}{\sqrt{N}}
\end{align*}
because of $\left|B(t,|t|)\right|_{\kappa}\asymp|t|^{2\kappa+d}$ according to (\ref{cube-estimate-1}).
Subsequently,
\begin{eqnarray}\label{sym-riez-4}
\int_{|x|\le\sqrt{N}}|f(x)|\,|[\tilde{\mathcal{R}}_j(\psi-\psi_N)](x)-b_N|\,d\omega_\kappa(x)
\lesssim \frac{\|f\|_{H_{\kappa}}\|\psi\|_{L^{\infty}}}{\sqrt{N}}.
\end{eqnarray}

For $|x|>\sqrt{N}$, since
$$
|f(x)|\lesssim (\sqrt{N}+|x|)^{-2|\kappa|-d-1}
\asymp(\sqrt{N}+|x|)^{-1}|B_{\sqrt{N}+|x|}(0)|_{\kappa}^{-1},
$$
we have
\begin{eqnarray*}
\int_{|x|>\sqrt{N}}|f(x)|\,|[\tilde{\mathcal{R}}_j(\psi-\psi_N)](x)-b_N|\,d\omega_\kappa(x)
\lesssim \int_{\RR^d}\frac{|[\tilde{\mathcal{R}}_j(\psi-\psi_N)](x)-b_N|} {(\sqrt{N}+|x|)|B_{\sqrt{N}+|x|}(0)|_{\kappa}}\,d\omega_\kappa(x),
\end{eqnarray*}
which, appealing to (\ref{bmo-3}), is further controlled by a multiple of $\|\tilde{\mathcal{R}}_j(\psi-\psi_N)\|_{*,\kappa}/\sqrt{N}\lesssim\|\psi\|_{L^{\infty}}/\sqrt{N}$.
Combining this with (\ref{sym-riez-4}) proves (\ref{sym-riez-3}), and the proof of the lemma is completed.
\end{proof}

We now turn to the proof of Theorem \ref{Riesz-BMC-a}.

Assume that $\psi\in L^\infty(\RR^d)$. Obviously $\tilde{\mathcal{R}}_j\psi$ satisfies condition (i) in Definition \ref{bmc-a}, since $\tilde{\mathcal{R}}_j\psi\in\B_\kappa(\RR^d)$ by Theorem \ref{Riesz-BMO-a}, and Proposition \ref{bmo-a} is applicable to $\tilde{\mathcal{R}}_j\psi$, with $x'=0$. It remains to show $\|\tilde{\mathcal{R}}_j\psi\|_{*,C_{\kappa}}\lesssim \|\psi\|_{L^{\infty}}$.

For a fixed ball $B=B_{\delta}(x')$ in $\RR^d$, write $\psi=\psi_1+\psi_2$, where $\psi_1=\psi\chi_{{\mathcal{O}}(2B)}$ and $\psi_2=\psi\chi_{[{\mathcal{O}}(2B)]^c}$. Similarly to the first part of the proof of Proposition \ref{bmc-b},
\begin{align*}
\iint_{T(B)}d\nu_{\kappa,\tilde{\mathcal{R}}_j\psi_1}(x_0,x)\le \frac{1}{2}\int_{\RR^d}|{\mathcal{R}}_j\psi_1|^2d\omega_\kappa
\end{align*}
in view of $\tilde{\mathcal{R}}_j\psi_1={\mathcal{R}}_j\psi_1-a$ with some constant $a$, and further by (\ref{Riesz-2-1}),
\begin{align}\label{Carleson-4}
\iint_{T(B)}d\nu_{\kappa,\tilde{\mathcal{R}}_j\psi_1}(x_0,x)\le \frac{1}{2}\int_{\RR^d}|\psi_1|^2d\omega_\kappa\lesssim \|\psi\|_{L^{\infty}}^2|B|_{\kappa}.
\end{align}

In order to consider the contribution of $\psi_2$, we need the formula
\begin{align}\label{conjugate-kernel-7}
{\mathcal{R}}_j\left(\nabla_{\tilde{\kappa}}^{(x_0,x)}[(\tau_{x}P_{x_0})(-\cdot)]\right)(t)
=-\nabla_{\tilde{\kappa}}^{(x_0,x)}[(\tau_{x}Q^{(j)}_{x_0})(-t)],
\end{align}
which can be checked by taking the Dunkl transforms of the both sides. Indeed, from (\ref{translation-2-0}), Propositions \ref{translation-2-b-1}(ii) and \ref{conjugate-kernel-a}(ii) one has
\begin{align*}
\SF_{\kappa}[(\tau_{x}Q^{(j)}_{x_0})(-\cdot)](\xi)=-\SF_{\kappa}[(\tau_{-x}Q^{(j)}_{x_0})](\xi)
=\frac{i\xi_j}{|\xi|}e^{-x_0|\xi|}E_{\kappa}(-ix,\xi),
\end{align*}
so that
\begin{align}\label{conjugate-kernel-8}
\nabla_{\tilde{\kappa}}^{(x_0,x)}[(\tau_{x}Q^{(j)}_{x_0})(-t)]
=c_{\kappa}\int_{\RR^{d}}\frac{i\xi_j}{|\xi|}\nabla_{\tilde{\kappa}}^{(x_0,x)}[e^{-x_0|\xi|}E_{\kappa}(-ix,\xi)]\,
E_{\kappa}(it,\xi)\,d\omega_{\kappa}(\xi).
\end{align}
On the other hand, it follows from (\ref{Poisson-kernel-1}) that $\nabla_{\tilde{\kappa}}^{(x_0,x)}[e^{-x_0|\xi|}E_{\kappa}(-ix,\xi)]$ is the Dunkl transforms of $\nabla_{\tilde{\kappa}}^{(x_0,x)}[(\tau_{x}P_{x_0})(-\cdot)]$, and so, by (\ref{Riesz-2-1}) the right hand side of (\ref{conjugate-kernel-8}) is also the inverse Dunkl transform of $-{\mathcal{R}}_j\left(\nabla_{\tilde{\kappa}}^{(x_0,x)}[(\tau_{x}P_{x_0})(-\cdot)]\right)(t)$. Thus (\ref{conjugate-kernel-7}) is concluded.

In view of (\ref{conjugate-kernel-7}), the estimates (\ref{Poisson-ker-2}) and (\ref{conjugate-kernel-5}) imply that for fixed $(x_0,x)\in\RR^{1+d}_+$, each component of the vector-valued function  $t\mapsto\nabla_{\tilde{\kappa}}^{(x_0,x)}[(\tau_{x}P_{x_0})(-t)$ is in the Hardy space $H_{\kappa}^1(\RR^d)$ and satisfies the required decreasing condition in Lemma \ref{sym-riez}. Thus from (\ref{Poisson-2-2}), (\ref{sym-riez-1}) and (\ref{conjugate-kernel-7}),
\begin{align*}
\nabla_{\tilde{\kappa}} u_{\tilde{\mathcal{R}}_j\psi_2}(x_0,x)
&=c_{\kappa}\int_{\RR^{d}}(\tilde{\mathcal{R}}_j\psi_2)(t) \, \nabla_{\tilde{\kappa}}^{(x_0,x)}[(\tau_{x}P_{x_0})(-t)]\,d\omega_{\kappa}(t)\\
&=c_{\kappa}\int_{\RR^{d}}\psi_2(t) \, \nabla_{\tilde{\kappa}}^{(x_0,x)}[(\tau_{x}Q^{(j)}_{x_0})(-t)]\,d\omega_{\kappa}(t),
\end{align*}
and by use of Lemma \ref{conjugate-Poisson-b},
\begin{align*}
\left|\nabla_{\tilde{\kappa}} u_{\tilde{\mathcal{R}}_j\psi_2}(x_0,x)\right|
\lesssim \|\psi\|_{L^{\infty}}
\int_{[{\mathcal{O}}(2B)]^c}\frac{(x_0+d(t,x))^{-1}}{\left|\BB(x,x_0+d(t,x))\right|_{\kappa}}\,d\omega_{\kappa}(t).
\end{align*}
Since for $(x_0,x)\in T(B)$ and $t\in[{\mathcal{O}}(2B)]^c$, $d(t,x)\ge d(t,x')-|x'-x|\ge d(t,x')/2\ge\delta$, we immediately get, for $(x_0,x)\in T(B)$,
\begin{align*}
\left|\nabla_{\tilde{\kappa}} u_{\tilde{\mathcal{R}}_j\psi_2}(x_0,x)\right|
\lesssim \|\psi\|_{L^{\infty}}/\delta,
\end{align*}
which takes the role of (\ref{Carleson-3}), and consequently,
\begin{align*}
\iint_{T(B)}d\nu_{\kappa,\tilde{\mathcal{R}}_j\psi_2}(x_0,x)\lesssim \|\psi\|_{L^{\infty}}^2|B|_{\kappa}.
\end{align*}
Combining this and (\ref{Carleson-4}) asserts that $\tilde{\mathcal{R}}_j\psi\in\BC_{\kappa}(\RR^d)$ and $\|\tilde{\mathcal{R}}_j\psi\|_{*,C_{\kappa}}\lesssim \|\psi\|_{L^{\infty}}$, and the proof of Theorem \ref{Riesz-BMC-a} is finished.

\section{Duality of $H_{\kappa}^1(\RR^d)$ and $\BC_{\kappa}(\RR^d)$}

For a vector valued function $F=(u_0,u_1,\cdots,u_d)$ defined on $\RR^{1+d}_+$, we construct a $(d+1)|G|$-dimensional vector valued function $U_F(x_0,x)$ following the work
\cite{ADH1}. If $\{\sigma_{ij}\}_{i,j=1}^d$ denotes the matrix of the group element $\sigma\in G$ in the ordinary sense, we set $F_{\sigma}=(u_{\sigma,0},u_{\sigma,1},\cdots,u_{\sigma,d})$, where $u_{\sigma,0}(x_0,x)=u_0(x_0,\sigma(x))$ and
$$
u_{\sigma,j}(x_0,x)=\sum_{i=1}^d\sigma_{ij}u_i(x_0,\sigma(x)),\qquad j=1,\cdots,d.
$$
The vector valued function $U_F(x_0,x)$ is given by
\begin{align}\label{UF-function-1}
U_F(x_0,x)=\left\{F_{\sigma}(x_0,x)\right\}_{\sigma\in G}.
\end{align}

\begin{proposition}{\rm (\cite[Lemma 6.1]{ADH1})}
If $F=(u_0,u_1,\cdots,u_d)$ satisfies the generalized Cauchy-Riemann equations in (\ref{C-R-1}) on $\RR^{1+d}_+$, then for each $\sigma\in G$,  $F_{\sigma}$ also satisfies the equations in (\ref{C-R-1}), and $|F_{\sigma}(x_0,x)|=|F(x_0,\sigma(x))|$. Moreover $|U_F(x_0,\sigma(x))|=|U_F(x_0,x)|$ for $\sigma\in G$.
\end{proposition}

\begin{lemma}\label{subbarmonicity-a}
If $F=(u_0,u_1,\cdots,u_d)$ satisfies the generalized Cauchy-Riemann equations in (\ref{C-R-1}) on $\RR^{1+d}_+$, then there exists a constant $C>0$ independent of $F$ and $\epsilon>0$, such that
\begin{align}\label{subbarmonicity-1}
(|U_F|^2+\epsilon)^{-1/2}|\nabla_{\tilde{\kappa}}U_F|^2
\le C\Delta_{\tilde{\kappa}}(|U_F|^2+\epsilon)^{1/2}.
\end{align}
\end{lemma}

\begin{proof}
Direct calculations show that
\begin{align}\label{subbarmonicity-2}
\Delta_{\tilde{\kappa}}(|U_F|^2+\epsilon)^{1/2}=\frac{1}{(|U_F|^2+\epsilon)^{1/2}}\left(V+\sum_{j=0}^d|\partial_jU_F|^2
-\frac{1}{|U_F|^2+\epsilon}\sum_{j=0}^d\langle\partial_jU_F,U_F\rangle^2\right),
\end{align}
where
$$
V=\left\langle\sum_{j=0}^d\partial_j^2U_F+2\sum_{\alpha\in R_+}\frac{\kappa(\alpha)}{\langle\alpha,x\rangle}\sum_{j=1}^d\alpha_j\partial_jU_F,U_F\right\rangle.
$$
But by \cite[(6.11) and (6.12)]{ADH1}, $V$ can be rewritten as
\begin{align}\label{subbarmonicity-3}
V=\sum_{\alpha\in R_+}\frac{\kappa(\alpha)}{\langle\alpha,x\rangle^2}|U_F-\sigma_{\alpha}U_F|^2,
\end{align}
and there exists some $q\in(0,1)$ independent of $F$ such that
\begin{align}\label{subbarmonicity-4}
\sum_{j=0}^d\langle\partial_jU_F,U_F\rangle^2
\le\frac{|U_F|^2}{2-q}\left(V+\sum_{j=0}^d|\partial_jU_F|^2\right).
\end{align}
Combining (\ref{subbarmonicity-2}) and (\ref{subbarmonicity-4}) yields
\begin{align*}
\Delta_{\tilde{\kappa}}(|U_F|^2+\epsilon)^{1/2}
\ge\frac{1-q}{2-q}\frac{1}{(|U_F|^2+\epsilon)^{1/2}}\left(V+\sum_{j=0}^d|\partial_jU_F|^2\right),
\end{align*}
but from (\ref{subbarmonicity-3}) it is obvious that $V+\sum_{j=0}^d|\partial_jU_F|^2\ge C_1|\nabla_{\tilde{\kappa}}U_F|^2$ with $C_1>0$ independent of $F$. Thus (\ref{subbarmonicity-1}) is proved for $C=(2-q)/C_1(1-q)$.
\end{proof}

\begin{lemma}\label{gradient-product-a}
If $f\in H_{\kappa,0}^1(\RR^d)$ and $\varphi\in \BC_\kappa(\RR^d)$, then
\begin{eqnarray}\label{gradient-product-1}
\iint_{\RR^{1+d}_+}x_0\left|\big\langle\nabla_{\tilde{\kappa}} u_f(x_0,x),\nabla_{\tilde{\kappa}} u_{\varphi}(x_0,x)\big\rangle\right| d\omega_\kappa(x)dx_0
\lesssim\|\varphi\|_{*,C_{\kappa}}\|f\|_{H_{\kappa}^1},
\end{eqnarray}
where $u_f(x_0,x)$ and $u_\varphi(x_0,x)$ are the $\kappa$-Poisson integrals of $f$ and $\varphi$ respectively.
\end{lemma}

\begin{proof}
By Corollary \ref{den-aug-2}, $u_f(x_0,x)$ is the first component of some $F\in\SH_{\kappa,0}^1(\RR^{1+d}_+)$, and by Theorem \ref{Hardy-thm-2},
$\|f\|_{H_{\kappa}^1}\asymp\|F\|_{\SH_{\kappa}^1}$.

For $M>0$, let $B_M^+:=B^+((0,0),M)$ denote the upper half ball with center $(0,0)$ and radius $M>0$ on $\RR^{1+d}$. For $\epsilon>0$ we have
\begin{align}\label{H-BMO-5-1}
\iint_{B_M^+}x_0\left|\big\langle\nabla_{\tilde{\kappa}} u_f,\nabla_{\tilde{\kappa}} u_{\varphi}\big\rangle\right| d\omega_\kappa(x)dx_0
\le & \iint_{B_M^+}x_0 |\nabla_{\tilde{\kappa}}U_F||\nabla_{\tilde{\kappa}} u_{\varphi}|\, d\omega_\kappa(x)dx_0 \nonumber\\
\le & J_{1}^{1/2}\times J_{2}^{1/2},
\end{align}
where $U_F(x_0,x)$ is given by (\ref{UF-function-1}), and
\begin{align*}
J_{1}=& \iint_{B_M^+}x_0 (|U_F|^2+\epsilon)^{-1/2}|\nabla_{\tilde{\kappa}}U_F|^2\, d\omega_\kappa(x)dx_0,\\
J_{2}=& \iint_{B_M^+}x_0 (|U_F|^2+\epsilon)^{1/2}|\nabla_{\tilde{\kappa}} u_{\varphi}|^2 \, d\omega_\kappa(x)dx_0.
\end{align*}

By Lemma \ref{subbarmonicity-a},
\begin{align*}
J_{1}\lesssim \iint_{B_M^+}x_0 \Delta_{\tilde{\kappa}}(|U_F|^2+\epsilon)^{1/2}\, d\omega_\kappa(x)dx_0,
\end{align*}
and applying the Green formula (\ref{Green-formula-2-1}) we have
\begin{align*}
J_{1}\lesssim & \int_{\partial B_M^+}\left(x_0\partial_{\mathbf{n}}(|U_F|^2+\epsilon)^{1/2}
-(|U_F|^2+\epsilon)^{1/2}\partial_{\mathbf{n}}x_0\right)W_{\kappa}(x)d\sigma(x_0,x),
\end{align*}
where $\partial_{\mathbf{n}}$ denotes the directional derivative of the outward normal, and $d\sigma(x_0,x)$ the area element on $\partial\Omega$.
Since $|\partial_j(|U_F|^2+\epsilon)^{1/2}|=|\langle\partial_jU_F,U_F\rangle|/(|U_F|^2+\epsilon)^{1/2}|
\le|\partial_jU_F|$, it follows that
\begin{align*}
J_{1}\lesssim & \int_{|x|\le M}(|U_F(0,x)|^2+\epsilon)^{1/2}d\omega_\kappa(x)\\
&\qquad +\int_{|(x_0,x)|=M,\,x_0>0}\left(x_0|\nabla U_F|
+(|U_F|^2+\epsilon)^{1/2}\right)W_{\kappa}(x)d\sigma(x_0,x).
\end{align*}
Furthermore, for $F\in\SH_{\kappa,0}^1(\RR^{1+d}_+)$, one has $|U_F|\lesssim M^{-2|\kappa|-d-1}$ and $|\nabla U_F|\lesssim M^{-2|\kappa|-d-2}$ when $|(x_0,x)|=M$ and $x_0>0$, so that
\begin{align*}
J_{1}\lesssim \int_{|x|\le M}(|U_F(0,x)|^2+\epsilon)^{1/2}d\omega_\kappa(x)
 +\frac{1}{M}+\epsilon M^{2|\kappa|+d}.
\end{align*}
Inserting this into (\ref{H-BMO-5-1}), letting $\epsilon\rightarrow0+$ and then $M\rightarrow\infty$, we obtain
\begin{align}\label{H-BMO-5-2}
\iint_{\RR_+^{1+d}}x_0\left|\big\langle\nabla_{\tilde{\kappa}} u_f,\nabla_{\tilde{\kappa}} u_{\varphi}\big\rangle\right| d\omega_\kappa(x)dx_0
\lesssim \|f\|_{H_{\kappa}^1}^{1/2}\, J_3^{1/2},
\end{align}
where
\begin{align}\label{J-3}
J_3=\iint_{\RR_+^{1+d}}x_0 |U_F||\nabla_{\tilde{\kappa}} u_{\varphi}|^2 \, d\omega_\kappa(x)dx_0.
\end{align}

We now construct a $\kappa$-harmonic majorization of $|U_F|^q$. To apply the maximum principle \cite[Theorem 4.2]{Ro2} in our case, we give a variant of the $\kappa$-subharmonic inequality in \cite[Theorem 6.3]{ADH1}.
Similarly to (\ref{subbarmonicity-2}), for $q>0$ and $\epsilon>0$ we have
\begin{align*}
\Delta_{\tilde{\kappa}}(|U_F|^2+\epsilon)^{q/2}=\frac{q}{(|U_F|^2+\epsilon)^{1-q/2}}\left(V+\sum_{j=0}^d|\partial_jU_F|^2
+\frac{q-2}{|U_F|^2+\epsilon}\sum_{j=0}^d\langle\partial_jU_F,U_F\rangle^2\right),
\end{align*}
and choose $q\in(0,1)$ satisfying (\ref{subbarmonicity-4}) (see \cite[(6.12)]{ADH1}), independent of $F$ and $\epsilon$, to conclude
$$
\Delta_{\tilde{\kappa}}(|U_F|^2+\epsilon)^{q/2}\ge0.
$$

For $\delta,\epsilon>0$ let $f_{\delta,\epsilon}(x)=(|U_F(\delta,x)|^2+\epsilon)^{q/2}$ and set $u_{f_{\delta,\epsilon}}(x_0,x)=(Pf_{\delta,\epsilon})(x_0,x)$, the $\kappa$-Poisson integral of $f_{\delta,\epsilon}$.
For $F\in\SH_{\kappa,0}^1(\RR^{1+d}_+)$, we have $|U_F(x_0,x)|\lesssim (1+x_0+|x|)^{-2|\kappa|-d-2}$, and $u_{f_{\delta,\epsilon}}(x_0,x)$ is a continuous extension to $\overline{\RR_+^{1+d}}$ of $f_{\delta,\epsilon}(x)$.
Consider the function
$$
V_{\delta,\epsilon}(x_0,x)=(|U_F(\delta+x_0,x)|^2+\epsilon)^{q/2}-u_{f_{\delta,\epsilon}}(x_0,x).
$$
Obviously $V_{\delta,\epsilon}(x_0,x)$ is continuous on $\overline{\RR_+^{1+d}}$ and $\Delta_{\tilde{\kappa}}V_{\delta,\epsilon}(x_0,x)\ge0$ in $\RR_+^{1+d}$; moreover $V_{\delta,\epsilon}(0,x)=0$ for $x\in\RR^d$.

For given $\eta>0$, choose $M>0$ such that
\begin{align}\label{H-BMO-5-3}
|U_F(x_0,x)|<\eta\qquad \hbox{whenever}\quad |(x_0,x)|\ge M.
\end{align}
It then follows that $|f_{\delta,\epsilon}(x)|\lesssim\eta+\epsilon$ for $|x|\ge M$.

Now for $|(x_0,x)|\ge2M$, from (\ref{Poisson-2-2}) and (\ref{Poisson-ker-bound-1}) we have
\begin{align*}
|u_{f_{\delta,\epsilon}}(x_0,x)|\lesssim
\eta+\epsilon+\int_{|t|\le M}\frac{d\omega_{\kappa}(t)}{\left|B(x,x_0+d(x,t))\right|_{\kappa}};
\end{align*}
but $x_0+d(x,t)\ge x_0+|x|-|t|\ge|(x_0,x)|-|t|\ge|(x_0,x)|/2$ for $|t|\le M$, so that
\begin{align*}
|u_{f_{\delta,\epsilon}}(x_0,x)|\lesssim
\eta+\epsilon+\left(\frac{M}{|(x_0,x)|}\right)^{2|\kappa|+d}.
\end{align*}
Choose $M'>2M$ so that $(M/M')^{2|\kappa|+d}<\eta$, and then for $|(x_0,x)|\ge M'$, $|u_{f_{\delta,\epsilon}}(x_0,x)|\lesssim\eta+\epsilon$.
From this and (\ref{H-BMO-5-3}) we get
\begin{align*}
|V_{\delta,\epsilon}(x_0,x)|\lesssim\eta+\epsilon\qquad \hbox{whenever}\quad |(x_0,x)|\ge M'.
\end{align*}
We apply the maximum principle \cite[Theorem 4.2]{Ro2} to $V_{\delta,\epsilon}(x_0,x)$ on the upper half ball $=B^+((0,0),M')$, and then obtain $|V_{\delta,\epsilon}(x_0,x)|\lesssim\eta+\epsilon$ for all $(x_0,x)\in\RR_+^{1+d}$. Letting $\eta\rightarrow0+$ gives
\begin{align}\label{H-BMO-5-4}
(|U_F(\delta+x_0,x)|^2+\epsilon)^{q/2}-u_{f_{\delta,\epsilon}}(x_0,x)\lesssim\epsilon\qquad \hbox{for all}\quad (x_0,x)\in\RR_+^{1+d}.
\end{align}

It is easy to see that $f_{\delta,\epsilon}(x)=(|U_F(\delta,x)|^2+\epsilon)^{q/2}$ converges to $|U_F(\delta,x)|^q$ uniformly for $x\in\RR^d$ as $\epsilon\rightarrow0+$, and so
$\lim_{\epsilon\rightarrow0+}u_{f_{\delta,\epsilon}}(x_0,x)=u_{|U_F(\delta,\cdot)|^q}(x_0,x)$ for $(x_0,x)\in\RR_+^{1+d}$. This and (\ref{H-BMO-5-4}) lead to
\begin{align}\label{H-BMO-5-5}
|U_F(\delta+x_0,x)|^q\le u_{|U_F(\delta,\cdot)|^q}(x_0,x)\qquad \hbox{for all}\quad (x_0,x)\in\RR_+^{1+d}.
\end{align}
Since $|U_F(\delta,\cdot)|^q\in L^{1/q}_{\kappa}(\RR^d)$ with $1/q>1$ and $\sup_{\delta>0}\||U_F(\delta,\cdot)|^q\|_{L^{1/q}_{\kappa}}\lesssim\|f\|_{H_{\kappa}^1}^q$.
The *-weak compactness of $L^{1/q}_{\kappa}(\RR^d)$ implies
that, there exists a sequence $\delta_n\rightarrow0+$ such that
$U_F(\delta_n,\cdot)$ converges *-weakly to a nonnegative function $g\in L^{1/q}_{\kappa}(\RR^d)$. Thus from (\ref{H-BMO-5-5}) we get
\begin{align*}
|U_F(x_0,x)|^q\le u_{g}(x_0,x)\qquad \hbox{for all}\quad (x_0,x)\in\RR_+^{1+d}.
\end{align*}
Now since $\|u_{|U_F(\delta_n,\cdot)|^q}(x_0,\cdot)\|_{L^{1/q}_{\kappa}}\le\||U_F(\delta_n,\cdot)|^q\|_{L^{1/q}_{\kappa}} \lesssim\|f\|_{H_{\kappa}^1}^q$, Fatou's lemma asserts that $\|u_{g}(x_0,\cdot)\|_{L^{1/q}_{\kappa}}\lesssim\|f\|_{H_{\kappa}^1}^q$, and letting $x_0\rightarrow0+$ gives $\|g\|_{L^{1/q}_{\kappa}}\lesssim\|f\|_{H_{\kappa}^1}^q$.

Finally from (\ref{J-3}) we have
\begin{align*}
J_3\le\iint_{\RR_+^{1+d}}x_0 u_{g}(x_0,x)^{1/q}|\nabla_{\tilde{\kappa}} u_{\varphi}(x_0,x)|^2 \, d\omega_\kappa(x)dx_0,
\end{align*}
and so, by Theorems \ref{car-pro-1}, $J_3\lesssim\|d\nu_{\kappa,\varphi}\|_{\kappa}\|g\|_{L^{1/q}_{\kappa}}^{1/q} \lesssim\|\varphi\|_{*,C_{\kappa}}^2\|f\|_{H_{\kappa}^1}$. Substituting this into (\ref{H-BMO-5-2}) proves (\ref{gradient-product-1}), and then, the proof of the lemma is completed.
\end{proof}

\begin{lemma}\label{inner-product-repres-a}
If $f\in H_{\kappa,0}^1(\RR^d)$ and $\varphi\in \BC_\kappa(\RR^d)$, then
\begin{eqnarray}\label{inner-product-repres-1}
\int_{\RR^d}f(x)\varphi(x)d\omega_\kappa(x)=2\iint_{\RR^{1+d}_+}x_0\big\langle\nabla_{\tilde{\kappa}} u_f(x_0,x),\nabla_{\tilde{\kappa}} u_{\varphi}(x_0,x)\big\rangle d\omega_\kappa(x)dx_0,
\end{eqnarray}
where $u_f(x_0,x)$ and $u_\varphi(x_0,x)$ are the $\kappa$-Poisson integrals of $f$ and $\varphi$ respectively.
\end{lemma}

\begin{proof}
From (\ref{H-1-BMO-1-1}) it is obvious that the integral on the left hand side of (\ref{inner-product-repres-1}) converges absolutely, and so does the one on the right hand side by Lemma \ref{gradient-product-a}.

For $x_0>0$ we have
\begin{align*}
\nabla_{\tilde{\kappa}} u_{\varphi}(x_0,x)=c_{\kappa}\int_{\RR^d}\varphi(t) \nabla_{\tilde{\kappa}}^{(x_0,x)} [(\tau_{x}P_{x_0})(-t)]\,d\omega_{\kappa}(t)
\end{align*}
which is legitimate on account of (\ref{Poisson-ker-bound-1}) and (\ref{Poisson-ker-bound-2}). By use of (\ref{Poisson-ker-2}),
\begin{align*}
\left|\nabla_{\tilde{\kappa}}^{(x_0,x)}[(\tau_{x}P_{x_0})(-t)]\right|
\lesssim \left(1+\frac{1+|x|}{x_0}\right)^{2|\kappa|+d+1}\frac{1}{(1+|t|)^{2|\kappa|+d+1}},
\end{align*}
so that, in view of (\ref{H-1-BMO-1-1}),
\begin{align*}
\int_{\RR^d}|\varphi(t)| \left|\nabla_{\tilde{\kappa}}^{(x_0,x)} [(\tau_{x}P_{x_0})(-t)]\right|\,d\omega_{\kappa}(t)
\lesssim \left(1+\frac{1+|x|}{x_0}\right)^{2|\kappa|+d+1}.
\end{align*}
Moreover for $f\in H_{\kappa,0}^1(\RR^d)$, $\int_{\RR^{d}}(1+x_0+|x|)^{2|\kappa|+d+1}\left|\nabla_{\tilde{\kappa}} u_f(x_0,x)\right| d\omega_\kappa(x)$ is finite.
Now Fubini's theorem yields
\begin{align}\label{inner-product-repres-2}
&\int_{\RR^{d}}\big\langle\nabla_{\tilde{\kappa}} u_f(x_0,x),\nabla_{\tilde{\kappa}} u_{\varphi}(x_0,x)\big\rangle d\omega_\kappa(x) \nonumber\\
=& c_{\kappa}\int_{\RR^d}\varphi(t) \int_{\RR^d}\big\langle\nabla_{\tilde{\kappa}} u_f(x_0,x),\nabla_{\tilde{\kappa}}^{(x_0,x)} [(\tau_{x}P_{x_0})(-t)]\big\rangle d\omega_\kappa(x)d\omega_{\kappa}(t).
\end{align}

By Proposition \ref{transform-a}(iii) and (\ref{Poisson-kernel-1}), the Dunkl transform of $x\mapsto\nabla_{\tilde{\kappa}}^{(x_0,x)} [(\tau_{x}P_{x_0})(-t)]$ is
$$
e^{-x_0|\xi|}E_{\kappa}(-it,\xi)(-|\xi|,i\xi_1,\cdots,i\xi_d),
$$
and by (\ref{Dunkl-transform-1}) and (\ref{Dunkl-transform-2}), $\left[\SF_{\kappa}(\nabla_{\tilde{\kappa}} u_f(x_0,\cdot))\right](\xi)=e^{-x_0|\xi|}\left(\SF_{\kappa}f\right)(\xi)(-|\xi|,i\xi_1,\cdots,i\xi_d)$.
Thus the Plancherel formula (Propositions \ref{transform-a}(v)) implies that the inner integral on the right hand side of (\ref{inner-product-repres-2}) becomes
\begin{align}\label{inner-product-repres-3}
2\int_{\RR^{d}}|\xi|^2 e^{-2x_0|\xi|}\left(\SF_{\kappa}f\right)(\xi)E_{\kappa}(it,\xi) d\omega_\kappa(\xi)
=\frac{1}{2}c_{\kappa}^{-1}\partial_{x_0}^2[u_{f}(2x_0,t)].
\end{align}
In addition, by Lemma \ref{gradient-product-a} the integral on the right hand side of (\ref{inner-product-repres-1}) converges absolutely, and then, Fubini's theorem again gives
\begin{align*}
\iint_{\RR_+^{1+d}}x_0\big\langle\nabla_{\tilde{\kappa}} u_f,\nabla_{\tilde{\kappa}} u_{\varphi}\big\rangle\, dx_0d\omega_\kappa(x)
=\frac{1}{2}\int_0^{\infty}x_0\int_{\RR^d}\varphi(t)\partial_{x_0}^2[u_{f}(2x_0,t)]\,d\omega_{\kappa}(t)dx_0.
\end{align*}
Noting that $|\Delta_{\kappa}^t [u_{f}(2x_0,t)]|\lesssim (1+x_0+|t|)^{-2|\kappa|-d-3}$ and in view of (\ref{H-1-BMO-1-1}), further changing order of the integration above is permitted, and hence
\begin{align}\label{inner-product-repres-4}
\iint_{\RR_+^{1+d}}x_0\big\langle\nabla_{\tilde{\kappa}} u_f,\nabla_{\tilde{\kappa}} u_{\varphi}\big\rangle d\omega_\kappa(x)
=\frac{1}{2}\int_{\RR^d}\varphi(t)\int_0^{\infty}x_0\partial_{x_0}^2[u_{f}(2x_0,t)]\,dx_0 d\omega_{\kappa}(t).
\end{align}
Finally, since $f\in H_{\kappa,0}^1(\RR^d)$ it follows from (\ref{inner-product-repres-3}) that
$$
\int_0^{\infty}x_0\partial_{x_0}^2[u_{f}(2x_0,t)]\,dx_0
=f(t),
$$
and inserting this into (\ref{inner-product-repres-4}) finishes the proof of (\ref{inner-product-repres-1}).
\end{proof}

\begin{theorem}\label{H-1-BMO}
If $\varphi\in \BC_\kappa(\RR^d)$, then the linear functional
\begin{align}\label{bmo-linear-function-1}
{\mathcal{L}}f=\int_{\RR^d}f(x)\varphi(x)d\omega_\kappa(x),
\end{align}
initially defined for $f\in H_{\kappa,0}^1(\RR^d)$, has a unique bounded extension to $H^{1}_\kappa(\RR^d)$, and $\|{\mathcal{L}}\|\lesssim \|\varphi\|_{*,C_{\kappa}}$, where the constant $C>0$ is independent of $\varphi$.
\end{theorem}

\begin{proof}
It suffices to show that, for $f\in H_{\kappa,0}^1(\RR^d)$,
$$
\Big|\int_{\RR^d}f(x)\varphi(x)d\omega_{\kappa}(x)\Big| \lesssim\|\varphi\|_{*,C_{\kappa}}\|f\|_{H_{\kappa}^1},
$$
but that is direct consequence of Lemmas \ref{gradient-product-a} and \ref{inner-product-repres-a}. The proof is completed.
\end{proof}

\begin{theorem}\label{H-1-BMO}
If $\mathcal{L}$ is a continuous linear functional on $H^{1}_\kappa(\RR^d)$, then there exists a unique element $\varphi\in\BC_\kappa(\RR^d)$ so that $\mathcal{L}$ can be realized by (\ref{bmo-linear-function-1}),
at least for $f\in H_{\kappa,0}^1(\RR^d)$, and  $\|\varphi\|_{*,C_{\kappa}}\le C'\|{\mathcal{L}}\|$, where the constant $C'>0$ is independent of $\mathcal{L}$.
Moreover, such a function $\varphi\in\BC_\kappa(\RR^d)$ can be expressed as
$$
\varphi=\varphi_0+\sum_{j=1}^d \tilde{\mathcal{R}}_j\varphi_j,
$$
where  $\varphi_0,\varphi_1,\cdot\cdot\cdot,\varphi_d\in L^\infty(\RR^d)$.
\end{theorem}

\begin{proof}
As in \cite[p. 145]{FS1}, we regard $H^{1}_\kappa(\RR^d)$ as a closed subspace of the direct sum $X:=L^{1}_\kappa\oplus L^{1}_\kappa\oplus\cdots
\oplus L^{1}_\kappa$ of $d+1$ copies of $L^{1}_\kappa(\RR^d)$. If we set
$X=\{(f_0,f_1,\cdots,f_d):\,\, f_j\in L^{1}_\kappa(\RR^d),\,\,j=0,1,\cdots d\}$ with the norm $\|(f_0,f_1,\cdots,f_d)\|=\sum_{j=1}^d\|f_j\|_{L^{1}_\kappa}$, it is known that the dual of $X$ is equivalent to $L^{\infty}\oplus L^{\infty}\oplus\cdots
\oplus L^{\infty}$. We consider the subspace $X_0$ of $X$ given by
$$
X_0=\{(f,{\mathcal{R}}_1f,\cdots,{\mathcal{R}}_df):\,\, f\in L^{1}_\kappa(\RR^d)\,\,\hbox{and}\,\,{\mathcal{R}}_jf\in L^{1}_\kappa(\RR^d),\,\,j=1,\cdots d\}.
$$
It is easy to see that $X_0$ is a closed subspace of $X$, and there is a natural isometric mapping of $H^{1}_\kappa(\RR^d)$ to $X_0$, that is, $f\mapsto(f,{\mathcal{R}}_1f,\cdots,{\mathcal{R}}_df)$. Thus if $\mathcal{L}$ is a continuous linear functional on $H^{1}_\kappa(\RR^d)$, the one on $X_0$ identified by it, by Hahn-Banach theorem, extends to a continuous linear functional on $X$, denoted by $\mathcal{L}$ too. Therefore, there exist $\varphi_0,\varphi_1,\cdot\cdot\cdot,\varphi_d\in L^\infty(\RR^d)$ such that
$$
{\mathcal{L}}(f_0,f_1,\cdots,f_d)=\sum_{j=0}^d\int_{\RR^d}f_j\varphi_jd\omega_\kappa,\qquad (f_0,f_1,\cdots,f_d)\in X,
$$
and $\|{\mathcal{L}}\|\asymp\sum_{j=0}^{d}\|\varphi_j\|_{L^{\infty}}$.
In particular, such a representation holds if $(f,{\mathcal{R}}_1f,\cdots,{\mathcal{R}}_df) \in X_0$, or equivalently, for $f\in H^{1}_\kappa(\RR^d)$,
${\mathcal{L}}f=\int_{\RR^d}\varphi_0f\,d\omega_\kappa+\sum_{j=1}^d\int_{\RR^d}\varphi_j{\mathcal{R}}_jf\,d\omega_\kappa$.
If, furthermore, we apply this to $f\in H_{\kappa,0}^1(\RR^d)$, and then appeal to Lemma \ref{sym-riez}, to get
$$
{\mathcal{L}}f=\int_{\RR^d} f\left(\varphi_0-\sum_{j=1}^d\tilde{\mathcal{R}}_j\varphi_j\right)\,d\omega_\kappa.
$$
According to Theorem \ref{Riesz-BMC-a}, the function $\varphi=\varphi_0-\sum_{j=1}^d\tilde{\mathcal{R}}_j\varphi_j$ is in $\BC_\kappa(\RR^d)$ and
$\|\varphi\|_{*,C_{\kappa}}\lesssim\sum_{j=0}^{d}\|\varphi_j\|_{L^{\infty}}\asymp\|{\mathcal{L}}\|$. The uniqueness is obvious, and the proof of the theorem is completed.
\end{proof}

Finally we summarize the conclusions in the above theorems as follows.

\begin{theorem}\label{H-1-BMO-b}
The dual of $H_{\kappa}^1(\RR^d)$ is $\BC_\kappa(\RR^d)$, and the norm of a continuous linear functional $\mathcal{L}$ on $H^{1}_\kappa(\RR^d)$, realized by (\ref{bmo-linear-function-1}), is equivalent to the norm of the representing function
$\varphi\in\BC_\kappa(\RR^d)$.
\end{theorem}

\begin{theorem}\label{H-1-BMO-1-b}
A function $\varphi\in L_{\kappa,{\rm loc}}(\RR^d)$ is in $\BC_{\kappa}(\RR^d)$ if and only if it has the representation
$$
\varphi=\varphi_0+\sum_{j=1}^d \tilde{\mathcal{R}}_j\varphi_j, \qquad \hbox{where}\quad \varphi_0,\varphi_1,\cdots,\varphi_d\in L^\infty(\RR^d);
$$
and moreover, $\|\varphi\|_{*,C_{\kappa}}\asymp\sum_{j=0}^{d}\|\varphi_j\|_{L^{\infty}}$.
\end{theorem}

\section{Remarks on the spaces $\BC_\kappa(\RR^d)$, $\B_\kappa(\RR^d)$ and $H^{1}_\kappa(\RR^d)$}

\subsection{$\BC_\kappa(\RR^d)$ and $\B_\kappa(\RR^d)$}

By Theorems \ref{Riesz-BMO-a}, \ref{H-1-BMO-b} and \ref{H-1-BMO-1-b} we have the following corollary.

\begin{corollary}\label{bmo-bmc-1}
The space $\BC_{\kappa}(\RR^d)$ is complete, and is a subclass of $\B_{\kappa}(\RR^d)$; and moreover, for $\varphi\in \BC_{\kappa}(\RR^d)$, $\|\varphi\|_{*,\kappa}\lesssim\|\varphi\|_{*,C_{\kappa}}$.
\end{corollary}

In order to find more connections between the two spaces, we introduce a variant of the weighted $\B$ space according to the $G$-orbits of balls in $\RR^d$.

\begin{definition}\label{bmo-G-a}
The class $\B_G(\RR^d)$ consists of functions $f\in L_{\kappa,{\rm loc}}(\RR^d)$ satisfying
\begin{eqnarray*}
\|f\|_{*,G}:=\sup_{B}\frac{1}{|{\mathcal{O}}(B)|_\kappa}\int_{{\mathcal{O}}(B)}|f-f_{{\mathcal{O}}(B)}|\,d\omega_{\kappa}<\infty,
\end{eqnarray*}
where the supremum ranges over all balls $B$ in $\RR^d$, and $f_{{\mathcal{O}}(B)}=|{\mathcal{O}}(B)|_{\kappa}^{-1}\int_{{\mathcal{O}}(B)}f\,d\omega_\kappa$.
\end{definition}

As usual, $f\in \B_G(\RR^d)$ represents an element of the quotient space of $\B_G(\RR^d)$ modulo constants.
We have the following proposition, whose proof is similar to that of Proposition \ref{bmo-a}.

\begin{proposition}\label{bmo-G-b}
There exists some $C>0$ such that for all $f\in\B_G(\RR^d)$, $x'\in\RR^d$ and for $\delta>0$,
\begin{eqnarray*}
\int_{\RR^d}\frac{|f(x)-f_{{\mathcal{O}}(B_{\delta}(x'))}|}{(\delta+d(x,x'))|B_{\delta+d(x,x')}(x')|_\kappa}
\,d\omega_\kappa(x)
\leq\frac{C}{\delta}\|f\|_{*,G}.
\end{eqnarray*}
\end{proposition}

The analog of the John-Nirenberg inequality is
\begin{align*}
\left|\left\{x\in B:|f(x)-f_{{\mathcal{O}}(B)}|>\lambda\right\}\right|_\kappa \leq c_1 e^{-c_0\lambda/\|f\|_{*,G}}|B|_\kappa,\qquad \lambda>0,
 \end{align*}
for $f\in\B_G(\RR^d)$ and all balls $B$ in $\RR^d$, where the constants $c_0,c_1>0$ are independent of $f$ and $B$.
A direct consequence of this inequality is that, for $f\in\B_G(\RR^d)$,
\begin{align*}
\|f\|_{*,G}\asymp\sup_{B}\left(\frac{1}{|{\mathcal{O}}(B)|_\kappa}\int_{{\mathcal{O}}(B)}|f-f_{{\mathcal{O}}(B)}|^2\,d\omega_{\kappa}\right)^{1/2}.
\end{align*}

The next proposition indicates partial relations of $\B_G(\RR^d)$, $\BC_{\kappa}(\RR^d)$, and $\B_{\kappa}(\RR^d)$.

\begin{proposition}\label{bmo-G-c}
{\rm (i)} The inclusions
\begin{align}\label{bmo-G-2}
\B_G(\RR^d)\subseteq\BC_{\kappa}(\RR^d)\subseteq\B_{\kappa}(\RR^d)
\end{align}
hold, and for $f\in \B_G(\RR^d)$, $\|f\|_{*,\kappa}\lesssim\|f\|_{*,C_{\kappa}}\lesssim\|f\|_{*,G}$;

{\rm (ii)} if $f\in \B_{\kappa}(\RR^d)$ is $G$-invariant, then $f\in\B_G(\RR^d)$ and $\|f\|_{*,\kappa}\asymp\|f\|_{*,C_{\kappa}}\asymp\|f\|_{*,G}$.
\end{proposition}

\begin{proof}
(i) By Corollary \ref{bmo-bmc-1} it remains to show the first inclusion in (\ref{bmo-G-2}). For $f\in\B_G(\RR^d)$ and for a fixed ball $B=B_{\delta}(x')$ in $\RR^d$,
write $f=f_{{\mathcal{O}}(2B)}+f_1+f_2$, where $f_1=(f-f_{{\mathcal{O}}(2B)})\chi_{{\mathcal{O}}(2B)}$ and $f_2=(f-f_{{\mathcal{O}}(2B)})\chi_{[{\mathcal{O}}(2B)]^c}$.
Similarly to the first part of the proof of Proposition \ref{bmc-b}, one has
\begin{align*}
\iint_{T(B)}d\nu_{\kappa,f_1}(x_0,x)\le \iint_{\RR^{1+d}_+}x_0|\nabla_{\tilde{\kappa}} u_{f_1}(x_0,x)|^2\,dx_0d\omega_\kappa(x)
\le\frac{1}{2}\int_{\RR^d}|f_1|^2d\omega_\kappa\lesssim \|f\|_{*,G}^2|B|_{\kappa}.
\end{align*}
From (\ref{Poisson-2-2}) and (\ref{Poisson-ker-2}) we get
\begin{align*}
|\nabla_{\tilde{\kappa}} u_{f_2}(x_0,x)|
\lesssim & \int_{[{\mathcal{O}}(2B)]^c}\frac{|f(t)-f_{{\mathcal{O}}(2B)}|}{(x_0+d(t,x))\left|B(x,x_0+d(t,x))\right|_{\kappa}}\,d\omega_{\kappa}(t);
\end{align*}
and since for $(x_0,x)\in T(B)$ and $t\in[{\mathcal{O}}(2B)]^c$, $d(t,x)\ge d(t,x')/2\ge\delta$, it follows that
\begin{align*}
|\nabla_{\tilde{\kappa}} u_{f_2}(x_0,x)|
\lesssim & \int_{\RR^d}\frac{|f(t)-f_{{\mathcal{O}}(2B)}|}{(\delta+d(t,x'))\left|B(x,\delta+d(t,x'))\right|_{\kappa}}\,d\omega_{\kappa}(t)
\lesssim\frac{1}{\delta}\|f\|_{*,G}
\end{align*}
by Proposition \ref{bmo-G-b}. Thus
$$
\iint_{T(B)}d\nu_{\kappa,f_2}(x_0,x)
\lesssim \frac{\|f\|_{*,G}^2}{\delta^2}\iint_{T(B)}x_0\,dx_0d\omega_\kappa(x)
\lesssim\|f\|_{*,G}^2|B|_{\kappa}.
$$
Combining this with the contribution of $f_1$ concludes $f\in\BC_{\kappa}(\RR^d)$ and $\|f\|_{*,C_{\kappa}}\lesssim\|f\|_{*,G}$.

(ii) If $f\in \B_{\kappa}(\RR^d)$ is $G$-invariant, for a ball $B=B_{\delta}(x')$ we have
\begin{align*}
\frac{1}{|{\mathcal{O}}(B)|_\kappa}\int_{{\mathcal{O}}(B)}|f-f_B|\,d\omega_{\kappa}
\le\frac{1}{|B|_\kappa}\sum_{\sigma\in G}\int_{\sigma(B)}|f-f_B|\,d\omega_{\kappa}
\le |G|\|f\|_{*,\kappa},
\end{align*}
so that
\begin{align*}
\frac{1}{|{\mathcal{O}}(B)|_\kappa}\int_{{\mathcal{O}}(B)}|f-f_{{\mathcal{O}}(B)}|\,d\omega_{\kappa}
\le\frac{2}{|{\mathcal{O}}(B)|_\kappa}\int_{{\mathcal{O}}(B)}|f-f_B|\,d\omega_{\kappa}
\le 2|G|\|f\|_{*,\kappa}.
\end{align*}
Thus $f\in\B_G(\RR^d)$ and $\|f\|_{*,G}\le 2|G|\|f\|_{*,\kappa}$. The proof of the proposition is completed.
\end{proof}

\subsection{An example in the rank-one case}

Now we give an example in the rank-one case, with $G=\ZZ_2$, that belongs to $\B_{\kappa}(\RR^1)$, but not to $\B_{\ZZ_2}(\RR^1)$. In this case $d\omega_{\kappa}(x)=|x|^{2\kappa}dx$, and $B_r(x')=(x'-r,x'+r)$ for $x'\in\RR^1$ and $r>0$; the single $\kappa$-Riesz transform is called the $\kappa$-Hilbert transform in \cite{LL1} and here we denote it by $\tilde{\mathcal{R}}$ or ${\mathcal{R}}$.
We need the following lemma.

\begin{lemma}\label{bmo-G-d}
If $f\in\B_{\ZZ_2}(\RR^1)$ is odd on $\RR^1$, then $f\in L^{\infty}(\RR^1)$.
\end{lemma}

\begin{proof}
For  $x'\in\RR^1$ and $r>0$, one has $\int_{{\mathcal{O}}(B_r(x'))}f\,d\omega_{\kappa}=0$ and so
\begin{align*}
\int_{{\mathcal{O}}(B_r(x'))}|f-f_{{\mathcal{O}}(B_r(x'))}|\,d\omega_{\kappa}
=\int_{{\mathcal{O}}(B_r(x'))}|f|\,d\omega_{\kappa}
\ge\int_{B_r(x')}|f|\,d\omega_{\kappa}.
\end{align*}
Thus
$|B_r(x')|_\kappa^{-1}\int_{B_r(x')}|f|\,d\omega_{\kappa}\lesssim\|f\|_{*,\ZZ_2}$, and the Lebesgue differentiation theorem asserts that for almost every $x'\in\RR^1$, $|f(x')|\lesssim\|f\|_{*,\ZZ_2}$. The proof is completed.
\end{proof}

For $f\in L^{\infty}(\RR^1)$ of compact support,  by Proposition \ref{Riesz-BMO-b} and (\ref{truncated-Riesz-4}) we have
$$
(\tilde{\mathcal{R}}f)(x)=({\mathcal{R}}f)(x)-a
$$
where $a=c_{\kappa}\int_{\RR^1}f(t)K^1(0,t)\,d\omega_{\kappa}(t)$. Since such an $f$ is in $L^2_{\kappa}(\RR^1)$, by \cite[Theorem 6.3]{LL1} its $\kappa$-Hilbert transform $({\mathcal{R}}f)(x)$ can realized by
\begin{align}\label{Riesz-transform-1}
({\mathcal{R}}f)(x)=\lim_{\epsilon\rightarrow0+}c_{\kappa}\int_{|t-x|>\epsilon}f(t)K(x,t)\,d\omega_{\kappa}(t),
\end{align}
where, in the present case, $K(x,t)$ can be represented as
\begin{align*}
K(x,t)&=
\frac{m_{\kappa}|x+t|^{-2\kappa}}{x-t}\,
{}_2\!F_{1}\Big[{\kappa,\kappa
\atop
  2\kappa+1};\frac{4xt}{|x+t|^2}\Big], \ \ \ \ \ \hbox{if} \ xt>0,\,t\neq x,\\
&=\frac{m_{\kappa}(x-t)}{(|x|+|t|)^{2\kappa+2}}\,
{}_2\!F_{1}\Big[{\kappa,\kappa+1
\atop
  2\kappa+1};\frac{4|xt|}{(|x|+|t|)^2}\Big], \ \ \ \ \ \hbox{if} \ xt<0,\,t\neq-x.
\end{align*}
Here $m_\kappa=2^{\kappa+1/2}\Gamma(\kappa+1)/\sqrt{\pi}$, and ${}_2\!F_{1}[a,b;c;t]$ is the Gauss hypergeometric function.

Now we consider the function $\varphi_0(x)=({\mathcal{R}}\chi_{[-1,1]})(x)$. It is an odd function on $\RR^1$ and in $\B_{\kappa}(\RR^1)$ by Theorem \ref{Riesz-BMO-a}. When $x>1$, from (\ref{Riesz-transform-1}) we have
\begin{align*}
\varphi_0(x)=c_{\kappa}\int_{-1}^1K(x,t)\,d\omega_{\kappa}(t)
\ge c_{\kappa}\int_{1/2}^1K(x,t)\,d\omega_{\kappa}(t),
\end{align*}
and for $1/2<t<1$, from the first formula of $K(x,t)$ above we have
\begin{align*}
K(x,t)\ge \frac{m_{\kappa}}{(x+t)^{2\kappa}(x-t)}
\ge \frac{m_{\kappa}}{(x+1)^{2\kappa}(x-t)}.
\end{align*}
Thus for $x>1$,
\begin{align*}
\varphi_0(x)\ge \frac{c_{\kappa} m_{\kappa}}{(x+1)^{2\kappa}}\int_{1/2}^1\frac{t^{2\kappa}}{x-t}\,dt
\ge \frac{c_{\kappa} m_{\kappa}}{2^{2\kappa}(x+1)^{2\kappa}}\log\left(\frac{x-1/2}{x-1}\right),
\end{align*}
which shows that $\varphi_0$ is unbounded on $\RR^1$. Therefore $\varphi_0\notin\B_{\ZZ_2}(\RR^1)$ by Lemma \ref{bmo-G-d}.

\subsection{The atomic decomposition of $G-$invariant functions in $H^{1}_\kappa(\RR^d)$}

It is well known that the atomic characterization of the usual Hardy space $H^{1}(\RR^d)$ is implied by the duality of $H^1(\RR^d)$ and $\B(\RR^d)$. This idea belongs to C. Fefferman. Here we consider the atomic decomposition of $G-$invariant functions in $H^{1}_\kappa(\RR^d)$ according to this kind of thoughts. Such a decomposition of general functions in $H^{1}_\kappa(\RR^d)$ has been proved in \cite{DH1} by a different approach, see the description in Section 1.


Let us recall the concept of atoms associated with the space $(\RR^d,|\cdot|,\omega_{\kappa})$ of homogeneous type based upon Coifman and Weiss \cite{CW1}. For $1<q\le\infty$, a function $a\in L_{\kappa,{\rm loc}}(\RR^d)$ is said to be a $(1,q)$-atom if there is a ball $B$ such that

(i) ${\rm supp}\,a\subset B$;

(ii) $\|a\|_{L^q_{\kappa}}\le|B|_{\kappa}^{q^{-1}-1}$;

(iii) $\int_{\RR^d}a\,d\omega_{\kappa}=0$.

The atomic Hardy space $H^{(1,q)}_\kappa(\RR^d)$ consists of those functions $f$, for which there exist a sequence $\{\lambda_{\ell}\}$ of complex numbers and a sequence $\{a_{\ell}\}$ of $(1,q)$-atoms such that
$f=\sum_{\ell}\lambda_{\ell}a_{\ell}$ and $\sum_{\ell}|\lambda_{\ell}|<\infty$. The norm of $f\in H^{(1,q)}_\kappa(\RR^d)$ is defined by
$$
\|f\|_{H^{(1,q)}_\kappa}=\inf\left\{\sum\nolimits_{\ell}|\lambda_{\ell}|\right\},
$$
where the infimum is taken over all possible representations as above.

By \cite[pp. 592-3, Theorems A and B]{CW1} we have the following two theorems.

\begin{theorem}\label{atom-Hardy-a}
For $1<q<\infty$, $H^{(1,q)}_\kappa(\RR^d)=H^{(1,\infty)}_\kappa(\RR^d)$, and the norms $\|\cdot\|_{H^{(1,q)}_\kappa}$ and $\|\cdot\|_{H^{(1,\infty)}_\kappa}$ are equivalent.
\end{theorem}

It is useful to choose the space $H^{(1,2)}_\kappa(\RR^d)$ and the norm $\|\cdot\|_{H^{(1,2)}_\kappa}$ to work with because of the Fourier structure of the Dunkl setting.

\begin{theorem}\label{dual-atom-Hardy-BMO-a}
The space $\B_\kappa(\RR^d)$ is the dual of $H^{(1,2)}_\kappa(\RR^d)$ in the following sense:

{\rm(i)} for each $\varphi\in \B_\kappa(\RR^d)$, the linear functional
\begin{align*}
{\mathcal{L}}f=\lim_{n\rightarrow\infty}\sum_{\ell=1}^n\lambda_{\ell}\int_{\RR^d}a_{\ell}\varphi\,d\omega_\kappa,\qquad f=\sum_{\ell}\lambda_{\ell}a_{\ell}\in H^{(1,2)}_\kappa(\RR^d),
\end{align*}
is well-defined and continuous on $H^{(1,2)}_\kappa(\RR^d)$, whose norm is equivalent to $\|\varphi\|_{*,\kappa}$;

{\rm(ii)} each continuous linear functional $\mathcal{L}$ on $H^{(1,2)}_\kappa(\RR^d)$ has this form.
\end{theorem}

\begin{proposition}\label{Hardy-Hardy-a}
The space $H^{(1,2)}_\kappa(\RR^d)$ is contained in $H_{\kappa}^1(\RR^d)$, and for $f\in H^{(1,2)}_\kappa(\RR^d)$, $\|f\|_{H^1_\kappa}\lesssim\|f\|_{H^{(1,2)}_\kappa}$. Moreover $H^{(1,2)}_\kappa(\RR^d)$ is closed in $H_{\kappa}^1(\RR^d)$.
\end{proposition}

We present a proof below by means of the $\kappa$-Riesz transforms ${\mathcal{R}}_j$ ($1\le j\le d$).

\begin{proof}
We first show that, for every $(1,2)$-atom $a$, $\|{\mathcal{R}}_ja\|_{L_{\kappa}^1}\lesssim1$.
Indeed, for $a\in L^2_{\kappa}(\RR^d)$, ${\mathcal{R}}_ja$ is well defined, and
\begin{align}\label{atom-Riesz-1}
\left(\int_{{\mathcal{O}}(3B)}|{\mathcal{R}}_ja|\,d\omega_{\kappa}\right)^2
\le|{\mathcal{O}}(3B)|_{\kappa}\int_{{\mathcal{O}}(3B)}|{\mathcal{R}}_ja|^2\,d\omega_{\kappa}
\lesssim|B|_{\kappa}\int_{B}|a|^2\,d\omega_{\kappa}
\le1.
\end{align}
Furthermore, if $x^B$ denotes the center of the ball $B$,
\begin{align*}
\int_{[{\mathcal{O}}(3B)]^c}|{\mathcal{R}}_ja|\,d\omega_{\kappa}
=c_{\kappa}\int_{[{\mathcal{O}}(3B)]^c}\left|\int_{B}a(t)(K_j(x,t)-K_j(x,x^B))\,d\omega_{\kappa}(t)\right|\,d\omega_{\kappa}(x)\\
\lesssim \int_{B}|a(t)|\int_{[{\mathcal{O}}(3B)]^c}|(K_j(x,t)-K_j(x,x^B))|\,d\omega_{\kappa}(x)d\omega_{\kappa}(t).
\end{align*}
It is easy to see that for $t\in B$, $[{\mathcal{O}}(3B)]^c\subseteq\{x:\,d(x,t)>2|t-x^B|\}$. Using (\ref{Riesz-Hormander-2}) we have
\begin{align}\label{atom-Riesz-2}
\int_{[{\mathcal{O}}(3B)]^c}|{\mathcal{R}}_ja|\,d\omega_{\kappa}
\lesssim \int_{B}|a(t)|\,d\omega_{\kappa}(t)\lesssim1.
\end{align}
Combining (\ref{atom-Riesz-1}) and (\ref{atom-Riesz-2}) proves $\|{\mathcal{R}}_ja\|_{L_{\kappa}^1}\lesssim1$.

In general, if $f=\sum_{\ell}\lambda_{\ell}a_{\ell}\in H^{(1,2)}_\kappa(\RR^d)$ and $f_n=\sum_{\ell=1}^n\lambda_{\ell}a_{\ell}$, then $\|f_n-f\|_{L^1_{\kappa}}\rightarrow0$ as $n\rightarrow\infty$. Furthermore,
since for $m>n$,
$$
\|{\mathcal{R}}_jf_m-{\mathcal{R}}_jf_n\|_{L^1_{\kappa}}
\le\sum_{\ell=n+1}^m|\lambda_{\ell}|\|{\mathcal{R}}_ja_{\ell}\|_{L_{\kappa}^1}\lesssim\sum_{\ell=n+1}^m|\lambda_{\ell}|,
$$
it follows that there exists some $g\in L_{\kappa}^1(\RR^d)$ such that $\|{\mathcal{R}}_jf_n-g\|_{L^1_{\kappa}}\rightarrow0$ as $n\rightarrow\infty$. We also have $\|g\|_{L^1_{\kappa}}\lesssim\sum_{\ell}|\lambda_{\ell}|$.

By Proposition \ref{transform-a}(i), for $\xi\in\RR^d\setminus\{0\}$ we have $|[{\SF}_{\kappa}({\mathcal{R}}_jf_n)](\xi)-({\mathcal{R}}_jg)(\xi)|\le\|{\mathcal{R}}_jf_n-g\|_{L^1_{\kappa}}$
and $|[{\SF}_{\kappa}({\mathcal{R}}_jf_n)](\xi)+i(\xi_j/|\xi|)\left(\SF_{\kappa}f\right)(\xi)|\leq\|f_n-f\|_{L_{\kappa}^1}$, which shows $({\mathcal{R}}_jg)(\xi)=-i(\xi_j/|\xi|)\left(\SF_{\kappa}f\right)(\xi)$. According to Definition \ref{Riesz-2-e}, $g={\mathcal{R}}_jf\in L_{\kappa}^1(\RR^d)$, so that $f\in H_{\kappa}^1(\RR^d)$. Moreover
$\|f\|_{H^1_\kappa}\lesssim\sum_{\ell}|\lambda_{\ell}|$, and hence $\|f\|_{H^1_\kappa}\lesssim\|f\|_{H^{(1,2)}_\kappa}$. The last assertion of the proposition follows from a standard argument (see \cite[pp. 257-8]{GR1}).
\end{proof}

We denote the $G$-invariant subspaces of $\BC_\kappa(\RR^d)$, $\B_\kappa(\RR^d)$, $\B_G(\RR^d)$, $H^{1}_\kappa(\RR^d)$ and $H^{(1,2)}_\kappa(\RR^d)$ by ${\text{IBMC}}_\kappa(\RR^d)$, ${\text{IBMO}}_\kappa(\RR^d)$, ${\text{IBMO}}_G(\RR^d)$, $IH^{1}_\kappa(\RR^d)$ and $IH^{(1,2)}_\kappa(\RR^d)$ respectively.  Based on what we have known, the following conclusions are valid.

\begin{theorem}\label{H-BMO-invariant-a}
{\rm(i)} $IH^{(1,2)}_\kappa(\RR^d)$ is closed in $IH_{\kappa}^1(\RR^d)$;

{\rm(ii)} $IH^{1}_\kappa(\RR^d)^*={\text{IBMC}}_\kappa(\RR^d)$;

{\rm(iii)} $IH^{(1,2)}_\kappa(\RR^d)^*={\text{IBMO}}_\kappa(\RR^d)$;

{\rm(iv)} ${\text{IBMO}}_G(\RR^d)={\text{IBMC}}_\kappa(\RR^d)={\text{IBMO}}_\kappa(\RR^d)$.
\end{theorem}

\begin{theorem}\label{atom-characterization-a}
The two spaces $IH_{\kappa}^1(\RR^d)$ and $IH^{(1,2)}_\kappa(\RR^d)$ are identical, and the norms $\|\cdot\|_{H^1_\kappa}$ and $\|\cdot\|_{H^{(1,2)}_\kappa}$ are equivalent.
\end{theorem}

Indeed, since the dual of the quotient space $IH_{\kappa}^1(\RR^d)/IH^{(1,2)}_\kappa(\RR^d)$ is isometrically isomorphic to (see \cite[Theorem 7.2]{Dur1}, for example)
$$
IH^{(1,2)}_\kappa(\RR^d)^{\bot}=\left\{{\mathcal{L}}\in IH_{\kappa}^1(\RR^d)^*:\,\, {\mathcal{L}}f=0\,\, \hbox{for all}\,\,f\in IH^{(1,2)}_\kappa(\RR^d)\right\},
$$
the unique function $\varphi\in {\text{IBMC}}_\kappa(\RR^d)={\text{IBMO}}_\kappa(\RR^d)$ associated to ${\mathcal{L}}\in IH^{(1,2)}_\kappa(\RR^d)^{\bot}$ satisfies $\int_{\RR^d}f\varphi\,d\omega_\kappa=0$ for all $f\in IH^{(1,2)}_\kappa(\RR^d)$.
But by Theorem \ref{dual-atom-Hardy-BMO-a}, this implies that $\varphi=0$, so that ${\mathcal{L}}=0$. Therefore $(IH_{\kappa}^1(\RR^d)/IH^{(1,2)}_\kappa(\RR^d))^*=\{0\}$, and then $IH^{(1,2)}_\kappa(\RR^d)=IH_{\kappa}^1(\RR^d)$ with equivalent norms.

\subsection*{Acknowledgement}

The authors would like to express their gratitude to Professor Liang Song since some of remarks in the final section were discussed with him.

%
%
%
%
%
%
%
%
%
%
%
%
%

\end{document}